\documentclass{amsart}
\usepackage{amsmath}
\usepackage{graphicx}
\usepackage{amsfonts}
\usepackage{amssymb}
\usepackage{color}
\usepackage{mathrsfs}
\usepackage{epsfig}
\usepackage{url}
\usepackage{mathrsfs,a4wide}
\usepackage{calligra}

\theoremstyle{plain}
\newtheorem{theorem}{Theorem}[section]
\newtheorem{example}{Example}[section]
\newtheorem{lemma}[theorem]{Lemma}
\newtheorem{proposition}[theorem]{Proposition}
\newtheorem{corollary}[theorem]{Corollary}
\newtheorem{remark}[theorem]{Remark}

\theoremstyle{definition}
\theoremstyle{remark}
\numberwithin{equation}{section}

\newcommand{\T}{\mathcal{T}}
\newcommand{\Fzero}{\mathscr{F}_0}

\newcommand{\B}{\mathcal{B}}

\newcommand{\ho}{\A^{\mathrm{hom}}}

\newcommand{\sub}{\subset}

\newcommand{\ep}{\varepsilon}
\newcommand{\e}{\varepsilon}

\newcommand{\ffi}{\varphi}

\newcommand{\R}{\mathbb{R}}
\newcommand{\N}{\mathbb{N}}
\newcommand{\Z}{\mathbb{Z}}

\newcommand{\E}{\mathfrak{E}}

\newcommand{\AS}{\mathcal{AF}}
\newcommand{\A}{\mathbb{A}}
\newcommand{\di}{\textrm{dist}}

\newcommand{\ud}{\;\mathrm{d}}

\newcommand{\Om}{\Omega}
\newcommand{\supp}{\mathrm{supp}\,}
\newcommand{\M}{\mathcal{M}}

\newcommand{\AF}{\mathcal{AF}}
\newcommand{\Huno}{\mathcal H^1}

\newcommand{\weakly}{\rightharpoonup}           % weak convergence
\newcommand{\weakstar}{\stackrel{*}{\weakly}}   % weakstarconvergence
\newcommand{\fla}{\stackrel{\mathrm{flat}}{\rightarrow}}
\newcommand{\flt}{\mathrm{flat}}

\newcommand{\rad}{\mathcal Rad}
\newcommand{\sr}{\mathcal{R}}

\newcommand{\flap}{{\mathrm{flat}}}

\newcommand{\Ccal}{{\mathcal{C}}}
\newcommand{\cu}{\mathrm{curl}\;}

\newcommand{\Fs}{\mathscr{F}}

\newcommand{\precrw}{F_{\delta_\varepsilon}}
\newcommand{\pregl}{G\!L_{\ep}}
\newcommand{\precr}{\mathscr{F}_{\ep}}
\newcommand{\FF}{F}

\newcommand{\con}{\mathcal{G\!L}_{\ep}}
\def\red#1{{#1}}

\def\Xint#1{\mathchoice
   {\XXint\displaystyle\textstyle{#1}}%
   {\XXint\textstyle\scriptstyle{#1}}%
   {\XXint\scriptstyle\scriptscriptstyle{#1}}%
   {\XXint\scriptscriptstyle\scriptscriptstyle{#1}}%
   \!\int}
\def\XXint#1#2#3{{\setbox0=\hbox{$#1{#2#3}{\int}$}
     \vcenter{\hbox{$#2#3$}}\kern-.5\wd0}}

\def\dashint{\Xint-}
\usepackage{latexsym}
\newcommand{\newatop}{\genfrac{}{}{0pt}{1}}

\makeatletter
\def\@splitop#1#2\@nil{$\mathscr{#1}\!\!$\calligra#2\,\,}
\newcommand*\DeclareCursiveOperator[2]{%

 \newcommand#1{\mathop{\mbox{\@splitop#2\@nil}}\nolimits}}
\makeatother
\DeclareCursiveOperator{\Anew}{A}
\DeclareCursiveOperator{\Bnew}{B}
\DeclareCursiveOperator{\Cnew}{C}
\DeclareCursiveOperator{\Dnew}{D}
\DeclareCursiveOperator{\Enew}{E}
\DeclareCursiveOperator{\Qnew}{Q}
\title[Topological Singularities in Periodic Media]{Topological singularities in periodic media:  Ginzburg-Landau and core-radius approaches}

\author[R. Alicandro]
{R. Alicandro}
\address[Roberto Alicandro]{DIEI, Universit\`a di Cassino e del Lazio meridionale, via Di Biasio 43, 03043 Cassino (FR), Italy}
\email[R. Alicandro]{alicandr@unicas.it}
\author[A. Braides]{A. Braides}
\address[Andrea Braides]{Dipartimento di Matematica, Universit\`a di Roma `Tor Vergata', via della Ricerca Scientifica, 00133 Rome, Italy}
\email[A. Braides]{braides@mat.uniroma2.it}
\author[M. Cicalese]{M. Cicalese}
\address[Marco Cicalese]{Zentrum Mathematik - M7, Technische Universit\"at M\"unchen, Boltzmannstrasse 3, 85747 Garching, Germany}
\email[M. Cicalese]{cicalese@ma.tum.de}
\author[L. De Luca]
{L. De Luca}
\address[Lucia De Luca]{Istituto per le Applicazioni del Calcolo ``M. Picone'', IAC-CNR, via dei Taurini 19, 00185 Rome, Italy}
\email[L. De Luca]{lucia.deluca@cnr.it}
\author[A. Piatnitski]
{A. Piatnitski}
\address[Andrey Piatnitski]{The Arctic University of Norway, campus Narvik,
P.O. Box 385, 8505,
Narvik, Norway and Institute for Information Transmission Problems of RAS, Bolshoy Karetny per., 19, Moscow, Russia, 127051}
\email[A. Piatnitski]{apiatni@iitp.ru}

\usepackage{hyperref}
\begin{document}
%%%%%%%%%%%%%%%%%%%%%%%%%%%
%%%%%%%%%%%%%%%%%%%%%%%%%%%
%%%%%%%%%%%%%%%%%%%%%%%%%%%
\begin{abstract}
We describe the emergence of topological singularities in periodic media within the Ginzburg-Landau model and the core-radius approach.
The energy functionals of both models are denoted by $E_{\ep,\delta}$, where $\ep$ represent the coherence length (in the Ginzburg-Landau model) or the core-radius size (in the core-radius approach) and  $\delta$ denotes the periodicity scale.
We carry out the $\Gamma$-convergence analysis of $E_{\ep,\delta}$  as $\e\to 0$ and $\delta=\delta_\e\to 0$ in the $|\log\e|$ scaling regime, showing that the $\Gamma$-limit consists in the energy cost of finitely many vortex-like point singularities of integer degree. After introducing the scale parameter (upon extraction of subsequences)
$$
\lambda=\min\Bigl\{1,\lim_{\e\to0} {|\log \delta_\e|\over|\log\e|}\Bigr\},
$$
we show that in a sense we always have a separation-of-scale effect: at scales less than $\e^\lambda$ we first have a concentration process around some vortices whose location is subsequently optimized, while for scales larger than $\e^\lambda$ the concentration process takes place ``after'' homogenization.

%We carry out the $\Gamma$-convergence analysis of $E_{\ep,\delta}$  as $\e\to 0$ in the $|\log\e|$ scaling regime for both the cases $\delta=\delta_\ep\lesssim\ep$ and $\delta=\delta_\ep\gg\ep$\,.

%In both cases the $\Gamma$-limit consists in the energy cost of finitely many vortex-like point singularities of integer degrees.
%Denoting by $F_{\delta}$ the gradient term in $E_{\ep,\delta}$ and by  $F_{\hom}$ its $\Gamma$-limit as $\delta\to 0$\,, for $\delta\lesssim\ep$ the cost above  agrees with that of a Ginzburg-Landau/core-radius approach functional in a homogeneous medium with leading energy equal to $F_{\hom}$\,. This result shows that, to some extent,
% for  $\delta\lesssim\e$\,, the concentration effect takes place ``after'' the homogenization process.
%Finally, we show that for $\delta\gg\ep$ the homogenization effect appears at a suitable intermediate scale between $\ep$ and $\delta$.

\vskip5pt
\noindent
\textsc{AMS subject classifications:}  
49J45   %Existence theories in calculus of variations and optimal control, optimization; Methods involving semicontinuity and convergence; relaxation
74Q05 %Homogenization, determination of effective properties in solid mechanics; Homogenization in equilibrium problems of solid mechanics
35Q56 %Partial differential equations of mathematical physics and other areas of application; Ginzburg-Landau equations  
49J10   %Existence theories in calculus of variations and optimal control,  Existence theories for free problems in two or more independent variables
74B15   %Elastic materials; Equations linearized about a deformed state (small deformations superposed on large)
\end{abstract}
%%%%%%%%%%%%%%%%%%%%%%%%%%%
%%%%%%%%%%%%%%%%%%%%%%%%%%%
%%%%%%%%%%%%%%%%%%%%%%%%%%%
\maketitle
%%%%%%%%%%%%%%%%%%%%%%%%%%%
%%%%%%%%%%%%%%%%%%%%%%%%%%%
%%%%%%%%%%%%%%%%%%%%%%%%%%%
\tableofcontents
%%%%%%%%%%%%%%%%%%%%%%%%%%%
%%%%%%%%%%%%%%%%%%%%%%%%%%%
%%%%%%%%%%%%%%%%%%%%%%%%%%%
\section*{Introduction}
Phase transitions mediated by the formation of topological defects characterize several physical
phenomena such as superfluidity, superconductivity and plasticity (see \cite{Lo1, Lo2, Mermin, HL, Kle_Lav, HB}). The study of such topological defects has become an extremely active research field in mathematics after the progresses achieved in the analysis of the Ginzburg Landau (GL) energy functional in the last decades (see e.g.~\cite{BBH, SS2}). In \cite{ACP} it has been proved that the GL functional, originally introduced to model the phenomenology of phase transitions in Type-II superconductors through the formation of vortex singularities of a complex order parameter, provides a good variational description for the emergence of vortices in $XY$ spin systems and of screw dislocations in crystal plasticity (see also \cite{AC, P, ADGP, DL, BCKO}).  The results obtained in \cite{ACP} suggest to exploit the GL theory for a phenomenological alternative description of several material-dependent variational models,
opening the way to a number of new mathematical problems involving the analysis of this functional. For instance, in the modeling of materials, one needs to suitably modify it to include the usual
kinematic constraints and material constants which are specific of crystal structures. As a first step in this direction, here we study a variant of the GL energy functional to include heterogeneities of the medium.

\par

Before describing the case of heterogeneous media, we briefly recall the analysis in the homogeneous case. Let $\Omega\subset\R^{2}$ be an open bounded set and let $\e$ denote the {\em coherence length} of the GL energy (proportional to the length scale of the core of a screw dislocation in a plastic crystal or to the lattice spacing in a $XY$ spin system). Let $a>0$ and let $\con: H^1(\Omega;\R^2)\to \R$ be  the Ginzburg-Laundau functional defined as
\begin{equation}\label{intro:GL-con}
\con(v):= a\int_\Om |\nabla v(x)|^{2} \ud x+\frac{1}{\ep^2}\int_\Om (1-|v(x)|^{2})^{2} \ud x\,.
\end{equation}
The asymptotic behavior of $\con$ as $\e\to 0$ has been studied in order to give an energetic description of the onset of vortices (see for instance \cite{BBH,SS2}). A prototypical vortex of degree $z\in\Z\setminus\{0\}$ at a point $x_{0}\in\Omega$ can be thought of as the point singularity of a vectorial order parameter $\bar v_\ep:\Omega\to\R^{2}$ which, outside the ball of radius $\e$ centered at $x_{0}$, winds around the center as $(\frac{x-x_{0}}{|x-x_{0}|})^{z}$. The energy of $\bar v_{\e}$ diverges at order $|\log\e|$ as $\ep\to 0$. As a consequence, to detect the effective energy cost of finitely many vortex singularities, one needs to study the $\con$ energy at a {\em  logarithmic scaling}; that is, to consider the asymptotic behavior of functionals $\frac{\con(v)}{|\log\e|}$. It has been proved in \cite{JS,ABO} that a sequence $\{v_{\e}\}$, along which these energy functionals are equi-bounded, has Jacobians $Jv_\e$ that, up to a subsequence, converge in the flat sense (see Section \ref{sec:npr}) to an atomic measure $\mu=\sum_{i=1}^{n}z_{i}\delta_{x_{i}}$ whose support represents the position of the limiting vortices. The $\Gamma$-limit of $\frac{\con}{|\log\e|}$ with respect to this convergence at $\mu$ is then given by $2\pi a\sum_{i=1}^n |z_i|$ (supposing $x_i\neq x_j$ if $i\neq j$). This value can be rewritten as $2\pi a|\mu|(\Omega)$ and thought of as a functional depending on the total variation $|\mu|(\Omega)$ of $\mu$ in $\Omega$.

Now, if more in general $\Omega$ is regarded as a reference configuration of a heterogeneous material, described by periodic heterogeneities at a length scale $\delta_{\e}$, we may consider the energies $\pregl: H^1(\Omega;\R^2)\to \R$ defined as
\begin{equation}\label{intro:GL0}
\pregl(v):=\int_\Om a\Bigl(\frac{x}{\delta_{\e}}\Bigr)|\nabla v(x)|^{2} \ud x+\frac{1}{\ep^2}\int_\Om (1-|v(x)|^{2})^{2} \ud x\,,
\end{equation}
where $a:\R^{2}\to [\alpha,\beta]$ ($0<\alpha<\beta$) is a $(0,1)^{2}$-periodic function describing the material properties of the media. Note that the energy $\pregl$ is controlled from (above and) below by a multiple of the GL energy $\con$ above. Therefore, setting
\begin{eqnarray*}
X(\Omega):=\Bigl\{\mu\in\M(\Omega)\,:\,\mu=\sum_{i=1}^{n}z_{i}\delta_{x_{i}}\,,\, n\in\N\,,\,z_{i}\in\Z\setminus\{0\}\,,\,x_{i}\in\Omega\Bigr\}\,,
\end{eqnarray*}
the following compactness result holds true.
\begin{theorem}\label{compgl}
\rm{Let $\left\{v_\ep\right\}_\ep\subset H^1(\Omega;\R^2)$ be such that $\pregl(v_\ep)\le C|\log\ep|$. Then, there exists $\mu\in X(\Omega)$ such that, up to subsequences, $Jv_\ep\fla\pi\mu$\,.}
\end{theorem}
%Jacobians of functions with equi-bounded $\frac{\pregl}{|\log\e|}$ energy still concentrate on a finite set of points in $\Omega$ and the energy cost of each singularity remains finite.

Assuming $\delta_{\e}\to 0$ as $\e\to 0$ we expect the effective limiting energy at the vortex scaling to be a homogeneous energy combining both homogenization and concentration effects. As these effects depend on the mutual rate of convergence of the vanishing parameters $\e$ and $\delta_{\e}$, different regimes are possible. Heuristically, at some extreme regimes we will have ``separation of scales''. Namely, if $\e$ tends to $0$ sufficiently fast with respect to $\delta=\delta_\e$ then we expect that $\delta$ can be thought of as an independent variable, the dependence on which separately dealt with after letting $\e\to 0$ with fixed $\delta$. In this case, the limit as $\e\to0$ with $\delta$ fixed gives an energy of the form $2\pi\sum_{i=1}^n |z_i|a\bigl({x_i\over\delta}\bigr)$, and the optimization of the location of vortices at minimum points for $a$ (we may assume here that $a$ be continuous), which tend to be dense as $\delta\to0$, finally provides a limit of the form   $$2\pi\min a\sum_{i=1}^n |z_i|. $$
Note that in order that this argument may work, the energy of a recovery sequence should be concentrated on a $\mathrm{o}(\delta)$-neighborhood of a minimum point of $a$. This gives a condition
\begin{equation}\label{intro:co1}
{|\log \delta|\over|\log\e|} \ll 1
\end{equation}
by testing with functions winding as $\frac{x-x_i}{|x-x_i|}$ around a vortex $x_i$.

Conversely, if $\delta=\delta_\e$ tends to $0$ sufficiently fast with respect to $\e$, we expect that the variable $\e$ be considered as fixed and a homogenization process may be first performed with $\delta\to0$. In this case, moreover, since the potential term in \eqref{intro:GL0} forces $v_\ep$ to have modulus equal to one as $\ep\to 0$, (neglecting for a moment the effect of singularities) we may regard the homogenization process to be restricted to the first part of the energy in \eqref{intro:GL0}, which can be written as
\begin{equation*}%\label{Gdelta}
G_{\delta}(u):=\int_{\Omega}a\Big(\frac {x} {\delta}\Big)|\nabla u|^2\ud x\,,
\end{equation*}
where $u$ is the lifting of $v$, i.e., $v=e^{\imath u}$\,.
The homogenization of functionals of this form has been extensively studied in terms of $\Gamma$-convergence (see \cite{BrDef}) and it has been shown that $G_\delta\overset{\Gamma}{\longrightarrow}G_0$ as $\delta\to 0$\,, where
$$
G_0(u):=\int_{\Omega}\langle \ho\nabla u,\nabla u\rangle \ud x,
$$
and $\ho$ is the  two-by-two symmetric matrix defined by
\begin{equation}\label{defahom}
\langle\ho \xi,\xi\rangle:=\inf\left\{\int_{(0,1)^{2}}a(y)|\xi+\nabla\ffi(y)|^2\ud y\,:\,\ffi\in W^{1,\infty}_{\mathrm{per}}((0,1)^{2})\right\}\,.
\end{equation}
At this point, the subsequent analysis involves the study of the $\Gamma$-limit as $\e\to0$ of a homogeneous but anisotropic energy functional related to $G_0$ at scale $|\log\e|$.
%, whose computation can be achieved by considering the contribution of $G_0(u_\e)\over|\log\e|$ as $\e\to 0$ and $Ju_\e\fla\pi \sum_{i=1}^{n}z_{i}\delta_{x_{i}}$. 
The validity of this separation of scales  can be formalized by using a coarea formula-type argument, which shows that the $\Gamma$-convergence of $\pregl$ can be obtained working within another well-known framework in the analysis of topological singularities; i.e., the so-called  {\it core-radius approach}. That approach consists in computing the gradient term in the energy outside small regions -- the {\it cores} -- around the singularities, and allows to directly work with $\mathcal S^1$-valued order parameters (see e.g.~\cite{BBH,AP}). In this framework, we may 
%compute limits of energies of the form $G_0(u_\e)\over|\log\e|$ as $Ju_\e\fla\pi \sum_{i=1}^{n}z_{i}\delta_{x_{I}}$, and
describe the energy around a vortex of degree $z$ by an asymptotic formula of the type
\begin{equation}\label{intro:psifor}
\psi(z)=\lim_{\frac{R}{r}\to+\infty}{1\over \log\frac{R}{r}}\min\Bigl\{\int_{B_R\setminus B_r}\langle \ho\nabla u,\nabla u\rangle \ud x: u\in H^1(B_R\setminus B_r),\ {\rm deg}\,(e^{\imath  u};B_r)=z\Bigr\},
\end{equation}
from which the $\Gamma$-limit is obtained by locally optimizing the degree (possibly approximating a vortex by more vortices). A computation eventually allows to conclude that the limit energy has the form
$$2\pi\sqrt {{\rm det}\,\ho}\sum_{i=1}^n |z_i|. $$
In order that this argument work, minimum problems in \eqref{intro:psifor} should be seen as limits of minimum problems
\begin{equation}\label{intro:psifor_e}
\min\Bigl\{\int_{B_R\setminus B_r}a\Bigl({x\over\eta}\Bigr)|\nabla u|^2 \ud x: u\in H^1(B_R\setminus B_r),\ {\rm deg}\,(e^{\imath  u};B_r)=z\Bigr\},
\end{equation}
for some choice of $r$ and $R$ with $R/r\to+\infty$, and $\eta\to0$. This can be done by a scaling argument if $\delta\ll\e$.  An approximation argument, more in general, allows to extend this result to
\begin{equation}\label{intro:co2}
|\log \delta|\ge|\log\e|
\end{equation}
using the scaling properties of the energies.

It is interesting to note that there is a scale gap between the two separation-of-scale regimes given by \eqref{intro:co1} and  \eqref{intro:co2}; i.e., when
\begin{equation}\label{intro:co3}
\lim_{\e\to0} {|\log \delta|\over|\log\e|}=\lambda\in (0,1)
\end{equation}
(the existence of the limit is not restrictive up to extraction of subsequences). In this case the behavior of the $\Gamma$-limit is a convex combination of the extreme ones. Recovery sequences are constructed with vortices concentrating close to minimum points for $a$, while they optimize oscillations at scales between $\e$ and $1$ so as to obtain a homogenized overall behavior at those scales. The final form of the $\Gamma$-limit is then
  $$2\pi\Big((1-\lambda)\min a+\lambda\sqrt {{\rm det}\,\ho}\Big)\sum_{i=1}^n |z_i|, $$
which comprises also the extreme cases, upon setting
\begin{equation}\label{intro:co4}
\lambda=\lim_{\e\to0} {|\log \delta|\over|\log\e|}\wedge 1.
\end{equation}

Note that, since in the logarithmic regime, the $GL_{\e}$ energies concentrate at any scale between $\ep$ and $1$, their behavior is very different from that of the corresponding scalar version, the inhomogeneous Cahn-Hilliard functionals given (after scaling) by
$$
{\rm C H}_\e(u):=\e\int_\Om a\Big(\frac{x}{\delta_{\e}}\Big)|\nabla u(x)|^{2} \ud x+\frac{1}{\ep}\int_\Om (1-|u(x)|^{2})^{2} \ud x\,  \qquad u\in H^1(\Omega),
$$
which concentrate at scale $\e$ producing sharp-interface models. In that case separation of scale occurs for $\e\ll\delta_\e$ and $\delta_\e\ll\e$, while in the {\em critical regime} $\delta_\e\sim\e$ the effective surface tension is described by an optimal-profile problem depending on $K:=\lim_{\e\to0}\delta_\e/\e$ (see \cite{ABCP}). In a sense, in the GL case we do not have a critical behavior and we always have separation of scales. The parameter $\lambda$ above can be seen as describing a threshold scale above and below which the two types of separation of scales take place.

\bigskip

Although suggested by the heuristics, the computation of the $\Gamma$-limits described above is highly non-trivial and needs several new ideas in order to combine techniques from GL and homogenization theories. We briefly outline some of the most relevant technical issues, and state the main results of the paper formalizing the heuristic description given above, subdividing the analysis in the cases $\delta_\ep\lesssim\ep$ and $\delta_\ep\gg\ep$

The following result is proven in Section \ref{sec:mainthmgl}.
\begin{theorem}\label{intro: mainthmgl}{If $\limsup\limits_{\ep\to 0}\frac{\delta_\ep}{\ep}< +\infty$,  then the following $\Gamma$-convergence result holds true.
\begin{itemize}
\item[(i)] (\,$\Gamma$-liminf inequality)
Let $\{v_\ep\}_\ep\subset H^1(\Omega;\R^2)$ be such that $J v_\ep\fla\pi\mu$ for some $\mu\in X(\Omega)$\,. Then
\begin{equation*}%\label{liminfgl}
\liminf_{\ep\to 0}\frac{\pregl(v_\ep)}{|\log\ep|}\ge 2\pi\sqrt{\det\ho}|\mu|(\Omega).
\end{equation*}
\item[(ii)](\,$\Gamma$-limsup inequality) For every $\mu\in X(\Omega)$, there exists a sequence $\{v_\ep\}_\ep\subset H^1(\Omega;\R^2)$ such that  $Jv_\ep\fla\pi\mu$ and
\begin{equation*}
\limsup_{\ep\to 0}\frac{\pregl(v_\ep)}{|\log\ep|}\le2\pi\sqrt{\det\ho}|\mu|(\Omega)\,.
\end{equation*}
\end{itemize}
}
\end{theorem}

Within the core-radius approach,
we carry out the $\Gamma$-convergence analysis for more general quadratic functionals than the one in the leading term of \eqref{intro:GL0}.
Specifically,
let $f:\R^2\times \R^{2\times 2}\to [0,+\infty)$ be a Carath\'eodory function satisfying the following assumptions:
\begin{align}\label{Pprop}\tag{{\bf P}}
&\textrm{$f(\cdot,M)$ is $(0,1)^2$-periodic for every $M\in\R^{2\times 2}$;}\\ \label{Gprop}\tag{{\bf G}}
&\textrm{there exist two constants $\alpha,\beta$ with $0<\alpha\le \beta$ such that }\\ \nonumber
&
 \alpha|M|^2\le f(y,M)\le \beta|M|^2\,,\quad\textrm{for every }M\in\R^{2\times 2} \textrm{ and for almost every }y\in\R^2\,;\\ \label{Hprop}\tag{{\bf H}}
&\textrm{$f(y,\cdot)$ is homogeneous of degree $2$ for almost every $y\in\R^{2}$}.
\end{align}
We describe the asymptotic behavior in the logarithmic regime of the functionals
\begin{equation*}%\label{nuovo}
\mathcal{F}_{\ep,\delta_\ep}(\mu,w):=\int_{\Omega_\ep(\mu)}f\Big(\frac{x}{\delta_\ep},\nabla w(x)\Big)\ud x\,,
\end{equation*}
defined for  $\mu=\sum_{i=1}^{n}z_i\delta_{x_i}\in X(\Omega)$ and $w\in \AS_\ep(\mu)$\,, with
\begin{equation}\label{admst}
\AS_\ep(\mu):=\{w\in H^1(\Omega_\ep(\mu);\mathcal{S}^1)\,:\,\deg(w,\partial B_\ep(x_{i}))=z_{i}\quad\textrm{for every }i=1,\ldots,n\}\,,
\end{equation}
where $\Omega_\ep(\mu):=\Omega\setminus \bigcup_{i=1}^{n} \overline B_\ep(x_{i})$\,.
Like the Jacobians in the GL theory, $\mu$ is the relevant parameter to keep track of energy concentration. Therefore, we let the functional depend only on $\mu$ by setting
\begin{equation*}%\label{intro:cra_en0}
\precr(\mu):=\inf_{w\in \AS_\ep(\mu)}\mathcal F_{\ep,\delta_\ep}(\mu;w)\,.
\end{equation*}

We prove that, for $\delta_\ep\lesssim \ep$, the functional $\precr$ asymptotically behaves as $\pregl$, namely, the homogenization process takes place ``before'' the concentration effect. This implies that the effective cost of a singularity depends on the homogenized energy associated to the functionals $\FF_\delta(\cdot;E)$ defined for any open set $E$ as
\begin{equation}\label{intro:cra_en_0}
\FF_\delta(w;E):=\int_{E}f\Big(\frac{x}{\delta},\nabla w(x)\Big)\ud x\,,\qquad w\in H^1(E;\mathcal S^1)\,.
\end{equation}
Notice indeed that %, with this definition, at end we can write
\begin{equation}\label{intro:cra_en}
\precr(\mu)=\inf_{w\in \AS_\ep(\mu)}\precrw(w;\Omega_\ep(\mu))\,.
\end{equation}
In \cite{BM} it has been proved that, as $\delta\to 0$, the functionals $\FF_\delta$ $\Gamma$-converge to the homogenized functional $F_{\hom}(\cdot;E):H^{1}(E;\mathcal S^1)\to[0,+\infty)$ defined as
\begin{equation*}%\label{intro:cra_en_hom}
F_{\hom}(w;E):=\int_{E}T\!f_{\hom} (w(x),\nabla w(x) )\ud x\,.
\end{equation*}
In the formula above the energy density $T\!f_{\hom}$ is the tangential homogenization of the function $f$ (see formula \eqref{genhom1}).

In order to make the core-radius functionals non-trivial, we define $\precr$ only on the set
\begin{equation}\label{measep}
X_{\ep}(\Omega):=\Bigl\{\mu=\sum_{i=1}^{n}z_{i}\delta_{x_{i}}\in X(\Omega)\,:\,\min_{i\neq j}\Bigl\{\frac  1 2{|x_{i}-x_{j}|}\,,\di(x_{i},\partial\Omega)\Bigr\}\ge 2\ep\Bigr\}\,.
\end{equation}
In view of  assumption \eqref{Gprop}  and of the classical results on the core-radius approach functional (see for instance \cite[Theorem 3.2]{AP}), we have that the functionals $\precr$ satisfy compactness properties analogous to the ones established in Theorem \ref{compgl}.
\begin{theorem}{Let $\left\{\mu_\ep\right\}_\ep\subset X(\Omega)$ be such that $\mu_\ep\in X_\ep(\Omega)$ for every $\ep>0$
and that $\precr(\mu_\ep)\le C|\log\ep|$. Then, there exists $\mu\in X(\Omega)$ such that, up to subsequences, $\mu_\ep\fla\mu$.}
\end{theorem}
The following theorem on the asymptotic limit of the functionals $\precr$ is proved in Section \ref{sec:mainthmcra}.
%%%%%%%%%%%%%%%%%%%%%%%%%%%
%%%%%%%%%%%%%%%%%%%%%%%%%%%
%%%%%%%%%%%%%%%%%%%%%%%%%%%
\begin{theorem}\label{mainthmcra}{ If $\limsup_{\ep\to 0}\frac{\delta_\ep}{\ep}< +\infty$, then the following $\Gamma$-convergence result holds true.
%then the functionals $\precr$ $\Gamma(\flt)$-converges to $\Fzero$ as $\ep\to 0$\,; i.e.,
\begin{itemize}
%\item[] $\precr$ $\Gamma(\flt)$-converges to $\Fzero$ as $\ep\to 0$; i.e.,
\vskip4pt
\item[(i)] (\,$\Gamma$-liminf inequality)
For any family $\left\{\mu_\ep\right\}_\ep\subset X(\Omega)$ such that $\mu_\ep\in X_\ep(\Omega)$ for every $\ep>0$ and
$\mu_\ep\fla\mu$ with $\mu\in X(\Omega)$  we have
\begin{equation*}%\label{liminfcra}
\liminf_{\ep\to 0}\frac{\precr(\mu_\ep)}{|\log\ep|}\ge\Fzero(\mu).
\end{equation*}
\item[(ii)](\,$\Gamma$-limsup inequality) For every $\mu\in X(\Omega)$, there exists a sequence
$\left\{\mu_\ep\right\}_\ep\subset X(\Omega)$ with $\mu_\ep\in X_\ep(\Omega)$ for every $\ep>0$
such that  $\mu_\ep\fla\mu$ and
\begin{equation*}%\label{lisufor}
\limsup_{\ep\to 0}\frac{\precr(\mu_\ep)}{|\log\ep|}\le\Fzero(\mu)\,.
\end{equation*}
\end{itemize}
}
\end{theorem}
%%%%%%%%%%%%%%%%%%%%%%%%%%%
%%%%%%%%%%%%%%%%%%%%%%%%%%%
%%%%%%%%%%%%%%%%%%%%%%%%%%%
In the statement above $\Fzero:X(\Omega)\to [0,+\infty)$ is the functional defined as
\begin{equation}\label{gali}
\mathscr{F}_0(\mu):=\sum_{i=1}^{n}\Psi(z_i;T\!f_{\hom})\qquad \textrm{for every }\mu=\sum_{i=1}^n z_i\delta_{x_i}\in X(\Omega)\,,
\end{equation}
where $\Psi(z\,;T\!f_{\hom})$, introduced in \eqref{defPsi}, is the asymptotic energy cost of a singularity of degree $z$ in a homogeneous medium whose energy is $F_{\hom}$.
The function $\Psi(z\,;T\!f_{\hom})$  is obtained via an asymptotic cell-problem formula and a relaxation procedure.
Loosely speaking, we first introduce the minimal $F_{\hom}$ energy in an annulus around a singularity with degree $z$ and we show that such a quantity admits a finite limit, denoted by $\psi(z;T\!f_{\hom})$, when the quotient of the radii goes to $+\infty$\,. Then $\Psi(\cdot;T\!f_{\hom})$ is obtained as the relaxation of the function $\psi(\cdot\,;T\!f_{\hom})$ on $\Z$\, (see formula \eqref{defPsi}), accounting for the fact that a singularity of degree $z$ can be approximated by a family of singularities of degree $z_j$ with $\sum_{j}z_j=z$\,. In the simple case that $f(x,M)=a(x)|M|^{2}$ with $a\in[\alpha,\beta]$, we actually prove that $\Psi(z\,;T\!f_{\hom})=2\pi\sqrt{\det\ho}|z|$ (see Proposition \ref{prop:psisimplf}).

Using a scaling argument, in Proposition \ref{homdeg} we show that $\psi(\cdot;T\!f_{\hom})$ is also the asymptotically minimal $\precr$ energy on ``fat'' annuli around a vortex of degree $z$. Here, ``fat'' stays for thick enough to contain infinitely many $\delta_{\e}$-periodicity cells. Such a property allows to apply the
 homogenization result in \cite{BM} that we show to hold even if the functionals are subject to a degree constraint %also in which the position and the degree of the vortex are prescribed
 (see Theorem \ref{homogen_thm}  \& Corollary \ref{homdeg0}).

A further technical aspect of our analysis is the use, in the proof of the lower bound, of a refinement of the celebrated {\it ball construction} introduced in \cite{S, J}. This method allows to find a one-parameter family $\B_\ep(t)$ of growing and merging balls, that in turn identify a family of annuli where the energy concentrates. In our case, using a strategy similar to \cite{DGP}, we stop the process at an appropriate ``time'' $t_{\e}$ at which the constructed family of annuli is ``fat'' enough to apply the analysis described above and to obtain the desired lower bound.

\par

Using the same strategy exploited for $\delta_\ep\lesssim\ep$ we are able to study the asymptotic behavior of the core-radius approach and of the Ginzburg-Landau functionals also for $\delta_\ep\gg\ep$ (see also \cite{DOn} for an example in this case).
More precisely, we assume that $\delta_\ep\to 0$ as $\ep\to 0$ and that
\begin{equation}\label{limlog}
\lambda:=\lim_{\ep\to 0}\frac{|\log\delta_\ep|}{|\log\ep|}\in [0,1)\,,
\end{equation}
which implies that $\lim_{\ep\to 0}\frac{\delta_\ep}{\ep}=+\infty$\,.
Furthermore, we assume that
\begin{equation}\label{aform}
\FF_{\delta}(w;E):=\int_{E}a\Big(\frac{x}{\delta}\Big)|\nabla w|^2\ud x\,,\qquad w\in H^1(E;\mathcal S^1)\,,
\end{equation}
%for any open set $E$ and for every $w\in H^1(E;\mathcal S^1)$\,.
for some measurable $(0,1)^2$-periodic function $a$ with $a(y)\in [\alpha,\beta]\subset(0,+\infty)$ for a.e. $y\in\R^2$.
%%%%%%%%%%%%%%%%%%%%%%%%%%%%%%%%%
%%%%%%%%%%%%%%%%%%%%%%%%%%%%%%%%%
%%%%%%%%%%%%%%%%%%%%%%%%%%%%%%%%%
%For every $\lambda\in[0,1)$ we define the functional $\mathscr{F}_{0,\lambda}:X(\Omega)\to (0,+\infty)$ as
%\begin{equation}\label{galidelta>ep}
%\mathscr{F}_{0,\lambda}(\mu):=2\pi\bigg(\Big(1-\lambda)\mathrm{ess}\inf a+ \lambda\sqrt{\det\ho}\bigg)|\mu|(\Omega)
%\end{equation}
%%%%%%%%%%%%%%%%%%%%%%%%%%%%%%%%%
%%%%%%%%%%%%%%%%%%%%%%%%%%%%%%%%%
%%%%%%%%%%%%%%%%%%%%%%%%%%%%%%%%%
The main results in this scaling regime are the following two theorems proven in Section \ref{sec:delta>ep}.
\begin{theorem}\label{cratbw}
%Let $\precrw,\,\precr$ be defined in \eqref{intro:cra_en_0}, \eqref{intro:cra_en}, respectively, with $f$ of the form \eqref{simplf}, where $a$ is a measurable $(0,1)^2$-periodic function satisfying $a(x)\in [\alpha,\beta]\red{\subset (0,+\infty)}$ for a.e. $x\in\R^2$\,. Let moreover $\ho$ be the matrix defined in \eqref{defahom}.
If
\eqref{limlog} is satisfied,
%\begin{equation}\label{limlog}
%\lambda:=\limsup_{\ep\to 0}\frac{|\log\delta_\ep|}{|\log\ep|}\in[0,1)\,,
%\end{equation}
then the following statements hold true.
\begin{itemize}
%\item[] $\precr$ $\Gamma(\flt)$-converges to $\Fzero$ as $\ep\to 0$; i.e.,
\vskip4pt
\item[(i)] (\,$\Gamma$-liminf inequality)
For any family $\left\{\mu_\ep\right\}_\ep\subset X(\Omega)$ such that $\mu_\ep\in X_\ep(\Omega)$ for every $\ep>0$ and
$\mu_\ep\fla\mu$ with $\mu\in X(\Omega)$  we have
\begin{equation*}%\label{liminfcra}
\liminf_{\ep\to 0}\frac{\precr(\mu_\ep)}{|\log\ep|}\ge2\pi\bigg(\Big(1-\lambda)\mathrm{ess}\inf a+ \lambda\sqrt{\det\ho}\bigg)|\mu|(\Omega)\,.
\end{equation*}
\item[(ii)](\,$\Gamma$-limsup inequality) For every $\mu\in X(\Omega)$, there exists a sequence
$\left\{\mu_\ep\right\}_\ep\subset X(\Omega)$ with $\mu_\ep\in X_\ep(\Omega)$ for every $\ep>0$
such that  $\mu_\ep\fla\mu$ and
\begin{equation*}%\label{lisufor}
\limsup_{\ep\to 0}\frac{\precr(\mu_\ep)}{|\log\ep|}\le2\pi\bigg(\Big(1-\lambda)\mathrm{ess}\inf a+ \lambda\sqrt{\det\ho}\bigg)|\mu|(\Omega)\,.
\end{equation*}
\end{itemize}
\end{theorem}
%%%%%%%%%%%%%%%%%%%%%%%%%%%%%%%
%%%%%%%%%%%%%%%%%%%%%%%%%%%%%%%
%%%%%%%%%%%%%%%%%%%%%%%%%%%%%%%
\begin{theorem}\label{gltbw}
%Let $\pregl^W$ be defined in \eqref{GL0} where $a$ is a measurable $(0,1)^2$-periodic function satisfying $a(x)\in [\alpha,\beta]$ for a.e. $x\in\R^2$.
%Let moreover $\ho$ be the symmetric matrix defined in \eqref{defahom}.
If \eqref{limlog} is satisfied,
then the following $\Gamma$-convergence result holds true.
\begin{itemize}
\item[(i)] (\,$\Gamma$-liminf inequality)
Let $\{v_\ep\}_\ep\subset H^1(\Omega;\R^2)$ be such that $J v_\ep\fla\pi\mu$ for some $\mu\in X(\Omega)$\,. Then
\begin{equation*}%\label{liminfgl}
\liminf_{\ep\to 0}\frac{\pregl(v_\ep)}{|\log\ep|}\ge 2\pi\bigg(\Big(1-\lambda)\mathrm{ess}\inf a+ \lambda\sqrt{\det\ho}\bigg)|\mu|(\Omega).
\end{equation*}
\item[(ii)](\,$\Gamma$-limsup inequality) For every $\mu\in X(\Omega)$, there exists a sequence $\{v_\ep\}_\ep\subset H^1(\Omega;\R^2)$ such that  $Jv_\ep\fla\pi\mu$ and
\begin{equation*}%\label{lsgl}
\limsup_{\ep\to 0}\frac{\pregl(v_\ep)}{|\log\ep|}\le2\pi\bigg(\Big(1-\lambda)\mathrm{ess}\inf a+ \lambda\sqrt{\det\ho}\bigg)|\mu|(\Omega)\,.
\end{equation*}
\end{itemize}
\end{theorem}

\par

%We believe that the analysis developed here could be extended to cover other relevant models.
%A natural question arising in the context of our results is what happens to the functionals $\pregl$ and $\precr$ if $\delta_\ep\gg\ep$\,.
%In such a case one cannot expect that the functionals exhibit the same asymptotic behavior as that described in Theorems \ref{intro: mainthmgl} and \ref{mainthmcra}.
%Indeed, in \cite{DOn}, for $\delta_\ep=\ep^\gamma$ with $\gamma\in (0,1)$ it has been proved that, for a specific choice of the periodic function  $a$ taking values in $\{\alpha,\beta\}$,  the minimal asymptotic energy $\Lambda$ of a singularity of degree 1 satisfies
%$$
%\Lambda\le 2\pi (1-\gamma)\alpha+2\pi\gamma\sqrt{\det\ho}<2\pi\sqrt{\det\ho}\,.
%$$
%Such a result suggests that, in contrast to the case considered in this paper, for $\delta_\ep\gg\ep$ the $\Gamma$-limit depends on the size of $\delta_\ep$\,.

We conclude the introduction with a few comments and remarks about perspectives.
A natural follow-up of our results  is the extension of our analysis to GL energies with more
%as in \eqref{intro:GL0} but with the quadratic form $a(\frac{x}{\delta})|M|^2$ replaced by a more
general integrands in the leading term as those considered in the core-radius approach. A necessary first step in this direction is the proof of a homogenization result for energies defined on maps taking values in a tubular neighborhood of $\mathcal S^1$. More specifically, one could relax the $\mathcal S^1$-constraint in the functionals $F_\delta(\cdot;E)$ in \eqref{intro:cra_en_0}, assuming the latter to be defined on  $H^1(E;B_{1+\tau}\setminus B_{1-\tau})$ for some $\tau\in (0,1)$, and then study their asymptotic behavior  when both $\delta$ and $\tau$ tend to $0$\,.
Another possible extension of our model is the analysis of the case of energy density $f$ satisfying mild coercivity assumptions.
%Another interesting issue would be the analysis of our models under weaker assumptions on the coercivity of $f$.
This would allow to analyze for instance the problem of topological singularities in presence of soft inclusions of the inhomogeneous material. In this respect, an analysis on the behavior of minimizers of GL functionals in perforated domains has been carried out in \cite{BeM}.
Another challenging issue is to look at a higher-order description of the functionals $\pregl$ and $\precr$, that in the homogeneous case leads to the so-called {\it renormalized energy} governing the dynamics of the singularities (see for instance \cite{SS1} for the GL theory and \cite{ADGP} for discrete models exhibiting topological singularities). In our case of vanishing inhomogeneities we expect the corresponding renormalized energy to depend on the functional $T\!f_{\hom}$.
Furthermore, we believe that some of the techniques developed in this paper can also be used to make progress in studying stochastic homogenization problems in concentration theory, as for instance those in which the energy density $f$ is replaced by a stationary random potential.

We finally note that inhomogeneities in the GL theory can also be introduced in the potential term; e.g., considering energies of the form
\begin{equation}\label{intro:GL7}
\pregl(v):=\int_\Om |\nabla v(x)|^{2} \ud x+\frac{1}{\ep^2}\int_\Om \Bigl(a\Bigl(\frac{x}{\delta_{\e}}\Bigr)-|v(x)|^{2}\Bigr)^{2} \ud x\,.
\end{equation}
For some homogenization results for energies \eqref{intro:GL7} see \cite{BCG, BR, DSMM, DS} and the references therein. The results obtained in those papers differ from ours, since the energy in \eqref{intro:GL7} describes a different physical system, namely Type II superconductors in presence of small impurities. Note that a complete study of energies of the form \eqref{intro:GL7} may require a very complex multi-scale analysis even in the scalar case (see e.g.~\cite{DLN,BZe,CFHP}).

\vskip5pt

\textsc{Acknowledgements:} The hospitality of the Scuola Internazionale Superiore di Studi Avanzati (SISSA) where part of this research was done is gratefully acknowledged. R. Alicandro, {A. Braides}, and L. De Luca are members of the Gruppo Nazionale per l'Analisi Matematica,
la Probabilit\`a e le loro Applicazioni (GNAMPA) of the Istituto Nazionale di Alta Matematica (INdAM).
A. Braides acknowledges the MIUR Excellence Department Project awarded to the Department of Mathematics, University of Rome Tor Vergata, CUP E83C18000100006.
M. Cicalese was supported by the DFG Collaborative Research Center TRR 109, ``Discretization in Geometry and Dynamics''.
%
%he problems considered here are not the only possible homogenization results within the. In fact, one might consider also the case when the oscillating coefficients appear in the potential energy. Such a problem requires an analysis that differs from the one developed here and has been carried out in several papers (see for instance \cite{DS,DSMM,BR, BCG} and the reference therein). \red{Should we add/remove some citations?}

% [CFHP] R. Cristoferi, I. Fonseca, A. Hagerty, and C. Popovici.
% A homogenization result in the gradient theory of phase transitions.
% Interfaces Free Bound., 21 (2019), 367--408.

%%%%%%%%%%%%%%%%%%%%%%%%%
%%%%%%%%%%%%%%%%%%%%%%%%%
%%%%%%%%%%%%%%%%%%%%%%%%%
\section {Notation and preliminary results}\label{sec:npr}
\noindent {\bf Basic notation.}
\noindent Given two vectors $x,y\in \R^2$, $x \cdot y$ denotes their scalar product. As usual, the norm of $x$ is denoted by $|x| = \sqrt{x \cdot x}$. For every $r>0$ and $x\in\R^{2}$,  $B_{r}(x)$ denotes the open ball of radius $r$ centered at $x$. For $x=0$ we also write $B_{r}$ in place of $B_{r}(0)$.  $\mathcal S^{1}$ denotes the boundary of $B_{1}$, namely the unit circle in $\R^2$.
Given $a\in\R$, $\lfloor a\rfloor:=\max\{z\in\Z:\, z\le a\}$ and $\lceil a \rceil:= \min\{z\in\Z:\, z\ge a\}$ denote the integer parts of $a$ from below and from above, respectively. The imaginary unit is denoted by  $\iota\in\mathbb C$ and the complex number $e^{\iota a}=\cos a+\iota \sin a\in\mathbb C$ is identified with the Euclidean vector $(\cos a,\sin a)\in\R^{2}$. The identification extends to all $\mathcal S^1$-valued maps that can be intended as complex functions as well, if needed. In particular, for every $z\in\Z$, by $(x/|x|)^z$ we mean the complex function obtained by taking the $z$-th complex power of the function $x/|x|$. We say that a family $\{g_\eta\}_{\eta}$ converges to $g_0$ as $\eta\to 0$ in the topology $\mathscr T$, and we write $g_{\eta}\overset{\mathscr{T}}{\longrightarrow} g_0$ whenever $g_{\eta_n}\overset{\mathscr{T}}{\longrightarrow} g_0$ for any null sequence $\{\eta_n\}_{n\in\N}$. With a little abuse of terminology the family $\{g_\eta\}_{\eta}$\, is still called a sequence. The letter $C$ denotes a positive constant whose value may change each time we write it. \\

\noindent {\bf Weak star and flat convergence.}
Let $\Omega\subset\R^{2}$ be an open and bounded set with Lipschitz boundary.  $C_c(\Omega)$ denotes the space of continuous functions compactly supported in $\Omega$
endowed with the supremum norm. We say that a sequence $\{\mu_n\}_{n\in\N}$ of measures converges weakly star in $\Omega$ to a measure $\mu$, and we write $\mu_{n}\weakstar\mu$ if for any $\ffi\in C_c(\Omega)$
$$
\langle\mu_n,\ffi\rangle\to\langle\mu,\ffi\rangle\qquad\textrm{as }n\to+\infty\,.
$$
$C^{0,1}(\Om)$ denotes the space of Lipschitz continuous functions on $\Om$  endowed with the following norm
$$
\|\psi\|_{C^{0,1}} := \sup_{x\in\Om} |\psi(x)|+ \sup_{\newatop{x,y\in\Om}{  x\neq y}} \frac{|\psi(x) - \psi(y)|}{|x-y|},
$$
and we let $C^{0,1}_c(\Om)$ be its subspace of functions with compact support.
The norm in the dual of $C_c^{0,1}(\Omega)$ will be denoted by $\|\cdot\|_\flap$ and referred to as {\it{flat norm}}, while
$\fla$ denotes the convergence with respect to this norm.\\

%%%%%%%%%%%%%%%%%%%%%%%%%%%%%%%%%%%%%%%
%%%%%%%%%%%%%%%%%%%%%%%%%%%%%%%%%%%%%%%
%%%%%%%%%%%%%%%%%%%%%%%%%%%%%%%%%%%%%%%
\noindent {\bf Jacobian, current, degree.}
Given $v=(v^1,v^2)\in H^1(\Om;\R^2)$, the Jacobian $J v$ of $v$ is the $L^1$
function defined as follows
$$
J v:= \text{det} \nabla v.
$$
For every $v\in H^1 (\Om;\R^2)$, we can interpret $J v$ as an
element of the dual of $C^{0,1}_c(\Om)$ by setting
$$
\langle J v , \psi\rangle := \int_{\Om} J v  \, \psi \, \ud x\,, \qquad \text{ for any }
\psi\in C^{0,1}_c({\Om}).
$$
Notice that $Jv$ can be written in a
divergence form as $J v= \text{div } (v^1 v^2_{x_2}, - v^1 v^2_{x_1})$, i.e., for any $\psi \in C^{0,1}_c(\Om)$,
\begin{equation}\label{jwds}
\langle J v, \psi\rangle= - \int_{\Om} v^1 v^2_{x_2} \psi_{x_1} - v^1
v^2_{x_1} \psi_{x_2} \ud x.
\end{equation}
Equivalently,   we have $J v = \cu (v^1 \nabla v^2)$ and $J v = \frac{1}{2} \cu j(v)$,
where
$$
j(v):= v^1 \nabla v^2 - v^2 \nabla v^1
$$
is the so-called {\it current} associated to $v$.

Let $A\subset \Om$ be an open set with Lipschitz boundary, and let $h\in H^{\frac12}(\partial A; \R^2)$ with $|h|\ge c >0$. The {\it degree} of $h$ is defined as
\begin{equation*}%\label{defcur2}
{{\deg(h,\partial A):= \frac{1}{2\pi} \int_{\partial A} \frac{h}{|h|}\cdot \frac{\partial}{\partial \tau} \Big(\frac{h_2}{|h|},- \frac{h_1}{|h|}\Big)\ud\Huno\,,}}
\end{equation*}
where $\tau$ is the tangent field to $\partial A$ {{and the product in the above formula is understood in the sense of the duality between $H^{\frac 1 2}$ and $H^{-\frac{1}{2}}$}}\,.
In \cite{Bo,BN} it is proven that the definition above is well-posed, it is  stable with respect to the strong convergence in $H^\frac12(\partial A;\R^2\setminus B_c)$ and that  $\deg(h,\partial A)\in\Z$.
Moreover, if $v\in H^1(A;\R^2\setminus B_c)$ then $\deg(v,\partial A) = 0$ (here and in  what follows we identify $v$ with its trace). Finally, if {$v\in H^1(A;\R^2)$ and}  $|v|=1$ on $\partial A$, by Stokes' theorem (and by approximating $v$ with smooth functions) one has that
\begin{equation}\label{defcur}
\int_A Jv \ud x= \frac{1}{2} \int_A \cu j(v)\ud x := \frac{1}{2} \int_{\partial A} j(v) \cdot \tau \ud\Huno= \pi\deg(v,\partial A).
\end{equation}

Note that any $v\in H^1(A;\R^2\setminus B_c)$ can be written in polar coordinates as $v(x) = \rho(x) e^{\iota u(x)}$ on $\partial A$ with $|\rho|\ge c$. The function $u$ is said to be a {\it lifting } of $v$. By \cite{BZ} (see also \cite[Theorem 3 and Remark 3]{BB}), if $A$ is simply connected, then  $\deg(v,\partial A)=0$ and  the lifting can be selected in $H^{\frac12} (\partial A)$ with the map $v\mapsto u$ being continuous. For $A$ not necessarily simply connected, if $\Gamma$ is a connected component of $\partial A$ and the degree of $v$ on $\Gamma$ is equal to $z\in \Z$, then the lifting jumps on $\Gamma$ by $2\pi z$, but it can be locally selected to belong to $H^{\frac12}$. For $0<r<R$ and $\xi\in\R^2$, let $A_{r,R}(\xi):=B_R(\xi)\setminus\overline{B}_r(\xi)$ be the annulus of radii $r$ and $R$ centered at $\xi$, and let $v\in H^1(A_{r,R}(\xi);\mathcal S^1)$. Then for every cut $L$ such that $A_{r,R}(\xi)\setminus L$ is a simply connected set, there exists a lifting $u\in H^{1}(A_{r,R}(\xi)\setminus L)$ of $v$. Hence, $j(v)=\nabla u$ and from \eqref{defcur} it follows that
\begin{equation*}%\label{curllifting}
\deg(v,\partial B_r(\xi))=\frac {1} {2\pi}\int_{\partial B_r(\xi)}\nabla u\cdot\tau\ud\Huno\,.
\end{equation*}
\par
We introduce a  notion of modified Jacobian (a variant of the notion introduced  in \cite{ABO}),  which  we will use in our $\Gamma$-convergence results.
Given $0<\zeta<1$ we define for $\rho\in [0,+\infty)$ the function $T_\zeta(\rho):= \min\left\{\frac{\rho}{\zeta}, 1\right\}$. If $v\in H^1(\Om;\R^2)$ we set
\begin{equation}\label{moja}
v_{\zeta}:= T_\zeta(|v|) \frac{v}{|v|} \qquad \mbox{ and }\qquad J_\zeta v: = J v_\zeta.
\end{equation}
Note that, for every $v:=(v^1, v^2)$ and $w:=(w^1, w^2)$ belonging  to $H^{1}(\Om;\R^2)$
it holds
\begin{equation}\label{for}
J v - J w = \frac{1}{2} \big(J (v^1 - w^1, v^2 + w^2) - J
(v^2-w^2,v^1+w^1)\big).
\end{equation}
Gathering together \eqref{jwds} and \eqref{for} one deduces the following lemma.
\begin{lemma}\label{cosu}
There exists a universal constant $C>0$ such that for any $v,w\in H^1(\Om;\R^2)$ it holds
$$
\|J v - J w\|_\flap\le C\,\|v - w\|_{2} (\|\nabla v\|_2 + \|\nabla w\|_{2})\,.
$$
\end{lemma}
As a corollary of Lemma \ref{cosu} we obtain  the following proposition.
\begin{proposition}\label{sempre}
Let $\{v_\ep\}_{\ep}$ be a sequence in $H^1(\Om;\R^2)$ such that $GL_\ep(v_\ep) \le C|\log\e|$, and let $\eta\in(0,\frac 12)$.
Then there exists $C_\eta>0$ such that
\begin{align*}
 \sup_{ \zeta\in(\eta,1-\eta)}\| J v_\ep - J_{\zeta} v_\ep\|_{\flap}\le  C_\eta\,\ep|\log\ep|, \qquad
 \sup_{ \zeta\in(\eta,1-\eta)}|J_{\zeta}v_\ep|(\Omega)\le C_\eta |\log\ep|\,.
\end{align*}
\end{proposition}
%%%%%%%%%%%%%%%%%%%%%%%%%%%%%%%%%%%%%%%
%%%%%%%%%%%%%%%%%%%%%%%%%%%%%%%%%%%%%%%
%%%%%%%%%%%%%%%%%%%%%%%%%%%%%%%%%%%%%%%

\noindent {\bf Periodic homogenization of energies defined on $\mathcal{S}^1$-valued maps.}
In the following paragraph we state some useful propositions regarding the periodic homogenization of energy functionals defined on maps from $\R^{2}$ to $\mathcal{S}^1$. The propositions below have been proven in \cite{BM} in the more general case of manifold-valued maps defined on $\R^d$ with $d\in\N$. We specialize them here in the $\mathcal{S}^1$-version that we exploit in the following sections.

 Let $f:\R^2\times \R^{2\times 2}\to [0,+\infty)$ be a Carath\'eodory function satisfying assumptions \eqref{Pprop} and \eqref{Gprop}.
For every $\delta>0$ and for every open bounded set $E\subset\R^2$ we define the functional $\FF_\delta(\cdot,E): L^2(E;\R^2)\to [0,+\infty]$ as
\begin{equation}\label{manhom0}
\FF_{\delta}(v;E)=\begin{cases}
\displaystyle \int_{E}f\Big(\frac x \delta, \nabla v\Big)\ud x&\textrm{if }v\in H^1(E;\mathcal{S}^1)\,,\\
+\infty&\textrm{otherwise}.
\end{cases}
\end{equation}
For every $s=(s^1,s^2)\in \mathcal{S}^1$ we set $s^{\perp}=(-s^2,s^1)$ and
 $\T_s(\mathcal S^1)=\R s^\perp=\{\lambda s^\perp\,:\,\lambda\in\R\}$
denotes the tangent space of $\mathcal S^1$ at the point $s$. We also introduce the set
 $$
 \T\mathcal S^1:=\{(s,M):s\in\mathcal S^1,\,M=s^\perp\otimes\xi, \xi\in\R^2\}
 $$
  and for every $(s,M)=(s,s^\perp\otimes\xi)\in \T\mathcal{S}^1$ we define
\begin{equation}\label{genhom1}
\begin{aligned}
T\!f_{\hom}(s,M):=&\lim_{t\to +\infty}\inf\Big\{\frac{1}{t^2}\int_{tQ}f(y,M+\nabla \phi(y))\ud y\,:\,\phi\in W_0^{1,\infty}(tQ;\T_s(\mathcal{S}^1))\Big\}\,\\
=&\lim_{t\to +\infty}\inf\Big\{\frac{1}{t^2}\int_{tQ}f(y,s^\perp\otimes(\xi+\nabla \ffi(y)))\ud y\,:\,\ffi\in W_0^{1,\infty}(tQ)\Big\}\,,
\end{aligned}
\end{equation}
where $Q:=(0,1)^2$\,.
The function $T\!f_{\hom}$ is called the {\sl tangential homogenization} of the function $f$.

 The function $T\!f_{\hom}$ is a {\sl tangentially quasi-convex} function according to the following definition. We say that a Borel function $h:\T\mathcal{S}^1\to[0,+\infty)$ is tangentially quasi-convex if for all $(s,M)\in\T\mathcal{S}^1$ and all $\varphi\in W^{1,\infty}_{0}(Q;\T_{s}\mathcal{S}^1)$ it holds
\begin{equation}\label{def:3QC}
h(s,M)\leq\int_{Q}h(s,M+\nabla \phi(y))\ud y.
\end{equation}

We note that the function $T\!f_{\hom}$ satisfies the following property:
\begin{eqnarray}\label{2C}
\alpha|M|^2\le T\!f_{\hom}(s,M) \le \beta|M|^2\qquad\textrm{for every }(s,M)\in\T\mathcal S^1\,.
\end{eqnarray}
Moreover, if $f(x,\cdot)$ is $2$-homogeneous for almost every $x\in\R^2$, i.e., if $f(x,\lambda M)=\lambda^2 f(x,M)$ for almost every $x\in\R^2$ and every $M\in\R^{2\times 2},\,\lambda\in\R$,
 then also $T\!f_{\hom}$ satisfies that
\begin{equation*}%\label{2hom}
T\!f_{\hom}(s,\lambda M)=\lambda^2\, T\!f_{\hom}(s,M)\quad\textrm{for every }(s,M)\in\T\mathcal S^1,\,\lambda\in\R.
\end{equation*}

We define the functional $F_{\hom}(\cdot; E): L^2(E;\R^2)\to [0,+\infty]$ as
\begin{equation}\label{manhom00}
F_{\hom}(v;E)=\begin{cases}
\displaystyle \int_{E}T\!f_{\hom} (v(x),\nabla v(x))\ud x&\textrm{if }v\in H^1(E;\mathcal{S}^1)\,,\\
+\infty&\textrm{otherwise}.
\end{cases}
\end{equation}
The following theorem has been proven in \cite[Theorem 1.1]{BM}.

\begin{theorem}\label{homogen_thm0}
Let $\{\FF_{\delta}(\cdot;E)\}_{\delta}$ be the sequence of functionals defined in \eqref{manhom0}. Then, as $\delta\to 0$, $\{\FF_{\delta}(\cdot;E)\}_{\delta}$ $\Gamma$-converge with respect to the strong $L^2$-convergence to the functional $F_{\hom}(\cdot; E)$ in \eqref{manhom00}.
\end{theorem}
%%%%%%%%%%%%%%%%%%%%%%%%%%
%%%%%%%%%%%%%%%%%%%%%%%%%%
%%%%%%%%%%%%%%%%%%%%%%%%%%
\begin{remark}\label{cutannforpsi}
{\rm
Note that if $f$ is of the form
\begin{equation}\label{simplf}
 f(y,M)=a(y)|M|^2 \quad\textrm{ for some $Q$-periodic Borel function } a:\R^2\to [\alpha,\beta],
 \end{equation}
then for every $(s,M)=(s,s^\perp\otimes\xi)\in\T\mathcal S^1$ we have that $f(y,M)=a(y)|\xi|^2$. Therefore, by \eqref{genhom1} and by standard homogenization results of quadratic forms (see \cite[Theorem 14.7]{BrDef}),
$$
T\!f_{\hom}(s,M)=\langle \ho\xi,\xi\rangle,
$$
where $\ho$ is the symmetric matrix defined in \eqref{defahom}.
}
\end{remark}
%%%%%%%%%%%%%%%%%%%%%%%%%%
%%%%%%%%%%%%%%%%%%%%%%%%%%
%%%%%%%%%%%%%%%%%%%%%%%%%%
For every $0<r<R$ and for every $x\in\R^2$ we set $A_{r,R}(x):=B_R(x)\setminus\overline{B}_r(x)$ and  $A_{r,R}:=A_{r,R}(0)$.
Moreover, for every $z\in\Z\setminus\{0\}$ we define
\begin{equation*}%\label{admring}
\Anew_{r,R}(z):=\{w\in H^1(A_{r,R};\mathcal{S}^1)\,:\,\deg(w,\partial B_r)=z\}\,.
\end{equation*}

Given $z\in\Z\setminus\{0\}$, for every $\delta>0$ we define the functionals $\FF_{\delta}^z(\cdot; A_{r,R}): H^1(A_{r,R};\R^2)\to [0,+\infty]$ as
\begin{equation*}%\label{manhom}
\FF_{\delta}^z(v;A_{r,R}):=\begin{cases}
\displaystyle \FF_\delta(v;A_{r,R})&\textrm{if }v\in\Anew_{r,R}(z)\,,\\
+\infty&\textrm{otherwise},
\end{cases}
\end{equation*}
and $\FF_{\hom}^z:H^1(A_{r,R};\R^2)\to [0,+\infty]$ as
\begin{equation}\label{homoF}
\FF_{\hom}^z(v;A_{r,R}):=\begin{cases}
\displaystyle \FF_{\hom}(v;A_{r,R})&\textrm{if }v\in\Anew_{r,R}(z)\,.\\
+\infty&\textrm{otherwise},
\end{cases}
\end{equation}
The next result is a consequence of Theorem \ref{homogen_thm0}.

\begin{theorem}\label{homogen_thm}
Let $z\in\Z\setminus\{0\}$ and let $\FF_{\delta}^z(\cdot;A_{r,R})$ be the functional defined in \eqref{homoF}. Then, as $\delta\to 0$, $\FF_{\delta}^z(\cdot;A_{r,R})$ $\Gamma$-converge with respect to the strong $L^2$-convergence to the functional $F_{\hom}^z(\cdot; A_{r,R})$.
\end{theorem}
%%%%%%%%%%%%%%%%%%%%%%%%%%
%%%%%%%%%%%%%%%%%%%%%%%%%%
%%%%%%%%%%%%%%%%%%%%%%%%%%
\begin{proof}
It is enough to prove that the constraint $\deg(v,\partial B_r)=z$ is closed with respect to the weak convergence in $H^1(A_{r,R}; \R^2)$\,.
Let $\{w_\delta\}_\delta\subset \Anew_{r,R}(z)$ be such that
$w_\delta\weakly w_0$ in $H^1(A_{r,R};\R^2)$ for some $w_0\in H^1(A_{r,R};\mathcal S^1)$\,.
 By standard Fubini arguments, for almost every $r<\rho<R$, we have that the trace of $w_\delta$ on $\partial B_\rho$ is bounded in $H^1(\partial B_\rho;\mathcal S^1)$ and hence (up to a not relabeled subsequence) it weakly converges to a function $g_\rho$\,. Since $\|w_\delta-w_0\|_{L^2(A_{r,R};\R^2)}\to 0$\,, we get that $g_\rho=w_0$ for a.e. $\rho\in (r,R)$. By the very definition of degree in \eqref{defcur}, $\deg(w_0,\partial B_r)=z$ and hence $w_0\in\Anew_{r,R}(z)$.
\end{proof}
%%%%%%%%%%%%%%%%%%%%%%%%%%%%%%
%%%%%%%%%%%%%%%%%%%%%%%%%%%%%%
%%%%%%%%%%%%%%%%%%%%%%%%%%%%%%
The following corollary holds true as a consequence of \eqref{Gprop}, \eqref{2C}, Theorem \ref{homogen_thm} and thanks to the well-known property of convergence of minima in $\Gamma$-convergence (see \cite{B02, B06, DM}).
%%%%%%%%%%%%%%%%%%%%%%%%%%%%%%
%%%%%%%%%%%%%%%%%%%%%%%%%%%%%%
%%%%%%%%%%%%%%%%%%%%%%%%%%%%%%
\begin{corollary}\label{homdeg0}
Let $z\in\Z\setminus\{0\}$. Then, for every $0<r<R$, it holds
\begin{equation*}%\label{lbhom00}
\lim_{\delta\to 0}\inf_{w\in\Anew_{r,R}(z)}F_{\delta}(w;A_{r,R})=\min_{w\in\Anew_{r,R}(z)}F_{\hom}(w;A_{r,R})\,.
\end{equation*}
\end{corollary}
%%%%%%%%%%%%%%%%%%%%%%%%%%%%%%%%%%%%%%%
%%%%%%%%%%%%%%%%%%%%%%%%%%%%%%%%%%%%%%%
%%%%%%%%%%%%%%%%%%%%%%%%%%%%%%%%%%%%%%%

\section{The effective energy of a singularity}
In this section we introduce and discuss the properties of the minimal energy cost $\Psi(z;h)$ of a vortex like singularity of degree $z$ for a homogeneous quadratic functional of energy density $h$ defined on $\mathcal S^{1}$-valued maps.
The function $\Psi(\cdot;h)$ is crucial in order to determine the $\Gamma$-limits for both the cases $\delta_\ep\lesssim\ep$ and $\delta_\ep\gg\ep$\,, choosing $h=T\!f_{\hom}$\,, with  $T\!f_{\hom}$ defined in \eqref{genhom1}.
% being the effective energy cost of a singularity of degree $z$ on all the scales that are ``large enough'' with respect to $\delta_\ep$\,.
On the one hand (see Section \ref{sec:delta<ep}) for $\delta_\ep\lesssim\ep$\,,  $\Psi(z;T\!f_{\hom})$ turns out to be the effective energy cost of a singularity of degree $z$  (see Theorems \ref{mainthmcradopo} and \ref{mainthmgldopo}).
On the other hand (see Section \ref{sec:delta>ep})  for $\delta_\ep\gg\ep$\,, recalling the definition of $\lambda$ in \eqref{limlog}, we have that
 $\lambda\Psi(z;T\!f_{\hom})$ is the effective energy cost of a singularity of degree $z$ on scales of order between $\delta_\ep$ and $1$\, (see Theorems \ref{cradelta>ep} and \ref{gldelta>ep})\,.
%on all the scales that are ``large enough'' with respect to $\delta_\ep$\,.

Let $h:\T\mathcal S^1\to[0,+\infty)$ be a continuous function,  tangentially quasi-convex according to \eqref{def:3QC},
and such that
\begin{equation}
\label{3H}
h(s,\lambda M)=\lambda^2\, h(s,M)\,,\quad\textrm{for every }(s,M)\in \T\mathcal S^1\textrm{ and }\forall\,\lambda\in\R\,.
\end{equation}
Assume moreover that there exist $\alpha,\beta$ such that $0<\alpha\le\beta$ and
\begin{equation}
\label{3C}
\alpha|M|^2\le h(s,M) \le \beta|M|^2\,,\quad\textrm{for every }(s,M)\in \T\mathcal S^1\,.
\end{equation}
 For every open bounded set $E\subset\R^2$ we define the functional $H(\cdot;E):H^1(E;\mathcal S^1)\to[0,+\infty)$ as
\begin{equation*}%\label{defbarG}
H(w;E):=\int_{E} h(w(x),\nabla w(x))\ud x\,.
\end{equation*}
Given $z\in\Z\setminus\{0\}$ and $0<r<R$ we set
\begin{equation}\label{glim0}
\psi_{r,R}(z;h):=\frac{1}{\log \frac R r}\min_{w\in\Anew_{r,R}(z)}H(w;A_{r,R})\,.
\end{equation}
Making the  change of variable $y=\frac{x}{r}$ and considering the $2$-homogeneity \eqref{3H} of the function $h$ we conclude that,
for every $w\in H^1(A_{r,R};\mathcal{S}^1)$, the following relation holds:
\begin{equation}\label{scal0}
\begin{aligned}
H(w;A_{r,R})=&\int_{A_{r,R}} h(w(x),\nabla w(x))\ud x=\int_{A_{1,\frac{R}{r}}} h(\hat w(y),\nabla \hat{w}(y))\ud y=H(\hat w;A_{1,\frac{R}{r}})\,,
\end{aligned}
\end{equation}
where $\hat{w}(y):=w(ry)$. Gathering together \eqref{scal0} and \eqref{glim0} we deduce that
\begin{equation}\label{psiscal}
\psi_{r,R}(z;h)=\psi_{1,\frac{R}{r}}(z;h)\,.
\end{equation}
%%%%%%%%%%%%%%%%%%%%%%%%%%%%%%
%%%%%%%%%%%%%%%%%%%%%%%%%%%%%%
%%%%%%%%%%%%%%%%%%%%%%%%%%%%%%
\begin{proposition}\label{prop:glim2}
Let $h:\T\mathcal S^1\to[0,+\infty)$ be a Carath\'eodory function satisfying \eqref{3C} and \eqref{3H}. For $z\in\Z\setminus\{0\}$ and $0<r<R$ let $\psi_{r,R}(z,h)$ be the function defined in \eqref{glim0}. Then there exists the limit
\begin{equation}\label{form:glim2}
\psi(z;h):=\lim_{\frac{R}{r}\to +\infty} \psi_{r,R}(z;h).
\end{equation}
\end{proposition}
\begin{proof}
In view of \eqref{psiscal}, it is enough to prove the inequality
\begin{equation}\label{glim2:claim}
\limsup_{R\to +\infty}\psi_{1,R}(z;h)\le \liminf_{R\to +\infty} \psi_{1,R}(z;h)\,.
\end{equation}
For $\rho\in\mathbb R$ with $1<\rho<R$ we define $K_{R,\rho}:=\lfloor \frac{\log R}{\log \rho}\rfloor$ and note that
$$
A_{1,R}\supset \bigcup_{k=1}^{K_{R,\rho}}A_{\rho^{k-1},\rho^k}(\rho)\,.
$$
Denoting by $w_R$ a minimizer of \eqref{glim0}, letting $\bar k=\bar k_{R,\rho}\in\{1,\ldots,K_{R,\rho}\}$ be such that
$$
H(w_R;A_{\rho^{\bar k-1},\rho^{\bar k}})\le H(w_R;A_{\rho^{k-1},\rho^{k}})\qquad \textrm{for all }k=1,\ldots,K_{R,\bar\rho}\,,
$$
and setting $\hat w_{\rho,\bar k}(y):=w_R(\rho^{\bar k-1}y)$, we obtain
\begin{equation*}%\label{four000}
\begin{aligned}
\min_{w\in \Anew_{1,R}(z)}
H(w;A_{1,R})&\ge \sum_{k=1}^{K_{R,\rho}} H(w_R;A_{\rho^{k-1},\rho^{k}})
\ge K_{R,\rho}\,  H(w_R;A_{\rho^{\bar k-1},\rho^{\bar k}})\\
&=K_{R,\rho}H(\hat w_{\rho,\bar k};A_{1,\rho})\,,
\end{aligned}
\end{equation*}
where the last equality follows by \eqref{3H}. By the very definition of $K_{R,\rho}$ we conclude that
\begin{equation*}
\psi_{1,R}(z;h)\ge K_{R,\rho}\frac{\log\rho}{\log R}\, \psi_{1,\rho}(z;h)\ge \Big(1-\frac{\log\rho}{\log R}\Big)\psi_{1,\rho}(z;h).
\end{equation*}
The inequality above yields \eqref{glim2:claim} on taking first the limit as $R\to +\infty$ and then as $\rho\to +\infty$.
\end{proof}
Notice that
\begin{equation}\label{quapsi}
2\pi\alpha|z|^2\le \psi(z;h)\le 2\pi\beta|z|^2\qquad\textrm{for every }z\in\Z\,,
\end{equation}
where $\alpha$ and $\beta$ are the constants appearing in \eqref{3C}.
We define the function $\Psi(\cdot; h):\Z\to[0,+\infty)$ as
\begin{equation}\label{defPsi}
\Psi(z; h):=\inf\biggl\{\sum_{j=1}^J\psi(z_{j};h)\,:\,\sum_{j=1}^{J}z_{j}=z\,, J\in \N,\,z_{j}\in\Z\biggr\}.
\end{equation}
%%%%%%%%%%%%%%%%%%%%%%%%
%%%%%%%%%%%%%%%%%%%%%%%%
%%%%%%%%%%%%%%%%%%%%%%%%
\begin{remark}\label{propePsi}
\rm{
It follows from \eqref{quapsi} that the infimum in problem  \eqref{defPsi} is actually a minimum and
\begin{equation}\label{linPsi}
2\pi\alpha|z|\le \Psi(z;h)\le 2\pi\beta|z|\qquad\textrm{for every }z\in\Z\,.
\end{equation}
Moreover, by definition,  the function $\Psi(\cdot;h)$ is sub-additive, i.e.,
$$
\Psi(z_1+z_2;h)\le\Psi(z_1;h)+\Psi(z_2;h)\qquad\textrm{for every }z_1,z_2\in\Z\,.
$$
Such a property implies that the functional $\Fs(\cdot;h):X(\Omega)\to [0,+\infty)$ defined by
$$
\Fs(\mu;h):=\sum_{i=1}^n\Psi(z_i;h)\qquad\textrm{for every }\mu=\sum_{i=1}^{n}z_i\delta_{x_i}
$$
is lower semi-continuous with respect to the flat convergence, while \eqref{linPsi} yields
$$
2\pi\alpha|\mu|(\Omega)\le \Fs(\mu;h)\le 2\pi\beta|\mu|(\Omega)\,.
$$
}
 \end{remark}

In the next proposition we show that if $f$ is of the form in \eqref{simplf}, then $\Psi(z;T\!f_{\hom})$ equals $|z|$ up to a constant pre-factor.

\begin{proposition}\label{prop:psisimplf}
 Let $f:\R^2\times \R^{2\times 2}\to [0,+\infty)$ satisfy \eqref{simplf}. Then
 \begin{equation}\label{psisimplf}
 \Psi(z;T\!f_{\hom})=2\pi\sqrt{\det\ho} |z|\qquad\textrm{for every }z\in\Z,
 \end{equation}
 where $\ho$ is defined in \eqref{defahom}.
 \end{proposition}
%%%%%%%%%%%%%%%%%%%%%%%%%%%%%%
%%%%%%%%%%%%%%%%%%%%%%%%%%%%%%
%%%%%%%%%%%%%%%%%%%%%%%%%%%%%%
\begin{proof}
For $r,\,R\in \mathbb R$,  $0<r<R$,  let $L:=\{(0,x_2)\,:\,-R\le x_2\le -r\}$ be a cut of the annulus $A_{r,R}$. Then the domain $A_{r,R}\setminus L$ is simply connected. We set
$$
\Anew^{L}_{r,R}(z):=\{u\in SBV^2(A_{r,R})\cap H^1(A_{r,R}\setminus L)\,:\,e^{\iota u}\in \Anew_{r,R}(z)\}.
$$
By Remark \ref{cutannforpsi} and by \eqref{glim0} we have
\begin{equation}\label{newp}
\begin{aligned}
\psi_{r,R}(z;T\!f_{\hom})&=\frac{1}{\log\frac R r}\min_{u\in \Anew_{r,R}^L(z)}\int_{A_{r,R}}\langle \ho \nabla u(x),\nabla u(x)\rangle\ud x\\
&=\frac{1}{\log\frac R r}\min_{u\in \Anew_{r,R}^L(z)}\int_{A_{r,R}}|\sqrt{\ho}\nabla u(x)|^2\ud x\,,
\end{aligned}
\end{equation}
where the last equality follows from the fact that $\ho$ is symmetric and hence $\sqrt{\ho}$ is.
Setting $\hat u(y):=u(\sqrt{\ho}y)$, we have that $\nabla \hat u(y):=\sqrt{\ho}\nabla u(\sqrt{\ho}y)$. Thus the change of variables $x=\sqrt{\ho}y$ in \eqref{newp} yields
\begin{equation}\label{newp2}
\psi_{r,R}(z;T\!f_{\hom})=\frac{\sqrt{\det\ho}}{\log\frac R r}\min_{u\in \widehat{\Anew}_{r,R}^L(z)}\int_{(\sqrt{\ho})^{-1} ({A_{r,R}})}|\nabla \hat u(y)|^2\ud y\,,
\end{equation}
where we have set
\begin{equation*}
\begin{aligned}
\widehat{\Anew}_{r,R}^L(z)&:=\{\hat u\in SBV^2((\sqrt{\ho})^{-1}(A_{r,R}))\cap H^1((\sqrt{\ho})^{-1}(A_{r,R}\setminus L))\,:\,\\
&\qquad \hat u=u\circ \sqrt{\ho}\,,\quad u\in \Anew^L_{r,R}(z)\}.
\end{aligned}
\end{equation*}
For sufficiently large $R/r$  there exist $0<\lambda<\Lambda$ that depend only on $\ho$ and do not depend on $r$ and $R$ such that $A_{\lambda r,\Lambda R}\subset (\sqrt{\ho})^{-1}(A_{r,R})$ so that, by \eqref{newp2},
\begin{equation*}
\begin{aligned}
\psi_{r,R}(z;T\!f_{\hom})&= \frac{\sqrt{\det\ho}}{\log\frac R r}\min_{u\in {\Anew}_{\lambda r,\Lambda R}^L(z)} \int_{A_{\lambda r,\Lambda R}}|\nabla \hat u(y)|^2\ud y+O\Big(\frac1{\log\frac Rr}\Big)\\
&=2\pi\sqrt{\det\ho}|z|^2+O\Big(\frac1{\log\frac Rr}\Big).
\end{aligned}
\end{equation*}
It follows that $\psi(z)=2\pi\sqrt{\det\ho}|z|^2$, whence \eqref{psisimplf} follows using the very definition of $\Psi$ in \eqref{defPsi}.
\end{proof}
%%%%%%%%%%%%%%%%%%%%%%%%%%%%%%%%%%%%%%%%%%%%
%%%%%%%%%%%%%%%%%%%%%%%%%%%%%%%%%%%%%%%%%%%%
%%%%%%%%%%%%%%%%%%%%%%%%%%%%%%%%%%%%%%%%%%%%
\section{Asymptotic analysis on annuli}
In this section we prove some auxiliary results on the asymptotic behavior of the minimal energy on an annulus when its inner and outer radii are powers of $\ep$. Such results will be crucial in the proofs of the $\Gamma$-convergence theorems in Section \ref{sec:delta<ep} and in Section \ref{sec:delta>ep}.
The next lemma states that the minimum in \eqref{glim0} for $H=F_\delta$ changes by at most a multiple of $z^{2}$ if the competitors are chosen with fixed trace $(x/|x|)^{z}$ instead of fixed degree $z$, thus belonging to a new appropriate set of admissible functions defined as
\begin{equation}\label{anew}
\widetilde\Anew_{r,R}(z):=\left\{w\in\Anew_{r,R}(z)\,:\, w(x)=\Big(\frac{x}{|x|}\Big)^z\textrm{ on }\partial B_r\cup\partial B_R\right\}\,.
\end{equation}
\begin{lemma}\label{lm:glim0}
Let $0<2r\leq R$ and let $F_{\delta}(\cdot,A_{r,R})$ be the functional defined in \eqref{manhom0} for $E=A_{r,R}$ and with $f$ satisfying condition \eqref{Gprop}. Then, there exists a constant $\bar C=\bar C(\alpha,\beta)>0$ such that, for every  $z\in\Z\setminus\{0\}$,
\begin{equation}\label{lateuse0}
 \inf_{w\in\Anew_{r,R}(z)} F_{\delta}(w;A_{r,R})\le \inf_{w\in\widetilde\Anew_{r,R}(z)} F_{\delta}(w;A_{r,R})\le \inf_{w\in\Anew_{r,R}(z)} F_\delta(w;A_{r,R})+\bar Cz^2\,.
\end{equation}
\end{lemma}
%%%%%%%%%%%%%%%%%%%%%%%%%%%%%%
%%%%%%%%%%%%%%%%%%%%%%%%%%%%%%
%%%%%%%%%%%%%%%%%%%%%%%%%%%%%%
\begin{proof}
The first inequality in \eqref{lateuse0} follows from the inclusion $\widetilde\Anew_{r,R}(z)\subset \Anew_{r,R}(z)$.  Hence, it is enough to show that
for every $w\in\Anew_{r,R}(z)$ there exists $\widehat w\in \widetilde\Anew_{r,R}(z)$ such that
\begin{equation}\label{w<hat}
F_{\delta}(\widehat w;A_{r,R})\le  F_\delta(w;A_{r,R})+\bar Cz^2,
\end{equation}
for some constant $\bar C$ depending only on the constants $\alpha$ and $\beta$ in \eqref{Gprop}.
Set $K:=\lfloor\frac{\log R-\log r}{\log 2}\rfloor$ and  $A_k:=A_{2^{k-1}r,2^kr}$\, for $k=1,2,\dots,K$. We have that
$$
A_{r,R}=\bigcup_{k=1}^K A_k\cup A_{2^K r,R}\,.
$$
Since $K\ge \frac{\log R-\log r}{\log 2}-1$\,, in view of \eqref{Gprop}, we notice that
\begin{equation}\label{bdryerr}
F_{\delta}\Big(\Big(\frac{x}{|x|}\Big)^z;A_{2^K r,R}\Big)\le \beta z^2\int_{A_{2^Kr,R}}\frac{1}{|x|^2}\ud x =2\pi\beta z^2\log\frac{R}{2^K r}
\le  2\pi\beta z^2\log 2\,.
\end{equation}
We first consider the case that there is at most one annulus $A_k$ such that
\begin{equation}\label{giuin}
F_{\delta}\Big(\Big(\frac{x}{|x|}\Big)^z;A_k\Big)\ge F_{\delta}(w;A_k)\,.
\end{equation}
Then, in view of \eqref{Gprop}, we have
\begin{equation*}
F_{\delta}\Big(\Big(\frac{x}{|x|}\Big)^z;A_k\Big)\le \beta z^2\int_{A_k}\frac{1}{|x|^2}\ud x= 2\pi\beta z^2\log 2,
\end{equation*}
whence, using also \eqref{bdryerr}, we deduce $F_{\delta}\big(\big(\frac{x}{|x|}\big)^z;A_{r,R}\big)\le F_{\delta}(w;A_{r,R})+\bar Cz^2$\,,
which proves \eqref{w<hat} for $\widehat w(x)=(x/|x|)^{z}$.

From now on, we can assume that \eqref{giuin} is satisfied by at least two of the annuli $A_k$. We let $k_1$ and $k_2$ denote the smallest and the largest $k\in\{1,\ldots,K\}$ satisfying \eqref{giuin}. Let moreover $L:=\{(0,x_2)\,:\,-R\le x_2\le -r\}$ be a cut of the annulus $A_{r,R}$ such that the domain $A_{r,R}\setminus L$ is simply connected. By \cite{BZ}, there exists a lifting $u\in H^{1}(A_{r,R}\setminus L; \R)$ of $w$ in $A_{r,R}\setminus L$. Moreover, since $\deg(w, \partial B_\rho)=z$ for every $\rho\in[r,R]$, we have that the function $u$ jumps by $2\pi z$ across $L$.
By the properties of the lifting,
\begin{equation}\label{used}
\|\nabla u\|_{L^2(E;\R^2)}=\|\nabla w\|_{L^2(E,;\R^{2\times 2})}\qquad\textrm{for every open set }E\subset A_{r,R}.
\end{equation}
Furthermore,
setting
\begin{equation}\label{deftheta}
\theta(x):=\begin{cases}
\arctan\frac{x_2}{x_1}&\textrm{ if }x_1>0\,,\\
\frac{\pi} 2&\textrm{ if }x_1=0, x_2>0\,,\\
\pi+\arctan\frac{x_2}{x_1}&\textrm{ if }x_1<0\,,\\
\frac{3} 2\pi&\textrm{ if }x_1=0, x_2<0\,,
\end{cases}
\end{equation}
for every $x\in\R^2\setminus \{0\}$\,, the function $z\theta\in SBV^2(A_{r,R})$ is a lifting of $(\frac{x}{|x|})^z$.
Using the complex notation we set $\widehat w:=e^{\iota\widehat u}$, where the lifting $\widehat u$ is defined as
\begin{equation}\label{constrw}
\widehat u(x):=\left\{\begin{array}{ll}
 z\theta(x)&\textrm{ if } r\le|x|\le 2^{k_1-1}r\,,\\           %x\in A_{1,2^{k_1-1}}\\
(1-\sigma_1(|x|))z\theta(x)+\sigma_1(|x|)u(x)&\textrm{ if }2^{k_1-1}r \le|x|\le 2^{k_1}r\,,\\%x\in A_{k_1}\\
u(x)&\textrm{ if }2^{k_1}r \le|x|\le 2^{k_2-1}r\,,\\   %x\in A_{2^{k_1},2^{k_2-1}}\\
\sigma_2(|x|)z\theta(x)+(1-\sigma_2(|x|))u(x)&\textrm{ if }2^{k_2-1}r\le |x|\le 2^{k_2}r\,,\\ % \in A_{k_2}\\
z\theta(x)&\textrm{ if }2^{k_2}r\le |x|\le R\,. %x\in A_{2^{k_2},R}\,.
\end{array}\right.
\end{equation}
In the formula above, for $i=1,2$ the function  $\sigma_i:[2^{k_i-1}r,2^{k_i}r]\to [0,1]$ is defined by
\begin{equation*}
\sigma_i(\rho):=\frac{1}{2^{k_i-1}r}(\rho-2^{k_i-1}r)
\end{equation*}
and satisfies
%\begin{equation*}%\label{etaproper}
$\sigma_i'(\rho)=\frac{1}{2^{k_i-1}r}\textrm{ and }\|\sigma_i\|_{L^\infty}\le 1$.
%\end{equation*}

Note that $\widehat w\in \widetilde\Anew_{r,R}(z)$.
By construction and by \eqref{bdryerr}, we have that
\begin{equation*}%\label{0form}
\begin{aligned}
F_{\delta}(\widehat w; A_{r,2^{k_1-1}r}\cup A_{2^{k_2}r,R})&=F_{\delta}\Big(\Big(\frac{x}{|x|}\Big)^z; A_{r,2^{k_1-1}r}\cup A_{2^{k_2}r,R}\Big)\\
&\le F_{\delta}(w; A_{r,2^{k_1-1}r}\cup A_{2^{k_2}r,R})+Cz^2\,.
\end{aligned}
\end{equation*}
Therefore, in view of \eqref{constrw} it is enough to prove that the energy of $\widehat w$ on $A_{k_1}$ and $A_{k_2}$ is bounded from above by $\bar{C}|z|^2$ for some constant $\bar C>0$ depending only on $\alpha$ and $\beta$. We prove this fact only for the annulus $A_{k_1}$, being the proof for $A_{k_2}$ similar.
To this end, we notice that in $A_{k_1}$ one has that
\begin{equation}\label{gradhatw}
\begin{aligned}
|\nabla \widehat w|^2&= |\nabla \widehat u|^2\le 3|\sigma_1'|^2|u-z\theta|^2+3|\sigma_1|^2|\nabla u-z\nabla\theta|^2+3z^{2}|\nabla\theta|^{2}\\
&\le  \frac{C}{2^{2(k_1-1)}r^2}(|u|^2+z^2|\theta|^2)+C(|\nabla u|^2+z^2|\nabla \theta|^2)\,.
\end{aligned}
\end{equation}
Set $l_{k_1}:=\dashint_{A_{k_1}} u\ud x$\,. Up to adding an integer multiple of $2\pi$, we can always assume that $|l^{k_1}|\le 2\pi$ and estimate
$\|l_{k_{1}}\|_{L^2(A_{k_1})}^{2}\leq(2\pi(|z|+1)^{2})2^{2k_{1}}r^{2}$. Hence
\begin{equation}\label{meanu}
\begin{aligned}
\|u\|_{L^2(A_{k_1})}^2&\le 2\|u-l_{k_1}\|_{L^2(A_{k_1})}^2+2\|l_{k_1}\|_{L^2(A_{k_1})}^2\\
&\le C2^{2k_1}r^2\|\nabla u\|_{L^2(A_{k_1};\R^2)}^2+ Cz^2 2^{2k_1}r^2\\
&\le C 2^{2k_1}r^2z^2 \|\nabla \theta\|_{L^2(A_{k_1};\R^2)}^2+Cz^2 2^{2k_1}r^2\le Cz^2 2^{2k_1}r^2,
\end{aligned}
\end{equation}
where the second inequality is a consequence of the Poincar\'e-Wirtinger inequality applied to the domain $A_{k_1}\setminus L$, and the third inequality follows on gathering together \eqref{giuin}, \eqref{used}, and \eqref{Gprop}. Note that all the constants appearing in \eqref{meanu} depend only on $\alpha$ and $\beta$.
By integrating \eqref{gradhatw} and using \eqref{meanu}, \eqref{giuin}, \eqref{used} and \eqref{Gprop},
 we deduce that
\begin{equation*}
\begin{aligned}
F_{\delta}(\widehat w; A_{k_1})&\le C\|\nabla \widehat w\|_{L^2(A_{k_1;\R^{2\times 2}})}^2\\
&\le\frac{C}{2^{2(k_1-1)}r^2}(\|u\|_{L^2(A_{k_1})}^2+z^2\|\theta\|_{L^2(A_{k_1})}^2)+
C (\|\nabla u\|_{L^2(A_{k_1};\R^2)}^2+z^2\|\nabla \theta\|_{L^2(A_{k_1};\R^2)}^2)
\\
&\le \frac{C\, 2^{2k_1}z^2}{2^{2(k_1-1)}}+Cz^2\|\nabla \theta\|_{L^2(A_{k_1};\R^2)}^2=:\bar{C}z^2\,,
\end{aligned}
\end{equation*}
thus concluding the proof of \eqref{w<hat}.
\end{proof}

In the next proposition we show that in the $|\log\ep|$ regime, to some extent, the homogenization process commutes  with the minimization process defining $\psi(d;T\!f_{\hom})$.
%%%%%%%%%%%%%%%%%%%%%%%%%%%
%%%%%%%%%%%%%%%%%%%%%%%%%%%
%%%%%%%%%%%%%%%%%%%%%%%%%%%
\begin{proposition}\label{homdeg}
Let $\precrw$ be defined in \eqref{intro:cra_en_0} with $f$ satisfying assumptions \eqref{Pprop}, \eqref{Gprop}, \eqref{Hprop}, and let $T\!f_{\hom}$ be defined in \eqref{genhom1}.
Then for any $s_1$ and $s_2$ such that $0\le s_1<s_2<1$ and $\lim_{\ep\to 0}\frac{\delta_\ep}{\ep^{s_2}}=0$ we have
\begin{equation}\label{limhom}
\lim_{\ep\to 0}\frac{1}{|\log\ep|}\inf_{w\in \Anew_{\ep^{s_2},\ep^{s_1}}(z)}
\precrw(w;A_{\ep^{s_2},\ep^{s_1}})
=(s_2-s_1)\psi(z;T\!f_{\hom})\,,
\end{equation}
where $\psi(z;T\!f_{\hom})$ is the function defined in \eqref{form:glim2} with $h=T\!f_{\hom}$\,.
\end{proposition}
%%%%%%%%%%%%%%%%%%%%%%%%%%%%%%%%%%%%%%%%%%%%
%%%%%%%%%%%%%%%%%%%%%%%%%%%%%%%%%%%%%%%%%%%%
%%%%%%%%%%%%%%%%%%%%%%%%%%%%%%%%%%%%%%%%%%%%
\begin{proof}
We first show that
\begin{equation}\label{lbhom}
\liminf_{\ep\to 0}\frac{1}{|\log\ep|}\inf_{w\in \Anew_{\ep^{s_2},\ep^{s_1}}(z)}
\precrw(w;A_{\ep^{s_2},\ep^{s_1}})
\ge (s_2-s_1)\psi(z;T\!f_{\hom})\,.
\end{equation}
To this purpose, we fix $R>1$, set $K_{\ep,R}=\lfloor (s_2-s_1)\frac{|\log\ep|}{\log R}\rfloor$, and note that
$\displaystyle
A_{\ep^{s_2},\ep^{s_1}}\supset \bigcup_{k=1}^{K_{\ep,R}}A_{R^{k-1}\ep^{s_2},R^{k}\ep^{s_2}}$.
Let moreover $w_\ep\in \Anew_{\ep^{s_2},\ep^{s_1}}(z)$  be such that
\begin{equation}\label{nuo}
\precrw(w_\ep ;A_{\ep^{s_2},\ep^{s_1}})\le \inf_{w\in \Anew_{\ep^{s_2},\ep^{s_1}}(z)}
\precrw(w;A_{\ep^{s_2},\ep^{s_1}})+C,
\end{equation}
for some constant $C$ (independent of $\ep$)
and let $\bar k=\bar k_{\ep,R}\in\{1,\ldots,K_{\ep,R}\}$ be such that
$$
\precrw(w_\ep;A_{R^{\bar k-1}\ep^{s_2},R^{\bar k}\ep^{s_2}})\le \precrw(w_\ep;A_{R^{k-1}\ep^{s_2},R^{k}\ep^{s_2}})\,,\qquad \textrm{for all }k=1,\ldots,K_{\ep,R}\,.
$$
Therefore
\begin{equation}\label{four0}
\precrw(w_\ep;A_{\ep^{s_2},\ep^{s_1}})\ge \sum_{k=1}^{K_{\ep,R}} \precrw(w_\ep;A_{R^{k-1}\ep^{s_2},R^{k}\ep^{s_2}})\ge K_{\ep,R}\,  \precrw(w_\ep;A_{R^{\bar k-1}\ep^{s_2},R^{\bar k}\ep^{s_2}})\,.
\end{equation}
By the change of variable $y=\frac{x}{R^{\bar k-1}\ep^{s_2}}$, $w'_{\ep,\bar k}(y):=w_\ep(R^{\bar k-1}\ep^{s_2}y)$ and by property \eqref{Hprop}, we have
\begin{equation}\label{five0}
\precrw(w_\ep;A_{R^{\bar k-1}\ep^{s_2},R^{k}\ep^{s_2}})=F_{\frac{\delta_\ep}{R^{\bar k-1}\ep^{s_2}}}(w'_{\ep,\bar k};A_{1,R})\,.
\end{equation}
Therefore, since by assumption $\limsup_{\ep\to 0}\frac{\delta_{\ep}}{R^k\ep^{s_2}}=0$ for every $k=1,\ldots,K_{\ep,R}$,  by using \eqref{nuo}, \eqref{four0}, \eqref{five0}, and Corollary \ref{homdeg0}, we deduce that
\begin{equation*}
\begin{aligned}
\liminf_{\ep\to 0}\frac{1}{|\log\ep|}\inf_{w\in \Anew_{\ep^{s_2},\ep^{s_1}}(z)}&
\precrw(w;A_{\ep^{s_2},\ep^{s_1}})\\
&\ge\liminf_{\ep\to 0} \frac{1}{|\log\ep|}\precrw(w_\ep;A_{\ep^{s_2},\ep^{s_1}})\\
&\ge \liminf_{\ep\to 0}\frac{K_{\ep,R}}{|\log\ep|}\inf_{w\in \Anew_{1,R}(z)} F_{\frac{\delta_\ep}{R^{\bar k-1}\ep^{s_2}}}(w; A_{1,R})\\
&\ge  \lim_{\ep\to 0}\Big(\frac{s_2-s_1}{\log R}-\frac{1}{|\log\ep|} \Big)\liminf_{\ep\to 0} \inf_{w\in \Anew_{1,R}(z)} F_{\frac{\delta_\ep}{R^{\bar k-1}\ep^{s_2}}}(w; A_{1,R})\\
&= \frac{s_2-s_1}{\log R}\min_{w\in \Anew_{1,R}(z)}F_{\hom}(w;A_{1,R})=(s_2-s_1)\psi_{1,R}(z;T\!f_{\hom})\,.
\end{aligned}
\end{equation*}
Formula \eqref{lbhom} follows from the estimate above as $R\to +\infty$ thanks to Proposition \ref{prop:glim2} applied to $h=T\!f_{\hom}$.
To conclude the proof of \eqref{limhom} we are left to show that
\begin{equation}\label{ubhom}
\limsup_{\ep\to 0}\frac{1}{|\log\ep|}\inf_{w\in \Anew_{\ep^{s_2},\ep^{s_1}}(z)}
\precrw(w; A_{\ep^{s_2},\ep^{s_1}})
\le (s_2-s_1)\psi(z;T\!f_{\hom})\,.
\end{equation}
To this purpose, we take $R>1$ and set $J_{\ep,R}:=\lceil (s_2-s_1)\frac{|\log\ep|}{\log R}\rceil$\,. We observe that
\begin{equation}\label{n1}
\inf_{w\in \Anew_{\ep^{s_2},\ep^{s_1}}(z)}
\precrw(w; A_{\ep^{s_2},\ep^{s_1}})
\le\sum_{j=1}^{J_{\ep,R}}\inf_{w\in \widetilde\Anew_{R^{j-1}\ep^{s_2},R^j\ep^{s_2}}(z)}
\precrw(w;A_{R^{j-1}\ep^{s_2},R^j\ep^{s_2}})\,.
\end{equation}
We also note that for every $R>1$, thanks to Corollary \ref{homdeg0}, there exists a modulus of continuity $\omega$ such that
$$
\inf F_{\delta}(w; A_{1,R})\le \min F_{\hom}(w;A_{1,R}) +\omega(\delta).
$$
Letting $\delta_{\ep,j}:=\frac{\delta_\ep}{R^{j-1}\ep^{s_2}}$ for every $j=1,\ldots,J_{\ep,R}$ and taking into account the fact that $\{\delta_{\ep,j}\}_{j}$ is decreasing yields $\delta_{\ep,j}\le \delta_{\ep,1}=\frac{\delta_{\e}}{\ep^{s_2}}\to 0$ as $\e\to 0$ for all $j=1,\ldots, J_{\ep,R}$. We set  $\omega(\delta_{\ep,\bar \jmath}):=\max_{j=1,\ldots,J_{\ep,R}} \omega(\delta_{\ep,j})$. Note that $\omega(\delta_{\ep,\bar\jmath})$ depends only on $\ep$ and $R$, moreover,   since $\delta_{\ep,\bar\jmath}\to 0$ as $\e\to 0$, \
$\omega(\delta_{\ep,\bar\jmath})$  vanishes as $\e\to 0$.
For every $j=1,\ldots, J_{\ep,R}$, using the change of variable $y=\frac{x}{R^{j-1}\ep^{s_2}}$, applying Lemma \ref{lm:glim0} with $\delta=\delta_{\ep,j}$
(see formula \eqref{lateuse0}) and Corollary \ref{homdeg0}, we have that
\begin{equation}\label{annlimsup}
\begin{aligned}
\inf_{w\in \widetilde\Anew_{R^{j-1}\ep^{s_2},R^j\ep^{s_2}}(z)}
\precrw(w; A_{R^{j-1}\ep^{s_2},R^j\ep^{s_2}})
&=\inf_{w\in \widetilde\Anew_{1,R}(z)}
F_{\delta_{\ep,j}}(w; A_{1,R})\,\\
&\le  \inf_{w\in \Anew_{1,R}(z)}
F_{\delta_{\ep,j}}(w; A_{1,R})+\bar Cz^2\\
&\le  \min_{w\in \Anew_{1,R}(z)}
F_{\hom}(w; A_{1,R})+\omega(\delta_{\ep,\bar \jmath})+\bar Cz^2\,,
\end{aligned}
\end{equation}
where the constant $\bar C>0$ is given in Lemma \ref{lm:glim0}. By combining \eqref{annlimsup} with \eqref{n1} we get that
\begin{equation*}
\begin{aligned}
\limsup_{\ep\to 0}\frac{1}{|\log\ep|}\inf_{w\in \Anew_{\ep^{s_2},\ep^{s_1}}(z)}
&\precrw(w;A_{\ep^{s_2},\ep^{s_1}})\\
&\le \frac{s_2-s_1}{\log R}\min_{w\in \Anew_{1,R}(z)}
F_{\hom}(w; A_{1,R})+ \frac{s_2-s_1}{\log R}\bar Cz^2\\
&= (s_2-s_1)\psi_{1,R}(z;T\!f_{\hom})+ \frac{s_2-s_1}{\log R}\bar Cz^2\,,
\end{aligned}
\end{equation*}
whence \eqref{ubhom} follows by taking the limit as $R\to +\infty$ and using Proposition \ref{prop:glim2} with $h=T\!f_{\hom}$.
This concludes the proof of \eqref{limhom}.
\end{proof}
%%%%%%%%%%%%%%%%%%%%%%%%%%%%%%
%%%%%%%%%%%%%%%%%%%%%%%%%%%%%%
%%%%%%%%%%%%%%%%%%%%%%%%%%%%%%
\begin{remark}\label{homdegrmk}
\rm{Note that \eqref{limhom} holds true also if the center of the annulus is a point $\xi_\ep$ depending on $\ep$\,, since all the estimates in the previous proof do not depend on the center of the annulus.}
\end{remark}
%%%%%%%%%%%%%%%%%%%%%%%%%%%%%%
%%%%%%%%%%%%%%%%%%%%%%%%%%%%%%
%%%%%%%%%%%%%%%%%%%%%%%%%%%%%%
\section{The ball construction}\label{bc:section}
{In this section we present the so-called {\it ball construction} introduced in \cite{J,S}, which provides lower bounds of the Dirichlet energy in the presence of topological singularities. We slightly revisit it, following the approach by Sandier \cite{S} and adopting the notation in \cite{DLP} (see also \cite{AP}).}

Let $\B=\{B_{r_1}(x_1), \ldots,B_{r_n}(x_n)\}$ be a finite family of open balls in $\R^2$ with disjoint closure $\bar B_{r_i}(x_i) \cap \bar B_{r_j}(x_j) = \emptyset$ for $i\neq j$ and let  $\mu=\sum_{i=1}^n z_i\delta_{x_i}$ with $z_i\in\Z\setminus\{0\}$\,.

Let moreover $\E(\B,\mu,\cdot)$ be an increasing set-function defined on open subsets of $\R^2$ satisfying the following properties:
\begin{itemize}
\item[(i)] $\E(\B,\mu,E_{1}\cup E_{2})\ge \E(\B,\mu,E_{1})+\E(\B,\mu,E_{2})$ for all $E_{1},\, E_{2}$ open disjoint subsets of $\R^2$;
\item[(ii)] for any annulus $A_{r,R}(x)=B_{R}(x)\setminus \bar{B}_r(x)$ with $A_{r,R}(x)\cap\bigcup_{i}\bar B_{r_i}(x_i)=\emptyset$, it holds
\begin{equation}\label{lbastratto}
\E(\B,\mu,A_{r,R}(x))\ge2\pi\alpha|\mu(B_r(x))|\log\frac{R}{r}\,,
\end{equation}
for some constant $\alpha>0$\,.
\end{itemize}

\begin{remark}\label{disuene0}
{\rm
Let   $ w\in H^1_{\mathrm{loc}} (\R^2\setminus \bigcup_{B\in\B} B; \mathcal S^1)$ be such that $\mu=\sum_{B\in\B} \deg(w,\partial B) \delta_{x_B}$, where   $x_B$ is the center of $B$. Then, an explicit example of admissible functional $\E(\B,\mu, \cdot)$ is given by
$$
\E(\B, \mu, A) :=  \alpha \int_{A\setminus \bigcup_{B\in\B} \overline B} |\nabla w|^2\ud x,
$$
for every open set $A\subset \R^2$\,.
For further details see Remark \ref{disuene}.
}
\end{remark}

For every ball $B\subset\R^2$, let $r(B)$ denote the radius of the ball $B$; moreover, for every  family $\Bnew$ of balls in $\R^2$ we set
\begin{equation*}%\label{defrad}
\rad(\Bnew):=\sum_{B\in\Bnew}r(B).
\end{equation*}
\begin{proposition}\label{ballconstr}
There exists a one-parameter family of open balls $\B(t)$ with $t\ge 0$  such that, setting $U(t):=\bigcup_{B\in\B(t)} B$, the following conditions are fulfilled:
\begin{enumerate}
\item $\B(0)=\B \, $;
\item $ U(t_1)\subset U(t_2)$  for any $0\le t_1<t_2 \, $;
\item the balls in $\B(t)$ are pairwise disjoint;
\item for any $0\le t_1<t_2$ and for any open set $U \subset \R^2 \,$,
\begin{equation}\label{sti3}
\E(\B,\mu,U\cap(U(t_2)\setminus \overline U(t_1)))\ge2\pi\alpha\sum_{\newatop{B\in\B(t_2)}{B\subset U}}|\mu(B)|\log\frac{1+t_2}{1+t_1}\,;
\end{equation}
\item $\rad(\B(t))\le (1+t)\rad(\B)$.
\end{enumerate}
\end{proposition}

\begin{proof}
In order to construct the family $\B(t)$,  we closely follow the  strategy of  Sandier and Jerrard in \cite{J,S}.
It consists in letting the balls alternatively expand and merge into each other as follows. In the expansion phase the balls expand, without changing their centers, in such a way that, at each (artificial) time $t$ the radius $r_i(t)$ of the ball centered at $x_i$ satisfies
\begin{equation}\label{timevel0}
\frac{r_i(t)}{r_i} = 1+t \qquad \text{ for all } i.
\end{equation}
The first expansion phase stops at the first time $T_1$ when two balls bump into each other. Then the merging phase begins. It consists in identifying a suitable partition $\{S^1_j\}_{j=1,\ldots, N_n}$ of the family $\left\{B_{r_i(T_1)}(x_i)\right\}$, and, for each subclass $S^1_j$, in  finding a ball $B_{r^1_j}(x^1_j)$ which contains all the balls in $S^1_j$ such that the following properties hold:
\begin{itemize}
\item[P1)] $B_{r^1_j}(x^1_j)\cap B_{r^1_l}(x^1_l)=\emptyset$ for all $j\neq l$;
\item[P2)] $r^1_j\ \le \sum_{B\in S^1_j} r(B)$.
\end{itemize}
After the merging phase another expansion phase begins:  we let the balls $\big\{B_{r^1_j}(x^1_j)\big\}$  expand in such a way that, for $t\geq T_1$, for every $j$ we have that
\begin{equation}\label{timevel}
\frac{r^1_j(t)}{r^1_j}=\frac{1+t}{1+T_1}\,.
\end{equation}
Again note that $r^1_j(T_1)=r^1_j$. We iterate this procedure thus obtaining a {\it discrete} set of merging times $\left\{T_1,\ldots,T_K\right\}$ with $K\le n$ and  a family $\B(t)$ for all $t\ge 0$. More precisely,  $\B(t)$ is given by $\{B_{r_j(t)}(x_j)\}_j$ for $t\in [0,T_1)$;
for  $t\in[T_k,T_{k+1})$, $\B(t)$ can be written as $\{B_{r^k_j(t)}(x_j^k)\}_j$ for all $k=1,\ldots, K-1$, while it consists of a single expanding ball for $t\ge T_K$ .
By construction, we clearly have properties (1), (2) and (3). Moreover, (5) is an easy consequence of \eqref{timevel0}, \eqref{timevel} and property P2).

It remains to show property (4). We preliminarily note that, by (2), for every open set $U\subset\R^{2}$
\begin{equation}\label{monovar}
\sum_{\newatop{B\in \B(\tau_1)}{B\subset U}}|\mu(B)|\ge \sum_{\newatop{B\in \B(\tau_2)}{ B\subset U }}|\mu(B)|\,\qquad\textrm{for any }0<\tau_1<\tau_2.
\end{equation}
 Let $t_1<\bar t<t_2$. In view of \eqref{monovar} and since $\E$ is an increasing set-function satisfying property (i), if we show that (4) holds true for  the pairs $(t_1,\bar t)$ and $(\bar t,t_2)$, then (4) follows also for $t_1$ and $t_2$. Therefore,  we can assume without loss of generality that $T_{k}\notin ]t_1,t_2[$  for any $k=1,\ldots,K$.

Let $t_1<\tau<t_2$ and let $B\in \B(\tau)$. Then there exists a unique ball $B'\in\B(t_1)$ such that $B'\sub B$. By construction, $\mu(B)=\mu(B')$ and by  \eqref{lbastratto} we have that
\begin{equation*}
\E(\B,\mu,B\setminus \bar B')\ge 2\pi\alpha|\mu(B)|\log\frac{1+\tau}{1+t_1},
\end{equation*}
which, summing up over all $B\in \B(\tau)$ with $B\subset U$, and using \eqref{monovar}, yields
\begin{eqnarray*}
\E(\B,\mu, U\cap(U(t_2)\setminus \overline U(t_1)))\ge 2\pi\alpha\sum_{\newatop{B\in \B(\tau)}{ B\subset U}}|\mu(B)|\log\frac{1+\tau}{1+t_1}\ge  2\pi\alpha\sum_{\newatop{B\in \B(t_2)}{ B\subset U}}|\mu(B)|\log\frac{1+\tau}{1+t_1} .
\end{eqnarray*}
Property (4) follows by letting $\tau\to t_2$.
\end{proof}
%%%%%%%%%%%%%%%%%%%%%%%%%%%%%%%%%%%%%%%%%%%%%%%%%%%%%%%%%%%%
%%%%%%%%%%%%%%%%%%%%%%%%%%%%%%%%%%%%%%%%%%%%%%%%%%%%%%%%%%%%
%%%%%%%%%%%%%%%%%%%%%%%%%%%%%%%%%%%%%%%%%%%%%%%%%%%%%%%%%%%%
We recall the following well-known lemma (see e.g., \cite[Lemma 2.2]{DLP}) for the reader's convenience.

\begin{lemma}\label{bc_lemma}
Let $\Bnew$ be a family of pairwise disjoint balls in $\R^2$  and let $\Cnew$ be the family of balls in $\Bnew$ which are contained in $\Omega$.
Let moreover $\nu_1, \, \nu_2$ be two Radon measures supported in $\Omega$ with
\begin{equation*}%\label{asslemma}
\supp\nu_1\subset\bigcup_{B\in\Cnew}B,\qquad \supp\nu_2\subset\bigcup_{B\in\Bnew}B\qquad\textrm{and}\quad\nu_1(B)=\nu_2(B)\quad \textrm{for any }B\in\Cnew\,.
\end{equation*}
Then, there exists a constant $C>0$ such that
$$
\|\nu_1-\nu_2\|_\flap\le C\,\rad (\Bnew)(|\nu_1|+|\nu_2|)(\Omega)\,.
$$
\end{lemma}

\section{General $\Gamma$-liminf inequality}\label{sec:genlinf}
In this section we state and prove an asymptotic lower-bound estimate for general core-radius approach functionals (see Propositions \ref{mainthmcranew} and \ref{mainthmcranewdelta>ep}); such results will be instrumental for the proofs of the $\Gamma$-liminf inequalities in Theorems \ref{intro: mainthmgl}, \ref{mainthmcra}, \ref{cratbw} and \ref{gltbw}.

We introduce the increasing set-function $\E$ satisfying the assumptions (i) and (ii) in Section \ref{bc:section} as follows.
Let $\B=\{B_{r_1}(x_1), \ldots,B_{r_N}(x_n)\}$ be a finite family of open balls in $\R^2$ with $\bar B_{r_i}(x_i) \cap \bar B_{r_j}(x_j) = \emptyset$ for $i\neq j$, and
let  $\mu=\sum_{i=1}^n z_i\delta_{x_i}$ with $z_i\in\Z\setminus\{0\}$\,.

If $A_{r,R}(x)$ is an annulus  that does not intersect any $B_{r_i}(x_i)$, we set
\begin{equation}\label{preminf}
\mathfrak{G}(\B, \mu, A_{r,R}(x)) :=2\pi \alpha |\mu(B_r(x))|  \log\Big(\frac{R}{r}\Big)\,,
\end{equation}
with $\alpha$ as in assumption \eqref{Gprop}. For every open set $A\subset\R^2$ we set
\begin{equation}\label{minf}
\E(\B, \mu,A):= \sup \sum_j \mathfrak{G}(\B, \mu, A_j),
\end{equation}
where the supremum is taken over all finite families of disjoint annuli $A_j\subset A$ that do not intersect any $B_{r_i}(x_i)$.
Note that, if $A$ is an annulus that does not intersect any $B_{r_i}(x_i)$,  then $\E(\B, \mu,A)= \mathfrak{G}(B, \mu,A)$.
\begin{remark}\label{disuene}
{\rm
    The convenience of introducing $\E$ in \eqref{minf} to prove a lower-bound inequality for (an appropriate scaling of) the functional $\precrw$ in \eqref{intro:cra_en_0} will be clear in the following sections. However, the following simple observation already points in the right direction. Let $\Om(\B) = \Om\setminus \bigcup_{B\in\B} \overline B$, $w\in H^1 (\Om(\B);\mathcal S^1)$ and $\mu:=\sum_{B\in\Ccal} \deg(w,\partial B) \delta_{x_B}$ where $\Ccal$ denotes the family of balls in $\B$ that are contained in $\Om$, and  $x_B$ is the center of $B$. Then, by Jensen's inequality and by the lower bound in \eqref{Gprop}, we deduce that
\begin{equation}\label{impbound}
\E(\B, \mu, U)\le  \int_{U\cap\tilde\Om} \alpha|\nabla w|^2\ud x\le  \precrw(w;{U\cap\Om(\B)})
\end{equation}
for every open set $U\subset \Om$\,.
}
\end{remark}
%%%%%%%%%%%%%%%%%%%%%%%%%%%%%%%%%%%%%%%%%%%%%%
%%%%%%%%%%%%%%%%%%%%%%%%%%%%%%%%%%%%%%%%%%%%%%
%%%%%%%%%%%%%%%%%%%%%%%%%%%%%%%%%%%%%%%%%%%%%%
For every $\mu\in X(\Omega)$ and for every family of pairwise disjoint balls $\B$ such that $\supp\mu\subset\bigcup_{B\in\B}B$, we set
\begin{equation*}%\label{adm'}
\AF(\mu,\B):=\{w\in H^1(\Omega(\B);\mathcal S^1)\,:\,\deg(w,\partial B)=\mu(B)\quad\textrm{for every }B\in\B\}.
\end{equation*}
In addition we set
\begin{equation}\label{cra_new}
\precr(\mu,\B):=\inf_{w\in\AF(\mu,\B)}\precrw(w;\Omega(\B))\,,
\end{equation}
where $F_{\delta_\ep}$ is defined in \eqref{intro:cra_en_0} and $f$ satisfies \eqref{Pprop}, \eqref{Gprop}, \eqref{Hprop}\,.

Note that if $\mu=\sum_{i=1}^{n}z_i\delta_{x_i}\in X_{\ep}(\Omega)$ for some $\ep>0$,  setting $\B_\ep=\{B_{\ep}(x_i)\}_{i=1,\ldots,n}$, we have that $\AF(\mu,\B_\ep)$ coincides with the set $\AF_\ep(\mu)$ defined in \eqref{admst} and that $\precr(\mu,\B_\ep)=\precr(\mu)$\,.
%%%%%%%%%%%%%%%%%%%%%%%%%%%%%%%%%%%%%%%%%%%%%%
%%%%%%%%%%%%%%%%%%%%%%%%%%%%%%%%%%%%%%%%%%%%%%
%%%%%%%%%%%%%%%%%%%%%%%%%%%%%%%%%%%%%%%%%%%%%%
We are now in a position to state the first main result of this section, concerning the case $\delta_\ep\lesssim\ep$\,.
\begin{proposition}\label{mainthmcranew}
Let $\{\mu_\ep\}_{\ep}\subset X(\Omega)$ be such that
\begin{equation}\label{tvbound0}
|\mu_\ep|(\Omega)\le C|\log\ep|
\end{equation}
and $\mu_\ep\fla\mu$ for some $\mu\in X(\Omega)$\,. For every $\ep>0$ let $\B_\ep$ be a finite family of pairwise disjoint open balls such that $\supp\mu_\ep\subset \bigcup_{B\in\B_\ep} B$ and
\begin{equation}\label{radbound0}
\rad(\B_\ep)\le C\ep|\log\ep|.
\end{equation}
If $\limsup_{\ep\to 0}\frac{\delta_{\ep}}{\ep}<+\infty$, then
\begin{equation*}%\label{newcl}
\liminf_{\ep\to 0}\frac{1}{|\log\ep|}\precr(\mu_\ep,\B_\ep)\ge \Fzero(\mu)\,,
\end{equation*}
where $\Fzero$ is defined in \eqref{gali} and $\Psi(\cdot;T\!f_{\hom})$ is defined in \eqref{defPsi} for $h=T\!f_{\hom}$.
\end{proposition}
%%%%%%%%%%%%%%%%%%%%%%%%%%%%%%%%%%%%%%%%%%%%%%
%%%%%%%%%%%%%%%%%%%%%%%%%%%%%%%%%%%%%%%%%%%%%%
%%%%%%%%%%%%%%%%%%%%%%%%%%%%%%%%%%%%%%%%%%%%%%
\begin{proof}
For every $\ep>0$, let $w_\ep\in \AF(\mu_\ep,\B_\ep)$ be such that
\begin{equation}\label{infmin0}
\precrw(w_\ep; \Omega(\B_\ep))\le \precr(\mu_\ep,\B_\ep)+C
\end{equation}
for some constant $C$ independent of $\ep$. We can assume without loss of generality that
\begin{equation}\label{energybound}
\precrw(w_\ep; \Omega(\B_\ep))\le \precr(\mu_\ep,\B_\ep)+C \le C|\log\ep|\,.
\end{equation}
Moreover, by a standard localization argument in $\Gamma$-convergence, we can assume that $\mu=z_{0}\delta_{x_{0}}$ for some $z_{0}\in\Z\setminus\{0\}$ and $x_{0}\in\Omega$\,. In view of \eqref{energybound}, by exploiting assumption \eqref{Gprop} and by applying \eqref{impbound} with $U=\Omega$,
we have that
\begin{equation}\label{impbound2}
 \E(\B_\ep, \mu_\ep, \Omega)\le \alpha\int_{\Omega(\B_\ep)}|\nabla w_\ep|^2\ud x\le C|\log\ep|\,,
\end{equation}
where $\E$ is defined in \eqref{preminf}-\eqref{minf}\,.
For every $\ep>0$, let $\B_{\ep}(t)$ be a time-parametrized family of balls introduced as in Proposition \ref{ballconstr}, starting from $\B_\ep=:\B_\ep(0)$\,.
For every $t\ge 0$, we set $\sr_\ep(t):=\rad(\B_\ep(t))$\,, $\Ccal_\ep(t):= \{ B\in\B_\ep(t)\,: \, B\subset \Om\}$ and $U_\ep(t):=\bigcup_{B\in\B_\ep(t)}B$\,.
Moreover, for any $0<p<1$ we set
\begin{equation*}
t_\ep(p):=\frac{1}{\sr_\ep^{1-p}(0)} - 1 \quad\text{ and }\quad
\mu_\ep(p):= \sum_{B\in \Ccal_\ep(t_\ep(p))}  \mu_\ep(B) \delta_{x_B}\,.
\end{equation*}
By \eqref{impbound2}, applying \eqref{sti3} with $U=\Omega$, $t_1=0$ and $t_2=t_\ep(p)$\,, and using \eqref{radbound0},  we obtain
\begin{align*}
 C |\log \ep| &\ge \E(\B_\ep, \mu_\ep,\Om\cap (U_\ep(t_\ep(p))\setminus U_\ep(0)))\ge
2\pi\alpha\sum_{B\in \Ccal_\ep(t_\ep(p))}|\mu_\ep(B)|(1-p)|\log\sr_\ep(0)|\\
&=2\pi \alpha(1-p)|\mu_\ep(p)|(\Omega)|\log\sr_\ep(0)|
\ge C(1-p)|\mu_\ep(p)|(\Omega)|\log\ep|\,
\end{align*}
 for sufficiently small $\ep$. Therefore
\begin{equation}\label{tvb}
|\mu_\ep(p)|(\Omega)\le C_{p},
\end{equation}
for some constant $C_{p}>0$ depending on $p$ (and independent of $\ep$).
By Proposition \ref{ballconstr}(5) and \eqref{radbound0}, we have that
\begin{equation*}%\label{radiibd}
\sr_\ep(t_\ep(p))\le \sr^p_\ep(0)\le C \ep^p|\log\ep|^{p}\,,
\end{equation*}
whence, by applying Lemma \ref{bc_lemma} with $\nu_1=\mu(p)_\ep$ and $\nu_2=\mu_\ep$, we deduce that
\begin{equation*}
\|\mu_\ep-\mu_\ep(p)\|_{\flap}\le C\sr_\ep(t_\ep(p))(|\mu_\ep|+|\mu_\ep(p)|)(\Omega)\le C\ep^p|\log\ep|^{1+p}\to 0\quad\textrm{as }\ep\to 0\,.
\end{equation*}
Combining this relation with \eqref{tvb} and the fact that $\mu_\ep\fla\mu$ yields
\begin{equation}\label{flastar}
\mu_\ep(p)\weakstar\mu=z_{0}\delta_{x_{0}}\,,\qquad\textrm{for every }0<p<1\,.
\end{equation}
Let $c>1$ be such that $\log c<\frac p 2\frac{|\log\sr_\ep(0)|}{|\mu_\ep|(\Omega)+1}$. Note that, since $|\log\sr_\ep(0)|\geq C|\log\ep|$ and $|\mu_{\e}|(\Omega)\leq C|\log\e|$, we are allowed to take the constant $c$ in the previous inequality independent of $\e$.
By Lemma \ref{lemmaus} below (applied with $p_1=p$ and $p_2=\frac p 2$) there exist $t_\ep(p)\le \hat t_{\ep,1}<\hat t_{\ep,2}\le t_\ep({\frac p 2})$ with $(1+\hat t_{\ep,2})=c(1+\hat t_{\ep,1})$ such that no merging occurs in the interval $[\hat t_{\ep,1},\hat t_{\ep,2})$  and
\begin{equation}\label{mva}
\begin{aligned}
\int_{\Omega\cap (U(\hat t_{\ep,2})\setminus \overline{U}(\hat t_{\ep,1}))}|\nabla w_\ep|^2\ud x &\le \frac{\log c \int_{\Omega_\ep(\mu_{\e}')}|\nabla w_\ep|^2\ud x}{\frac p 2|\log\sr_\ep(0)|-\log c(|\mu_\ep|(\Omega)+1)}\\[1.5mm]
&\le\frac{\frac{\log c}{\alpha}\precrw(w_\ep; \Omega(\B_\ep))}{\frac p 2(|\log\ep|-\log|\log\ep|+C)-\log c\,(C|\log\ep|+1)}\le  C\,,
\end{aligned}
\end{equation}
where the last but one inequality follows from \eqref{radbound0} and \eqref{tvbound0}, whereas the last inequality is a consequence of \eqref{energybound}.
We classify the balls in $\Ccal_\ep(\hat t_{\ep,1})$ into two subclasses, namely
\begin{equation}\label{cccal}
\Ccal_\ep^{=0}(\hat t_{\ep,1}):=\left\{B\in\Ccal_\ep(\hat t_{\ep,1})\,:\,\mu_\ep(B)=0\right\}\text{ and }
\Ccal_\ep^{\neq 0}(\hat t_{\ep,1}):=\left\{B\in\Ccal_\ep(\hat t_{\ep,1})\,:\,\mu_\ep(B)\neq 0\right\}\,.
\end{equation}
We first consider the balls in $\Ccal_\ep^{=0}(\hat t_{\ep,1})$\,. For every such ball $B$ we let $\hat B$ denote the only ball in $\Ccal_\ep(\hat t_{\ep,2})$ containing $B$\,. Note that, the center $x_B$ of $B$ is the same as the center of $\hat B$\,. By \eqref{mva}, we have that
$$
\sum_{B\in \Ccal_\ep^{=0}(\hat t_{\ep,1})}\int_{\hat B\setminus B}|\nabla w_\ep|^2\ud x\le C\,.
$$
Now we extend the function $w_\ep$ to a function $\hat w_\ep\in H^1(\Omega_\ep(\mu_\ep)\cup \bigcup_{B\in  \Ccal_\ep^{=0}(\hat t_{\ep,1})}B;\mathcal S^1)$ in such a way that for every $B$ and $\hat B$ as above
\begin{equation}\label{extension}
\|\nabla \hat w_\ep\|_{L^2(\hat B;\R^{2\times 2})}\le \hat C\, \|\nabla w_\ep\|_{L^2(\hat B\setminus B;\R^{2\times 2})}\,,
\end{equation}
for some universal constant $\hat C$. We consider $B=B_{R}(\xi)$ and $\hat B=B_{cR}(\xi)$ two balls as above. Since $\deg(w_\ep,\partial  B_R(\xi))=\deg(w_\ep, \partial B_{cR}(\xi))=0$, by arguing as in \cite{BZ} (see also \cite {BB}), one can show that there exists a lifting $u_\ep^{A_{R,cR}(\xi)}\in H^1(A_{R,cR}(\xi))$ of $w_\ep$ in $A_{R,cR}(\xi)$. Let $U_\ep:A_{\frac{R}{c},R}(\xi)\to \R$ be the extension by reflection of the function $u_\ep^{A_{R,cR}(\xi)}$ to the annulus  $A_{\frac{R}{c},R}(\xi)$, i.e., $U_\ep(x):=u_\ep^{A_{R,cR}(\xi)}(\xi-c(x-\xi)+(1+c)R\frac{x-\xi}{|x-\xi|})$. We let $\bar U_\ep$ denote the average of $U_\ep$ on $A_{\frac{R}{c},R}(\xi)$. Let $\eta:[\frac{R}{c}, R]\to\R$ be the cut-off function defined by $\eta(\rho)=\frac{c\rho-R}{R(c-1)}$\,. We define the function $\hat u_\ep^{\hat B}:B_{cR}(\xi)\to \R$ as
$$
\hat u_\ep^{\hat B}(x):=\left\{\begin{array}{ll}
u_\ep^{A_{R,cR}(\xi)}&\textrm{if }x\in A_{R,cR}(\xi),\\
\eta(|x|)U_\ep(x)+(1-\eta(|x|))\bar U_\ep&\textrm{if }x\in A_{\frac{R}{c},R}(\xi),\\
\bar U_\ep&\textrm{if }x\in B_{\frac R c}(\xi).
\end{array}\right.
$$
By the Poincar\'e-Wirtinger inequality and by the very definition of $U_\ep$ we have that there exists a constant $\hat C$ (independent of $\ep$) such that
\begin{equation*}%\label{countexte}
\begin{aligned}
\int_{A_{\frac{R}{c},R}(\xi)}|\nabla \hat u_\ep^{\hat B}|^2\ud x
=&\int_{A_{\frac{R}{c},R}(\xi)}|\nabla (\eta(|x|)(U_\ep(x)-\bar U_\ep))|^2\ud x\\
\le\,& 2\,\frac{c^2}{R^2(c-1)^2}\int_{A_{\frac{R}{c},R}(\xi)}|U_\ep(x)-\bar U_\ep(x)|^2\ud x+2\int_{A_{\frac{R}{c},R}(\xi)}|\nabla U_\ep|^2\ud x\\
\le\,& C\,\int_{A_{\frac{R}{c},R}(\xi)}|\nabla U_\ep|^2\ud x\le \hat C\int_{A_{\frac{R}{c},R}(\xi)}|\nabla u_\ep^{A_{R,cR}(\xi)}|^2\ud x\,.
\end{aligned}
\end{equation*}
Therefore, setting
\begin{equation*}
\hat w_\ep(x):=\left\{\begin{array}{ll}
e^{\iota \hat u_\ep^{\hat B}(x)}&\textrm{if }x\in \hat B\textrm{ for some }\hat B\in \Ccal^{= 0}_{\ep}(\hat t_{\ep,1})\,,\\
w_\ep(x)&\textrm{elsewhere in }\Omega(\B_\ep)\,,
\end{array}
\right.
\end{equation*}
we have that $\hat w_\ep\in H^1(\Omega(\B_\ep)\cup \bigcup_{B\in  \Ccal_\ep^{=0}(\hat t_{\ep,1})}B;\mathcal S^1)$ and satisfies \eqref{extension}.
Then from \eqref{mva} and \eqref{extension} we deduce that
\begin{equation}\label{tozero}
\sum_{B\in \Ccal^{=0}_{\ep}(\hat t_{\ep,1})}\int_{B}|\nabla \hat w_\ep|^2\ud x\le C\,.
\end{equation}

We now focus on the balls in $\Ccal^{\neq 0}_{\ep}(\hat t_{\ep,1})$\,. We set
$\mu(\hat t_{\ep,1}):=\sum_{B\in\Ccal^{\neq 0}_{\ep}(\hat t_{\ep,1})}\mu_\ep(B)\delta_{x_B}$\,. In view of the ball construction in Section \ref{bc:section} and of \eqref{tvb}, we have that $\sharp \Ccal_\ep^{\neq 0}(\hat t_{\ep,1})\le C_{p}$. Therefore, up to extracting a subsequence we may assume that $\sharp \Ccal_\ep^{\neq 0}(\hat t_{\ep,1})=L$ for every $\ep>0$ and for some $L\in\N$\,.
For every $l=1,\ldots, L$, let $x_\ep^l$ be the center of the $l$-th ball $B_\ep^l$ in  $\Ccal_\ep^{\neq 0}(\hat t_{\ep,1})$\,.
Up to a further subsequence, we can assume that the points $x^l_\ep$ converge to some points in the finite set $\{\xi_{0}=x_{0},\xi_{1},\ldots,\xi_{{L'}}\}\subset\bar\Omega$, where $L'\le L$\,. Let $\rho>0$ be such that $B_{2\rho}(x^0)\sub\sub\Omega$ and $B_{2\rho}(\xi_{j})\cap B_{2\rho}(\xi_{k})=\emptyset$ for all $j\neq k$\,. Then $x_\ep^l\in B_{\rho}(\xi_{j})$ for some $j=1,\ldots,L'$ and for $\ep$ small enough.
We set
$$
\tilde\mu_\ep:=\sum_{x_{\ep}^l\in B_\rho(x_{0})}\mu_\ep(B^l_\ep)\delta_{x_\ep^l}\,.
$$
By construction,  we have that
\begin{equation}\label{befdsum}
|\tilde\mu_\ep|(\Omega)\le |\mu_\ep(p)|(\Omega)\quad\text{and }\quad\|\tilde\mu_\ep-\mu_\ep(p)\|_{\flap}\to 0\,,
\end{equation}
which, in view of \eqref{tvb} and \eqref{flastar}, implies that, up to a subsequence,
%\begin{equation*}%\label{dsum}
$\tilde\mu_\ep\weakstar\mu=z_{0}\delta_{x_{0}}$\,.
%\end{equation*}
Therefore,  for sufficiently small $\ep$,
\begin{equation}\label{dsum2}
\tilde\mu_\ep(B_{2\rho}(x_{0}))=\sum_{x_{\ep}^l\in B_\rho(x_{0})}\mu_\ep(B^l_\ep)=z_{0}\,.
\end{equation}
Thanks to \eqref{tozero} and the assumption \eqref{Gprop}, we have that
\begin{equation}\label{almostfin}
\begin{aligned}
\precrw(w_\ep;\Omega(\B_\ep))&\ge\int_{\Omega(\B_\ep)\setminus \cup_{B\in\Ccal^{=0}_{\ep}(\hat t_{\ep,1}) }B}f\Big(\frac{x}{\delta_\ep},\nabla w_\ep\Big)\ud x\\
&\ge \int_{\Omega(\B_\ep)\cap B_{2\rho}(x_{0})}f\Big(\frac{x}{\delta_\ep},\nabla \hat w_\ep\Big)\ud x-C\,.
\end{aligned}
\end{equation}
It remains to prove the lower bound for the right-hand side of $\eqref{almostfin}$\,. To this end, we take $0<p'<p$ such that
$\sr_\ep(\hat t_{\ep,1})\le \ep^{p'}$ (note that such $p'$ always exists since, by Lemma \ref{lemmaus}, $\sr_\ep (\hat t_{\ep,1})< \sr_\ep^{\frac p 2}(0)\le C\ep^{\frac p 2}|\log\ep|^{\frac p 2}$),  choose $0<\bar p<p'$\,  and let $g_\ep:[\bar p,p']\to\{1,\ldots,L\}$ denote the function which associates to any $q\in [\bar p,p']$ the number $g_\ep(q)$ of connected components of the set $\bigcup_{l=1}^{L}B_{\ep^q}(x_\ep^l)$\,. For every $\ep>0$\,, the function $g_\ep$ is monotonically non decreasing so that it can have at most $\hat L\leq L$ discontinuities. Let $q_\ep^j$, for $j=1,\ldots,\hat L$, denote the discontinuity points of $g_{\e}$ and assume that
$$
\bar p\le q_\ep^1<\ldots<q_\ep^{\hat L}\le p'\,.
$$
There exists a finite set $\triangle=\{q^1,q^{2},\ldots,q ^{\tilde L}\}$ with $q^i<q^{i+1}$ and $\tilde L\le \hat L$ such that, up to a subsequence, $\{q_\ep^j\}_{\ep}$ converges to some point in $\triangle$\,, as $\ep\to 0$ for every $j=1,\ldots,\hat L$\,. Without loss of generality we may assume that $q^1=\bar p$\,, and that $q^{\tilde L}=p'$\,. Let $\lambda>0$ be such that $4\lambda<\min \{q^{i+1}-q^{i}:\, i\in\{1,2,\dots,\tilde L\}\}$ and let $\ep$ be so small that for every $j=1,\ldots, \hat L$\,, $|q_\ep^j-q^i|<\lambda$ for some $q^i\in\triangle$. Then the function $g_\ep$ is constant in the interval $[q^i+\lambda, q^{i+1}-\lambda]$\,, its value being denoted by $M_\ep^i$. For every $i=1,\ldots,\tilde L-1$ we construct a family of $M_\ep^i\le \tilde L-1$ annuli that we let $C_\ep^{i,m}:=B_{\ep^{q^i+\lambda}}(y_\ep^m)\setminus \overline{B}_{\ep^{q^{i+1}-\lambda}}(y_\ep^m)$ with $y_\ep^m\in B_{\rho}(x_{0})$ and  $m=1,\ldots,M_\ep^i$. The annuli $C_\ep^{i,m}$ can be taken pairwise disjoint for all $i$ and $m$ and such that
$$
\bigcup_{x_\ep^l\in B_{\rho}(x_{0})} B_\ep^l\subset\bigcup_{m=1}^{M_\ep^i}B_{\ep^{q^{i+1}-\lambda}}(y_\ep^m)
$$
for all $i=1,\ldots,\tilde L-1$\,. Note that, for $\ep$ small enough, $C_\ep^{i,m}\subset B_{2\rho}(x_{0})$ for all $i$ and $m$\,. By \eqref{befdsum} we have that $|\mu_\ep (B_{\ep^{q^{i+1}-\lambda}}(y_\ep^m))|\le C$ for every $i=1,\ldots,\tilde L-1$ and $m=1,\ldots, M_\ep^i$\,. Therefore, up to passing to a further subsequence, we can assume that $M_\ep^i=M^i$ and that $\mu_\ep (B_{\ep^{q^{i+1}-\lambda}}(y_\ep^m))=z_{i,m}\in\Z\setminus\{0\}$\,, with $M^i$ and $z_{i,m}$ independent of $\ep$\,. Finally, in view of \eqref{dsum2}, we have that
\begin{equation}\label{dsum4}
\sum_{m=1}^{M^i} z_{i,m}= z_{0}\,.
\end{equation}
Observe that the assumption $\limsup_{\ep\to 0}\frac{\delta_{\ep}}{\ep}<+\infty$ implies the inequality $\lim_{\ep\to 0}\frac{\delta_{\ep}}{\ep^{q^{i+1}-\lambda}}=0$ for every $i$. Hence, we can apply Proposition \ref{homdeg}  %and in particular \eqref{lbhom}
with $s_{1}=q^{i}+\lambda<q^{i+1}-\lambda=s_2$ (see also Remark \ref{homdegrmk})
 to get that for every $i$ and $m$ there exists a modulus of continuity $\omega$ such that
\begin{equation*}
\frac{1}{|\log\ep|}\int_{C_\ep^{i,m}}f\Big(\frac{x}{\delta_\ep},\nabla \hat w_\ep\Big)\ud x\ge (q^{i+1}-q^i-2\lambda)\psi(z_{i,m};T\!f_{\hom})-\omega(\ep)\,.
\end{equation*}
Summing the previous inequality over $m$ and $i$ and using \eqref{almostfin} yields
\begin{equation}\label{final2}
\begin{aligned}
\frac 1 {|\log\ep|}  {\precrw(w_\ep;\Omega(\B_\ep))} &\ge \sum_{i=1}^{\tilde L-1}\sum_{m=1}^{M^i} (q^{i+1}-q^i-2\lambda)\psi (z_{i,m};T\!f_{\hom})-\omega(\ep)\\
&\ge \sum_{i=1}^{\tilde L-1} (q^{i+1}-q^i-2\lambda) \Psi(z_{0};T\!f_{\hom})-\omega(\ep)\\
&= (p'-\bar p- 2(\tilde L-1)\lambda)\Psi(z_{0};T\!f_{\hom})-\omega(\ep)\,,
\end{aligned}
\end{equation}
where the second inequality follows from \eqref{dsum4} and from the very definition of $\Psi$ in \eqref{defPsi}.
Then, the claim follows by \eqref{final2} taking the limits as $\ep\to 0$\,, $\lambda\to 0$\,, $\bar p\to 0$\,, and $p,p'\to 1$ and using \eqref{infmin0}.
\end{proof}
%%%%%%%%%%%%%%%%%%%%%%%%%%%%%%%%%%%%%%%
%%%%%%%%%%%%%%%%%%%%%%%%%%%%%%%%%%%%%%%
%%%%%%%%%%%%%%%%%%%%%%%%%%%%%%%%%%%%%%%
We turn to the technical lemma that has been exploited in the proof of Proposition \ref{mainthmcranew} above (see formula \eqref{mva}).
For every $p\in (0,1)$ let $t(p):=\frac{1}{\rad^{1-p}(\B)}-1$\,.
%in order to construct a finite energy extension of the order parameter in zero average clusters.
\begin{lemma}\label{lemmaus}
Let $\mu\in X(\Omega)$ and  let $\B$ be a finite family of pairwise disjoint open balls such that $\supp\mu\subset \bigcup_{B\in\B} B$ and $\rad(\B)<1$.
Assume that  $0<p_2<p_1<1$, and  $c>1$ be such that $\log c<(p_1-p_2)\frac{|\log\rad(\B)|}{|\mu|(\Omega)+1}$. Assume also that $\B(t)$ is a time parametrized family of balls constructed as in Proposition \ref{ballconstr}, starting from $\B=:\B(0)$, and $U(t):=\bigcup_{B\in\B(t)}B$ for every $t\ge 0$\,.
Then, there exist $\hat t_{1}, \hat t_{2}\in [t(p_1), t( p_2))$ with $1+\hat t_{2}=c(1+\hat t_{1})$ such that no merging occurs in the interval $[ \hat t_{1}, \hat t_{2})$ and
\begin{equation*}%\label{formuse}
\int_{\Omega\cap (U(\hat t_{2})\setminus \overline{U}(\hat t_{1}))}|\nabla w|^2\ud x\le \frac{\log c \int_{\Omega(\B)}|\nabla w|^2\ud x}{(p_1- p_2)|\log\rad(\B)|-\log c(|\mu|(\Omega)+1)}\,,
%\quad\textrm{for every $w\in\AF'(\mu,\B)$\,,}
\end{equation*}
for every $w\in\AF(\mu,\B)$.
\end{lemma}
%%%%%%%%%%%%%%%%%%%%%%%%%%%%%%%%%%%%%%%
%%%%%%%%%%%%%%%%%%%%%%%%%%%%%%%%%%%%%%%
%%%%%%%%%%%%%%%%%%%%%%%%%%%%%%%%%%%%%%%
\begin{proof}
We set $J:=\lfloor (p_1-p_2)\frac{|\log\rad(\B)|}{\log c}\rfloor$ and $\hat t^j:=c^j(1+t(p_1))-1$ for every $j=0,1,\ldots,J$.
Note that  $\frac{1+\hat t^{j+1}}{1+\hat t^j}=c$ for every $j=0,1,\ldots, J-1$\,.
We let $\mathcal{J}$ denote the set of indices in $\{0,1,\ldots,J-1\}$ for which no merging occurs in the interval $[\hat t^j,\hat t^{j+1})$\,. Since the number of merging times is bounded from above by $|\mu|(\Omega)$, we have
 $\sharp \mathcal J\ge J-|\mu|(\Omega)\ge (p_1-p_2)\frac{|\log\rad(\B)|}{\log c}-|\mu|(\Omega)-1$\,.
 For every $j\in\mathcal J$ and for every ball $B(\hat t^j)\in \B(\hat t^j)$,  $B(\hat t^{j+1})$ denotes the unique ball in $\B(\hat t^{j+1})$ such that $B(\hat t^j)\subset B(\hat t^{j+1})$.
By the mean-value theorem, there exists $k\in\mathcal J$ such that
 \begin{eqnarray*}
 \int_{\Omega(\B)}|\nabla w|^2\ud x&\ge& \sum_{j\in\mathcal J}\sum_{B(\hat t^j)\in \B(\hat t^j)}\int_{\Omega\cap(B(\hat t^{j+1})\setminus \bar B(\hat t^{j}))}|\nabla w|^2\ud x\\
 &\ge &\sharp \mathcal J \int_{\Omega\cap(U(\hat t^{k+1})\setminus \bar U(\hat t^{k}))}|\nabla w_\ep|^2\ud x\\
 &\ge & \Big((p_1-p_2)\frac{|\log\rad(\B)|}{\log c}-|\mu|(\Omega)-1\Big)\int_{\Omega\cap(U(\hat t^{k+1})\setminus \bar U(\hat t^{k}))}|\nabla w|^2\ud x\,,
 \end{eqnarray*}
 from which the claim follows setting $\hat t_{1}:=\hat t^{k}$ and $\hat t_{2}:=\hat t^{k+1}$\,.
\end{proof}
%%%%%%%%%%%%%%%%%%%%%%%%%%%%%%%%%%%%%%%
%%%%%%%%%%%%%%%%%%%%%%%%%%%%%%%%%%%%%%%
%%%%%%%%%%%%%%%%%%%%%%%%%%%%%%%%%%%%%%%
As for the case $\delta_\ep\gg\ep$\,, we restrict our analysis to functionals of the form \eqref{aform}. In such a case the main result is the following.
%%%%%%%%%%%%%%%%%%%%%%%%%%%%%%%%%%%%%%%
%%%%%%%%%%%%%%%%%%%%%%%%%%%%%%%%%%%%%%%
%%%%%%%%%%%%%%%%%%%%%%%%%%%%%%%%%%%%%%%
\begin{proposition}\label{mainthmcranewdelta>ep}
Let $\precrw,\,\precr$ be defined in \eqref{aform}, \eqref{cra_new}, respectively, where $a$ is a measurable $(0,1)^2$-periodic function satisfying $a(x)\in [\alpha,\beta]\subset (0,+\infty)$ for a.e. $x\in\R^2$\,.
Let $\{\mu_\ep\}_{\ep}\subset X(\Omega)$ be such that
\begin{equation}\label{tvbound02}
|\mu_\ep|(\Omega)\le C|\log\ep|
\end{equation}
and $\mu_\ep\fla\mu$ for some $\mu\in X(\Omega)$\,. For every $\ep>0$ let $\B_\ep$ be a finite family of pairwise disjoint open balls such that $\supp\mu_\ep\subset \bigcup_{B\in\B_\ep} B$ and
\begin{equation}\label{radbound02}
\rad(\B_\ep)\le C\ep|\log\ep|.
\end{equation}
If \eqref{limlog} is satisfied, then
\begin{equation*}%\label{newcl}
\liminf_{\ep\to 0}\frac{1}{|\log\ep|}\precr(\mu_\ep,\B_\ep)\ge2\pi\bigg(\Big(1-\lambda)\mathrm{ess}\inf a+ \lambda\sqrt{\det\ho}\bigg)|\mu|(\Omega)\,,
\end{equation*}
where $\ho$ is defined in \eqref{defahom}.
\end{proposition}
%%%%%%%%%%%%%%%%%%%%%%%%%%%%
%%%%%%%%%%%%%%%%%%%%%%%%%%%%
%%%%%%%%%%%%%%%%%%%%%%%%%%%%
\begin{proof}
The proof closely resembles the one of Proposition \ref{mainthmcranew}; here we only highlight the main changes that are needed to prove the different lower bound in the regime \eqref{limlog0}.

Let $w_\ep\in \AF_\ep(\mu_\ep,\Omega(\B_\ep))$ be such that
\begin{equation}\label{infmin0000}
\precrw(w_\ep; \Omega(\B_\ep))\le \precr(\mu_\ep,\B_\ep)+C
\end{equation}
for some constant $C$ independent of $\ep$.
%We can assume without loss of generality that
%\begin{equation}\label{energybound4}
%%\alpha\int_{\Omega_\ep(\mu_\ep)}|\nabla w_\ep|^2\ud x\le
%\precrw(w_\ep; \Omega(\B_\ep))\le \precr(\mu_\ep,\B_\ep)+C \le C|\log\ep|\,.
%\end{equation}
By a standard localization argument in $\Gamma$-convergence, we can assume that $\mu=z_{0}\delta_{x_{0}}$ for some $z_{0}\in\Z\setminus\{0\}$ and $x_{0}\in\Omega$\,.

For every $\ep>0$, let $\B_{\ep}(t)$ be a time-parametrized family of balls introduced as in Proposition \ref{ballconstr}, starting from $\B_\ep=:\B_\ep(0)$\,.
For every $t\ge 0$, we set $\sr_\ep(t):=\rad(\B_\ep(t))$\,, $\Ccal_\ep(t):= \{ B\in\B_\ep(t)\,: \, B\subset \Om\}$ and $U_\ep(t):=\bigcup_{B\in\B_\ep(t)}B$\,.
For every $0<p<1$ we set
\begin{equation*}
t_\ep(p):=\frac{1}{\sr_\ep^{1-p}(0)} - 1 \quad\text{ and }\quad
\mu_\ep(p):= \sum_{B\in \Ccal_\ep(t_\ep(p))}  \mu_\ep(B) \delta_{x_B}.
\end{equation*}
Fix $\lambda<p<1$\,. By arguing as in the proof of \eqref{tvb} and \eqref{flastar}, we have that
\begin{equation}\label{tvb2}
|\mu_\ep(p)|(\Omega)\le C_{p}\quad\textrm{ and }\quad \mu_{\ep}(p)\weakstar\mu=z_0\delta_{x_0}\quad\textrm{as }\ep\to 0\,.
\end{equation}
Following the reasoning in the  proof of Proposition \ref{mainthmcranew} we have that for every $0<\eta<p-\lambda$ there exists $t_\ep(p)\le \hat t_{\ep,1}\le t_\ep(p-\eta)$ and a  a map $\hat w_\ep\in H^1(\Omega(\B_\ep)\cup\bigcup_{B\in\Ccal_{\ep}^{=0}(\hat t_{\ep,1})};\mathcal S^1)$ (with $\Ccal_{\ep}^{=0}(\hat t_{\ep,1})$ defined in \eqref{cccal}) satisfying \eqref{tozero}.

%Let $\Ccal_{\delta_\ep}^{\neq 0}$ be a family of balls of radius $\delta_\ep$ such that
%$$
%\bigcup_{B\in \Ccal_{\ep}^{\neq 0}(\hat t_{\ep,1})}B\subset\bigcup_{B\in \Ccal^{\neq 0}_{\delta_\ep}}B\,,
%$$
%with  $\Ccal^{\neq 0}_{\ep}(\hat t_{\ep,1})$ defined in \eqref{cccal}. Furthermore, we set
%$$
%\mu_{\delta_\ep}:=\sum_{B\in\Ccal_{\delta_\ep}^{\neq 0}}\mu_\ep(B)\bf{\delta}_{x_B}\,.
%$$
%In view of the ball construction in Section \ref{bc:section} and of \eqref{tvb2}, we have that $\sharp\Ccal_{\delta_\ep}^{\neq 0}\le \sharp \Ccal_\ep^{\neq 0}(\hat t_{\ep,1})\le C_{p}$.
% Therefore, up to extracting a subsequence we may assume that $\sharp \Ccal_{\delta_\ep}^{\neq 0}=L$ for every $\ep>0$ and for some $L\in\N$\,.
%For every $l=1,\ldots, L$, let $x_\ep^l$ be the center of the $l$-th ball in  $\Ccal_{\delta_\ep}^{\neq 0}(\hat t_{\ep,1})$\,.
%Up to a further subsequence, we can assume that the points $x^l_\ep$ converge to some points in the finite set $\{x_{0}=\xi_0,\xi_{1},\ldots,\xi_{{L'}}\}\subset\bar\Omega$, where $L'\le L$\,. Let $\rho>0$ be such that $B_{2\rho}(x^0)\sub\sub\Omega$ and $B_{2\rho}(\xi_{j})\cap B_{2\rho}(\xi_{k})=\emptyset$ for all $j\neq k$\,. Then $x_\ep^l\in B_{\rho}(\xi_{j})$ for some $j=1,\ldots,L'$ and for $\ep$ small enough.
Recalling the definition of $\Ccal^{\neq 0}_{\ep}(\hat t_{\ep,1})$ in \eqref{cccal}, we set
$\mu(\hat t_{\ep,1}):=\sum_{B\in\Ccal^{\neq 0}_{\ep}(\hat t_{\ep,1})}\mu_\ep(B)\delta_{x_B}$\,. In view of the ball construction in Section \ref{bc:section} and of \eqref{tvb2}, we have that $\sharp \Ccal_\ep^{\neq 0}(\hat t_{\ep,1})\le C_{p}$.
Therefore, up to extracting a subsequence we may assume that $\sharp \Ccal_\ep^{\neq 0}(\hat t_{\ep,1})=L$ for every $\ep>0$ and for some $L\in\N$\,.
For every $l=1,\ldots, L$, let $x_\ep^l$ be the center of the $l$-th ball $B_\ep^l$ in  $\Ccal_\ep^{\neq 0}(\hat t_{\ep,1})$\,.
Up to a further subsequence, we can assume that the points $x^l_\ep$ converge to some points in the finite set $\{x_{0}=\xi_0,\xi_{1},\ldots,\xi_{{L'}}\}\subset\bar\Omega$, where $L'\le L$\,. Let $\rho>0$ be such that $B_{2\rho}(x^0)\sub\sub\Omega$ and $B_{2\rho}(\xi_{j})\cap B_{2\rho}(\xi_{k})=\emptyset$ for all $j\neq k$\,. Then $x_\ep^l\in B_{\rho}(\xi_{j})$ for some $j=1,\ldots,L'$ and for $\ep$ small enough.
Setting
$$
\tilde\mu_\ep:=\sum_{x_{\ep}^l\in B_\rho(x_{0})}\mu_\ep(B^l_\ep)\delta_{x_\ep^l}\,,
$$
by construction,  we have that
\begin{equation}\label{befdsum2}
|\tilde\mu_\ep|(\Omega)\le |\mu_\ep(p)|(\Omega)\quad\text{and }\quad\|\tilde\mu_\ep-\mu_\ep(p)\|_{\flap}\to 0\,,
\end{equation}
which, in view of \eqref{tvb2}, implies that, up to a subsequence,
\begin{equation*}%\label{dsum}
\tilde\mu_\ep\weakstar\mu=z_{0}\delta_{x_{0}}\,.
\end{equation*}
Therefore,  for sufficiently small $\ep$,
\begin{equation}\label{dsum22}
\tilde\mu_\ep(B_{2\rho}(x_{0}))=\sum_{x_{\ep}^l\in B_\rho(x_{0})}\mu_\ep(B^l_\ep)=z_{0}\,,
\end{equation}
and, by arguing as in the proof of \eqref{almostfin}, we obtain
\begin{equation}\label{almostfindopo}
\begin{aligned}
\precrw(w_\ep;\Omega(\B_\ep))\ge \int_{\Omega(\B_\ep)\cap B_{2\rho}(x_{0})}a\Big(\frac{x}{\delta_\ep}\Big)|\nabla \hat w_\ep|^2\ud x-C\,.
\end{aligned}
\end{equation}
It remains to prove the lower bound for the right-hand side of $\eqref{almostfindopo}$\,. To this end, we take $\lambda<p'<p$ such that
$\rad(\hat t_{\ep,1})\le \ep^{p'}$ (note that such $p'$ always exists since $\sr_\ep(\hat t_{\ep,1})< \sr_\ep^{p-\eta}(0)\le C\ep^{ p -\eta}|\log\ep|^{p -\eta}$ for every $\eta<p-\lambda$), choose $0<\bar p<\lambda<p'$\,  and  let $g_\ep:[\bar p,p']\to\{1,\ldots,L\}$ denote the function which associates to any $q\in [\bar p,p']$ the number $g_\ep(q)$ of connected components of the set $\bigcup_{l=1}^{L}B_{\ep^q}(x_\ep^l)$\,. For every $\ep>0$\,, the function $g_\ep$ is monotonically non decreasing so that it can have at most $\hat L\leq L$ discontinuities.
Let $q_\ep^j$ for $j=1,\ldots,\hat L_1$ and   $\kappa_\ep^j$ for $j=1,\ldots,\hat L_2$ denote the discontinuity points of $g_{\e}$ in $[\bar p, \lambda]$ and in in $[\lambda,p']$, respectively. Assume that
$$
\bar p\le q_\ep^1<\ldots<q_\ep^{\hat L_1}\le \lambda\le \kappa_\ep^1\le\ldots\le \kappa_\ep^{\hat L_2}\le p'\,.
$$
%$$
%\bar p\le q_\ep^1<\ldots<q_\ep^{\hat L_1} \le \lambda\le s_\ep^1<\ldots<s_\ep^{\hat L_2}\le p'\,.
%$$
There exists a finite set $\triangle_1=\{q^1,\ldots,q ^{\tilde L_1}\}$ (resp., $\triangle_2=\{\kappa^1,\ldots,\kappa ^{\tilde L_2}\}$)
with $q^i<q^{i+1}$ (resp., $\kappa^i<\kappa^{i+1}$), and $\tilde L_1\le \hat L_1$ (resp., $\tilde L_2\le \hat L_2$) such that, up to a subsequence, $\{q_\ep^j\}_{\ep}$ converges to some point in $\triangle_1$\,, as $\ep\to 0$ for every $j=1,\ldots,\hat L_1$ (resp., $\{\kappa_\ep^j\}_{\ep}$ converges to some point in $\triangle_2$\,, as $\ep\to 0$ for every $j=1,\ldots,\hat L_2$). Without loss of generality we may assume that $q^1=\bar p$\,, $q^{\tilde L_1}=\lambda=\kappa^1$\,, and $\kappa^{\tilde L_2}=p'$\,.
Let $\eta>0$ be such that $4\eta<\min \{q^{i+1}-q^{i}:\, i\in\{1,\dots,\tilde L_1\}\}$
and  $4\eta<\min \{\kappa^{i+1}-\kappa^{i}:\, i\in\{1,\dots,\tilde L_2\}\}$
and let $\ep$ be so small that for every $j=1,\ldots, \hat L_1$\,, $|q_\ep^j-q^i|<\eta$ for some $q^i\in\triangle_1$ and for every $j=1,\ldots, \hat L_2$\,, $|\kappa_\ep^j-\kappa^i|<\eta$ for some $s^i\in\triangle_2$\,. Then the function $g_\ep$ is constant in the intervals $[q^i+\eta, q^{i+1}-\eta]$ and in the intervals $[\kappa^i+\eta, \kappa^{i+1}-\eta]$ , and in both cases we let its value be denoted by $M_\ep^i$. For every $i=1,\ldots,\tilde L_1-1$ (resp.,  $i=1,\ldots,\tilde L_2-1$) we construct a family of $M_\ep^i\le \tilde L_1-1$ (resp., $M_\ep^i\le \tilde L_2-1$) annuli that are denoted by $C_\ep^{i,m}:=B_{\ep^{q^i+\eta}}(y_\ep^m)\setminus \overline{B}_{\ep^{q^{i+1}-\eta}}(y_\ep^m)$
(resp., $C_\ep^{i,m}:=B_{\ep^{\kappa^i+\eta}}(y_\ep^m)\setminus \overline{B}_{\ep^{\kappa^{i+1}-\eta}}(y_\ep^m)$)
with $y_\ep^m\in B_{\rho}(x_{0})$ and  $m=1,\ldots,M_\ep^i$. The annuli $C_\ep^{i,m}$ can be taken pairwise disjoint for all $i$ and $m$ and such that
$$
\bigcup_{x_\ep^l\in B_{\rho}(x_{0})} B_\ep^l\subset\bigcup_{m=1}^{M_\ep^i}B_{\ep^{q^{i+1}-\eta}}(y_\ep^m)\qquad\textrm{and}\qquad \bigcup_{x_\ep^l\in B_{\rho}(x_{0})} B_\ep^l\subset\bigcup_{m=1}^{M_\ep^i}B_{\ep^{\kappa^{i+1}-\eta}}(y_\ep^m)\,.
$$
Note that, for $\ep$ small enough, $C_\ep^{i,m}\subset B_{2\rho}(x_{0})$ for all $i$ and $m$ and, by \eqref{befdsum2}\,, we get
\begin{equation*}
|\mu_\ep (B_{\ep^{q^{i+1}-\eta}}(y_\ep^m))|\le C\qquad\textrm{and}\qquad|\mu_\ep (B_{\ep^{\kappa^{i+1}-\eta}}(y_\ep^m))|\le C\,.
\end{equation*}
%for every $i=1,\ldots,\tilde L-1$ and $m=1,\ldots, M_\ep^i$\,.
Therefore, up to passing to a further subsequence, we can assume that $M_\ep^i=M^i$ and that $\mu_\ep (B_{\ep^{q^{i+1}-\eta}}(y_\ep^m))=z_{i,m}\in\Z\setminus\{0\}$ and $\mu_\ep (B_{\ep^{\kappa^{i+1}-\eta}}(y_\ep^m))=z_{i,m}\in\Z\setminus\{0\}$\,, with $M^i$ and $z_{i,m}$ independent of $\ep$\,. Finally, in view of \eqref{dsum22}, we have that
\begin{equation}\label{dsum42}
\sum_{m=1}^{M^i} z_{i,m}= z_{0}\,.
\end{equation}
For every $i=1,\ldots,\tilde L_2-1$ and for every $m=1,\ldots,M^i$ we have that
\begin{equation*}%\label{nuo1}
\begin{aligned}
\int_{C_\ep^{i,m}}a\Big(\frac{x}{\delta_\ep}\Big)|\nabla \hat w_\ep|^2\ud x\ge&\ \mathrm{ess}\inf a\int_{C_\ep^{i,m}}\frac {1}{|x-y_\ep^m|^2}\ud x\\
\ge &\  2\pi \mathrm{ess}\inf a \,(\kappa^{i+1}-\kappa^i-2\eta)|\log\ep||z_{i,m}|^2\\
\ge &\  2\pi \mathrm{ess}\inf a\, (\kappa^{i+1}-\kappa^i-2\eta)|\log\ep||z_{i,m}|
\end{aligned}
\end{equation*}
which, summing over $m$ and over $i$, dividing by $|\log\ep|$ and using \eqref{dsum42}, yields
\begin{equation}\label{nuo1}
\begin{aligned}
\frac{1}{|\log\ep|}\sum_{i=1}^{\tilde L_2-1}\sum_{m=1}^{M^i}\int_{C_\ep^{i,m}}a\Big(\frac{x}{\delta_\ep}\Big)|\nabla \hat w_\ep|^2\ud x\ge&\  2\pi \,\mathrm{ess}\inf a\sum_{i=1}^{\tilde L_2-1}(\kappa^{i+1}-\kappa^i-2\eta)|z_0|\\
\ge&\ 2\pi\, \mathrm{ess}\inf a\,(\kappa^{\tilde L_2}-\kappa^{1}-2\eta\tilde L_2)|z_0|\\
=&\  2\pi \,\mathrm{ess}\inf a\,(p'-\lambda-2\eta\tilde L_2)|z_0|\,.
\end{aligned}
\end{equation}
Moreover, since $q^i+\eta<q^{\tilde L_1-1}+\eta<\lambda$ for every $i=1,\ldots,\tilde L_0-1$\,, by \eqref{limlog0}, we have that $\lim_{\ep\to 0}\frac{\delta_\ep}{\ep^{q^i+\eta}}=0$ for every $i=1,\ldots,\tilde L_1-1$\,.
Therefore, we can apply Proposition \ref{homdeg}  %and in particular \eqref{lbhom}
with $s_{1}=q^{i}+\eta<q^{i+1}-\eta=s_2$ (see also Remark \ref{homdegrmk}) and Proposition \ref{prop:psisimplf}
 to get that for every $i$ and $m$ there exists a modulus of continuity $\omega$ such that
\begin{equation*}
\begin{aligned}
\frac{1}{|\log\ep|}\int_{C_\ep^{i,m}}a\Big(\frac{x}{\delta_\ep}\Big)|\nabla \hat w_\ep|^2\ud x\ge&\ 2\pi(q^{i+1}-q^i-2\eta)\sqrt{\det\ho}|z_{i,m}|^2-\omega(\ep)\\
\ge&\ 2\pi(q^{i+1}-q^i-2\eta)\sqrt{\det\ho}|z_{i,m}|^2-\omega(\ep)\,.
\end{aligned}
\end{equation*}
Summing the previous inequality over $m$ and $i$ and using \eqref{dsum42} yields
\begin{equation}\label{nuo2}
\begin{aligned}
\frac{1}{|\log\ep|}\sum_{i=1}^{\tilde L_1-1}\sum_{m=1}^{M^i}\int_{C_\ep^{i,m}}a\Big(\frac{x}{\delta_\ep}\Big)|\nabla \hat w_\ep|^2\ud x\ge&\
 2\pi\sqrt{\det\ho}(q^{\tilde L_1}-\bar p-2\eta\tilde L_1) |z_0|-\omega(\ep)\\
= &\ 2\pi\sqrt{\det\ho}(\lambda-\bar p- 2\eta\tilde L_1)|z_0|-\omega(\ep)\,.
\end{aligned}
\end{equation}
By \eqref{almostfindopo}, summing \eqref{nuo1} and \eqref{nuo2}, the claim follows taking the limits as $\ep\to 0$\,, $\eta\to 0$, $\bar p\to 0$ and $p,p'\to 1$ and using \eqref{infmin0000}.
\end{proof}
%%%%%%%%%%%%%%%%%%%%%%%%%%%%%%%%%%%%%%%%%%%%%%
%%%%%%%%%%%%%%%%%%%%%%%%%%%%%%%%%%%%%%%%%%%%%%
%%%%%%%%%%%%%%%%%%%%%%%%%%%%%%%%%%%%%%%%%%%%%%
%%%%%%%%%%%%%%%%%%%%%%%%%%%%%%%%%%%%%%%%%%%%%%
%%%%%%%%%%%%%%%%%%%%%%%%%%%%%%%%%%%%%%%%%%%%%%
%%%%%%%%%%%%%%%%%%%%%%%%%%%%%%%%%%%%%%%%%%%%%%
\section{The case $\delta_\ep\lesssim\ep$}\label{sec:delta<ep}
This section is devoted to the proofs of Theorems \ref{mainthmcra} and \ref{intro: mainthmgl}.

\subsection{The core-radius approach}\label{sec:mainthmcra}
%This section is devoted to the proof of Theorem \ref{mainthmcra}.
For the reader's convenience, we  re-state Theorem \ref{mainthmcra} and  recall that $X_\ep(\Omega)$ is defined in \eqref{measep}.
\begin{theorem}\label{mainthmcradopo}
Let $\precrw,\,\precr$ be defined in \eqref{intro:cra_en_0}, \eqref{intro:cra_en}, respectively, with $f$ satisfying \eqref{Pprop}, \eqref{Gprop}, \eqref{Hprop}. Let moreover $\Fzero$ be defined by formula \eqref{gali} with $\Psi(\cdot;T\!f_{\hom})$ given by \eqref{defPsi} and, in the latter formula,  $h=T\!f_{\hom}$.
 If $\limsup_{\ep\to 0}\frac{\delta_\ep}{\ep}< +\infty$,
then the following statements hold true.
\begin{itemize}
\item[(i)] (\,$\Gamma$-liminf inequality)
For any $\left\{\mu_\ep\right\}_\ep\subset X(\Omega)$ such that $\mu_\ep\in X_\ep(\Omega)$ for each $\ep>0$ and
$\mu_\ep\fla\mu$ with $\mu\in X(\Omega)$\, the following inequality holds:
\begin{equation*}%\label{liminfcradopo}
\liminf_{\ep\to 0}\frac{\precr(\mu_\ep)}{|\log\ep|}\ge\Fzero(\mu).
\end{equation*}
\item[(ii)](\,$\Gamma$-limsup inequality) For every $\mu\in X(\Omega)$, there exists a sequence
$\left\{\mu_\ep\right\}_\ep\subset X(\Omega)$ with $\mu_\ep\in X_\ep(\Omega)$ for every $\ep>0$
such that  $\mu_\ep\fla\mu$ and
\begin{equation*}%\label{lisufordopo}
\limsup_{\ep\to 0}\frac{\precr(\mu_\ep)}{|\log\ep|}\le\Fzero(\mu)\,.
\end{equation*}
\end{itemize}
\end{theorem}
%\subsection{Proof of the $\Gamma$-liminf inequality}
%%%%%%%%%%%%%%%%%%%%%%%%%%%%%%%
%%%%%%%%%%%%%%%%%%%%%%%%%%%%%%%
%%%%%%%%%%%%%%%%%%%%%%%%%%%%%%%
%\begin{proof}
%%%%%%%%%%%%%%%%%%%%%%%%%%%%%%%%%%%%%%%

{\it Proof of {\rm(i)}.} For every $\ep>0$ we set $\B_\ep:=\{B_\ep(x)\,:\,x\in \supp(\mu_\ep)\}$ and choose $w_\ep\in \AF_\ep(\mu_\ep)$ in such a way that
\begin{equation*}%\label{infmin00}
\precrw(w_\ep; \Omega_\ep(\mu_\ep))\le \precr(\mu_\ep)+C
\end{equation*}
for some constant $C$ independent of $\ep$. We can assume without loss of generality that
\begin{equation}\label{energybound3}
\alpha\int_{\Omega_\ep(\mu_\ep)}|\nabla w_\ep|^2\ud x\le\precrw(w_\ep; \Omega_\ep(\mu_\ep))\le \precr(\mu_\ep)+C \le C|\log\ep|\,,
\end{equation}
where the first inequality is a consequence of assumption \eqref{Gprop}\,.
\medskip

In view of \eqref{energybound3}, by applying \eqref{impbound} with $U=\Omega$, we have that
\begin{equation}\label{impbound3}
\E(\B_\ep, \mu_\ep, \Omega)\le C|\log\ep|\,,
\end{equation}
where $\E$ is defined in \eqref{preminf}-\eqref{minf}\,.
By \eqref{impbound3} and the Jensen inequality, considering the definition of $X_\ep(\Omega)$, we get
%and by applying Proposition \ref{ballconstr} (4) with $t_1=0$ and $t_2=1$
%and by applying Jensen inequality, we have that
\begin{equation}\label{vartotlim}
C|\log\ep|\ge \E(\B_\ep,\mu_\ep,\Omega)\ge 2\pi\alpha \log 2\,\sum_{{B\in\B_\ep}}|\mu_\ep|(B)=2\pi\alpha \log 2\,|\mu_\ep|(\Omega)\,,
\end{equation}
%Note that the last equality above follows from the fact that $\mu_\ep\in X_{\ep}(\Omega)$\,.
%By combining \eqref{vartotlim} with \eqref{impbound2}, we deduce that $|\mu_\ep|(\Omega)\le C|\log\ep|$ so that \red{(I would only write: We deduce that)}
whence we deduce that
\begin{equation}\label{radbound}
\rad(\B_\ep)\le \ep|\mu_\ep|(\Omega)\le C\ep|\log\ep|\to 0\qquad\textrm{as }\ep\to 0\,.%\red{(\rad(\B_\ep)\leq\ep|\mu_\ep|(\Omega)?)}
\end{equation}
%Let us set $\mu_\ep:=\mu_{\ep}$ and observe that, by the very definition of $\precr'(\mu_\ep,\B'_\ep)$ in \eqref{cra_new}, we have that $\precr'(\mu'_\ep,\B'_\ep)\le \precr(\mu_\ep)$.
The claim follows by Proposition \ref{mainthmcranew} whose assumptions are fulfilled in view of \eqref{vartotlim} and \eqref{radbound}.

%%%%%%%%%%%%%%%%%%%%%%%%%%%%%%%%%%%%%%
{\it Proof of {\rm(ii)}.}
%\subsection{Proof of the $\Gamma$-limsup inequality}\label{ss:lscra}
Let $\widetilde{\mathscr{F}}_{0}:X(\Omega)\to [0,+\infty)$ be the functional defined by
$$
\widetilde{\mathscr{F}}_{0}(\mu):=\sum_{i=1}^n\psi(z_i;T\!f_{\hom})\qquad\textrm{for every }\mu=\sum_{i=1}^{n}z_i\delta_{z_i}\in X(\Omega)
$$
and note that the functional $\Fzero$ in \eqref{gali} is the lower semicontinuous envelope of $\widetilde{\mathscr{F}}_{0}$ with respect to the flat convergence.
Hence, given $\mu=\sum_{i=1}^n z_i\delta_{x_i}\in X(\Omega)$,
it is enough to construct $w_{\ep}\in \AS_\ep(\mu)$ such that
\begin{equation}\label{proofub}
\limsup_{\ep\to 0}\frac{1}{|\log\ep|}\precrw(w_\ep;\Omega_\ep(\mu))\le\widetilde{\mathscr{F}}_{0}(\mu)= \sum_{i=1}^{n}\psi(z_{i};T\!f_{\hom})\,.
\end{equation}
For this purpose, we take $\rho>0$ such that $\overline B_{2\rho}(x_{i})\subset\Omega$ for every $i=1,\ldots,n$, and
${\overline B_{2\rho}(x_{i})\cap \overline B_{2\rho}(x_{j})=\emptyset}$ for every $i,j=1,\ldots,n$ with $i\neq j$\,.
Since by assumption $\limsup_{\ep\to 0}\frac{\delta_\ep}{\ep}$ is finite, then for any $s$ with
 $0<s<1$ it holds true that
\begin{equation*}%\label{ok}
\lim_{\ep\to 0}\frac{\delta_\ep}{\ep^s}=0\,.
\end{equation*}
We set $\bar\rho:=\min\{\rho,\frac 1 2\}$ and for every $i=1,\ldots,n$, we let  $w_{\ep,s}^i\in \widetilde\Anew_{\ep^{s},2\bar\rho}(z_{i})$ (where $\widetilde\Anew_{\ep^{s},2\bar\rho}(z_{i})$ is defined in \eqref{anew}) be such that
\begin{equation}\label{forls}
 \precrw(w^{i}_{\ep,s};A_{\ep^{s},2\bar\rho})\le \inf_{w\in \widetilde\Anew_{\ep^{s},2\bar\rho}(z_{i})} \precrw(w;A_{\ep^{s},2\bar\rho})+C\,,
\end{equation}
for some constant $C$ independent of $\ep$. By arguing as in the proof of Lemma \ref{lm:glim0}, we can write $w_{\ep,s}^i=e^{\iota u_{\ep,s}^i}$ for some function  $ u_{\ep,s}^i\in SBV^2(A_{\ep^s,2\bar\rho}(x_{i}))$ with  $u_{\ep,s}^i(\cdot)=z_{i}\theta(\cdot)$ on $\partial B_{\ep^s}\cup\partial B_{2\bar\rho}$ (where $\theta$ is defined in \eqref{deftheta}).
Let furthermore $\sigma:[\bar\rho,2\bar\rho]\to[0,1]$ be the function defined by $\sigma(r):=\frac{1}{\bar\rho}(r-\bar\rho)$ and set
$\Theta(\cdot):=\sum_{k=1}^n z_{k}\theta(\cdot-x_{k})$\,. We define the function $w_{\ep,s}:\Omega_\ep(\mu)\to\mathcal S^1$ as $w_{\ep,s}:=e^{\iota u_{\ep,s}}$ where
\begin{equation}\label{reccra}
u_{\ep,s}(x)=\begin{cases}
z_{i}\theta(x-x_{i})&\textrm{ if }x\in A_{\ep,\ep^s}(x_{i})\textrm{ for some }i\,,
\\
u_{\ep,s}^i(x)&\textrm{ if }x\in A_{\ep^s,\bar\rho}(x_{i})\textrm{ for some }i\,,
\\
(1-\sigma(|x-x_{i}|)z_{i}\theta(x-x_{i})+\sigma(|x-x_{i}|)\Theta(x)&\textrm { if }x\in A_{\bar\rho,2\bar\rho}(x_{i})\textrm{ for some }i\,,
\\
\Theta(x)&\textrm{ elsewhere\,.}
\end{cases}
\end{equation}
The function $w_{\ep,s}$ belongs to $\AS_\ep(\mu)$\,.
By property \eqref{Gprop}, for every $i=1,\ldots,n$ there exists a constant $C=C(\beta,\bar\rho,\Omega,\{z_{i}\}_i)>0$ such that
\begin{equation}\label{innann}
\precrw(w_{\ep,s};A_{\ep,\ep^s}(x_{i}))\le \beta\int_{A_{\ep,\ep^s}(x_{i})}|\nabla w_{\ep,s}|^2\ud x=2\pi\beta|z_{i}|^2 (1-s)|\log\ep|,
\end{equation}
 \begin{equation}\label{annrho}
\precrw(w_{\ep,s};\Omega_{2\bar\rho}(\mu))\le\beta \int_{\Omega_{2\bar\rho}(\mu)}|\nabla w_{\ep,s}|^2\ud x
\le C\,,
\end{equation}
and
\begin{equation}\label{outann}
\begin{aligned}
\precrw(w_{\ep,s};A_{\bar\rho,2\bar\rho}(x_{i}))
&\le \beta\int_{A_{\bar\rho,2\bar\rho}(x_{i})}|\nabla w_{\ep,s}|^2\ud x\\
&\le C \sum_{\newatop{k=1}{k\neq i}}^n |z_{k}|^2 \int_{A_{\bar\rho,2\bar\rho}(x_{i})}|\sigma'(|x|)|^2 |\theta(x-x_{k})|^2\ud x\\
&\qquad+C\sum_{k=1}^n |z_{k}|^2\int_{A_{\bar\rho,2\bar\rho}(x_{i})}|\nabla\theta(x-x_{k})|^2\ud x\le C\,.
\end{aligned}
\end{equation}
In addition, since $2\bar\rho \le 1$, by  \eqref{forls}, Lemma \ref{lm:glim0} and Proposition \ref{homdeg}, there exists a modulus of continuity $\omega$ such that for every $i=1,\ldots,n$ we have
\begin{equation}\label{ubcore}
\begin{aligned}
\frac{1}{|\log\ep|}\precrw(w_{\ep,s}; A_{\ep^{s},\bar\rho}(x_{i}))&\le
\frac{1}{|\log\ep|}\precrw(w_{\ep,s}; A_{\ep^{s},2\bar\rho}(x_{i}))
\\
&\le\frac{1}{|\log\ep|}\inf_{w\in \Anew_{\ep^s,2\bar\rho}(x_{i})}\precrw(w; A_{\ep^{s},2\bar\rho}(x_{i}))+\omega(\ep)
\\
&=\Big(s-\frac{|\log(2\bar\rho)|}{|\log\ep|}\Big)\psi(z_{i};T\!f_{\hom})+\omega({\ep})\,.
\end{aligned}
\end{equation}
Finally, due to \eqref{innann}, \eqref{annrho}, \eqref{outann} and \eqref{ubcore} we can choose $\omega$ in such a way that
\begin{equation}\label{conclub}
\frac{1}{|\log\ep|}\precrw(w_{\ep,s}; \Omega_{\ep}(\mu))\le  s\sum_{i=1}^n \psi(z_{i};T\!f_{\hom})+2\pi \beta(1-s)\sum_{i=1}^{n}|z_{i}|^2+\omega(\ep)\,.
\end{equation}
Suitably choosing $s_\ep\to 1$ as $\ep\to 0$, we have that $w_\ep=w_{\ep,s_\ep}$ satisfies the relations in \eqref{proofub}.
\qed
%\end{proof}
%%%%%%%%%%%%%%%%%%%%%%%%%%%%
%%%%%%%%%%%%%%%%%%%%%%%%%%%%
%%%%%%%%%%%%%%%%%%%%%%%%%%%%
\subsection{The Ginzburg-Landau model}\label{sec:mainthmgl}
This subsection is devoted to the proof of Theorem \ref{intro: mainthmgl}, which we prove here under slightly more general assumptions on the potential term.
More specifically, we consider $W\in C^0([0,+\infty))$ such that $W(\tau)\ge 0$, $W^{-1}(0)=\{1\}$ and
$$
\liminf_{\tau\to 1}\frac{W(\tau)}{(1-\tau^2)}>0,\qquad \qquad\liminf_{\tau\to +\infty}W(\tau)>0\,.
$$
and, we define $\pregl^W: H^1(\Omega;\R^2)\to \R$ as
\begin{equation}\label{GL0}
\pregl^W(v):=\int_\Om a\Big(\frac{x} {\delta_\ep}\Big)|\nabla v(x)|^2 \ud x+\frac{1}{\ep^2}\int_\Om W(|v(x)|)  \ud x\,.
\end{equation}

We can re-state Theorem \ref{intro: mainthmgl} as follows.

\begin{theorem}\label{mainthmgldopo}{Let $\pregl^W$ be defined in \eqref{GL0} where $a$ is a measurable $(0,1)^2$-periodic function satisfying $a(x)\in [\alpha,\beta]$ for a.e. $x\in\R^2$.
Let moreover $\ho$ be the symmetric matrix defined in \eqref{defahom}.
If $\limsup_{\ep\to 0}\frac{\delta_\ep}{\ep}< +\infty$,
then the following $\Gamma$-convergence result holds true.

\begin{itemize}
\item[(i)] (Compactness) Let $\left\{v_\ep\right\}_\ep\subset H^1(\Omega;\R^2)$ be such that $\pregl^W(v_\ep)\le C|\log\ep|$. Then, there exists $\mu\in X(\Omega)$ such that, up to subsequences, $Jv_\ep\fla\pi\mu$.
\item[(ii)] (\,$\Gamma$-liminf inequality)
Let $\{v_\ep\}_\ep\subset H^1(\Omega;\R^2)$ be such that $J v_\ep\fla\pi\mu$ for some $\mu\in X(\Omega)$\,. Then
\begin{equation*}%\label{liminfgldopo}
\liminf_{\ep\to 0}\frac{\pregl^W(v_\ep)}{|\log\ep|}\ge 2\pi\sqrt{\det\ho}|\mu|(\Omega).
\end{equation*}
\item[(iii)](\,$\Gamma$-limsup inequality) For every $\mu\in X(\Omega)$, there exists a sequence $\{v_\ep\}_\ep\subset H^1(\Omega;\R^2)$ such that  $Jv_\ep\fla\pi\mu$ and
\begin{equation}\label{lsgldopo}
\limsup_{\ep\to 0}\frac{\pregl^W(v_\ep)}{|\log\ep|}\le2\pi\sqrt{\det\ho}|\mu|(\Omega)\,.
\end{equation}
\end{itemize}
}
\end{theorem}
%%%%%%%%%%%%%%%%%%%%%%%%%%%%
%%%%%%%%%%%%%%%%%%%%%%%%%%%%
%%%%%%%%%%%%%%%%%%%%%%%%%%%%
\begin{proof}
%\subsection{Proof of the $\Gamma$-liminf inequality}\label{linfglh}
Since $a\le\beta$ a.e., the compactness property (i) is a corollary of classical results in the variational analysis of the classical GL functional (see for instance \cite[Theorem 4.1]{AP}).
%In view of the assumption \eqref{Gprop}, the compactness property (i) is a corollary of classical results in the variational analysis of the classical GL functional (see for instance \cite[Theorem 4.1]{AP}).

{\it Proof of {\rm(ii)}.}
The strategy of the proof is to bound from below $\pregl^W(v_\ep)$ with $\precr(\mu_\ep,\B_\ep)$ defined in \eqref{cra_new} for a suitable choice of $\mu_\ep$ and $\B_\ep$ satisfying the assumptions of Proposition \ref{mainthmcranew}.
% and \ref{prop:psisimplf}
Without loss of generality we can assume that
\begin{equation}\label{energbound}
\pregl^W(v_\ep)\le C|\log\ep|
\end{equation}
and by the standard density arguments we can also assume that $\{v_\ep\}_{\e}\subset H^{1}(\Omega;\R^2)\cap C^1(\Omega;\R^2)$.
For every $0<\gamma_1,\gamma_2<\frac 1 2$ and for every $\ep>0$ we set
\begin{equation*}%\label{badcomp}
K_{\ep,\gamma_1,\gamma_2}:=\{|v_\ep|\le 1-\gamma_1\}\cup\{|v_\ep|\ge 1+\gamma_2\}\qquad\textrm{and}\qquad \Lambda_{\ep,\gamma_1,\gamma_2}:=\partial K_{\ep,\gamma_1,\gamma_2}\setminus\partial\Omega.
\end{equation*}

By \eqref{simplf} and by the Young inequality we have that
\begin{equation}\label{ste1}
\begin{aligned}
C|\log\ep|\ge &\int_{\Omega}\alpha|\nabla|v_\ep||^2+\frac 1 {\ep^2}W(v_\ep)\ud x
\ge 2\frac{\sqrt{\alpha}}{\ep}\int_{\Omega}\sqrt{W(|v_\ep|)}|\nabla|v_\ep||\ud x.
\end{aligned}
\end{equation}
For every $t\in\R$ we set
\begin{equation*}%\label{defh}
h(t)=\int_{t}^1\sqrt{W(s)}\ud s
\end{equation*}
and we define the function $\hat v_\ep:\Omega\to\R^+$ as $\hat v_\ep(x)=|h(|v_\ep(x)|)|$. Note that $\hat v_\ep\in H^1(\Omega)$ and that $|\nabla \hat v_\ep|=\sqrt{W(|v_\ep|)}|\nabla|v_\ep||$, so that
by \eqref{ste1}, the coarea formula and the mean-value theorem, for every $\bar \tau\in (0,1)$ there exists $\bar\tau_\ep\in (0,\bar\tau)$ such that
\begin{equation}\label{ste2}
C\ep|\log\ep|\ge \int_{\frac{\bar\tau}{2}}^{\bar \tau}\mathcal{H}^1(\{\hat v_\ep=\tau\})\ud \tau\ge \frac{\bar\tau}{2} \mathcal{H}^1(\{\hat v_\ep=\bar\tau_\ep\}).
\end{equation}
We set $\gamma^\ep_1:=1-h^{-1}(\bar\tau_\ep)$ and $\gamma^\ep_2:=h^{-1}(-\bar\tau_\ep)-1$ and note that, by construction, there exists $\gamma^{\bar \tau}\in (0,1)$ such that $\gamma^{\bar \tau}\to 0$ as $\bar\tau\to 0$  and there exists a constant $0<c<1$ (independent of $\bar\tau$) such that $\gamma^\ep_1,\gamma^\ep_2\in(c\gamma^{\bar \tau},\gamma^{\bar\tau})$.
Moreover, we have that
$$
\{\hat v_\ep<\bar\tau_\ep\}=\{h^{-1}(\bar\tau_\ep)<|v_\ep|<h^{-1}(-\bar\tau_\ep)\}=
\{1-\gamma^\ep_1<|v_\ep|<1+\gamma^\ep_2\}=\Omega\setminus K_{\ep,\gamma_1^\ep,\gamma_2^\ep}
$$
and from the regularity of the function $v_\ep$ it follows that
$$
\{\hat v_\ep=\bar\tau_\ep\}=\partial \{\hat v_\ep<\bar\tau_\ep\}\setminus\partial\Omega=\partial(\Omega\setminus K_{\e,\gamma^\ep_1,\gamma^\ep_2})\setminus\partial\Omega
=\partial K_{\ep,\gamma^\ep_1,\gamma^\ep_2}\setminus\partial\Omega=\Lambda_{\ep,\gamma^\ep_1,\gamma^\ep_2}.
$$
Therefore, by \eqref{ste2}, we obtain that
\begin{equation}\label{ste3}
\mathcal H^1(\Lambda_{\ep,\gamma_1^\ep,\gamma_2^\ep})\le C_{\bar\tau}\ep|\log\ep|.
\end{equation}
By \eqref{energbound}, we have that
\begin{equation}\label{ste4}
C\ep^2|\log\ep|\ge \int_{K_{\ep,\gamma^\ep_1,\gamma^\ep_2}}W(|v_\ep|)\ud x\ge C_{\bar\tau}|K_{\ep,\gamma^\ep_1,\gamma^\ep_2}|.
\end{equation}
As a result, thanks to the Lipschitz regularity of $\partial\Omega$ and to \eqref{ste3}, we have that
\begin{equation}\label{ste5}
\mathcal H^1(\partial K_{\ep,\gamma^\ep_1,\gamma^\ep_2})\le C_{\bar\tau}\mathcal H^1(\Lambda_{\ep,\gamma^\ep_1,\gamma^\ep_2})\le C_{\bar\tau}\ep|\log\ep|.
\end{equation}
Note that, by definition of Hausdorff measure, since $\partial K_{\ep,\gamma^\ep_1,\gamma^\ep_2}$ is compact, it is always contained in a finite union of balls $B_{r_i}(y_i)$ such that $\sum_{i}r_i\le\mathcal{H}^1(\partial K_{\ep,\gamma^\ep_1,\gamma^\ep_2})$\,. %\red{Not important: does it hold with the constant $1$? }
Moreover, after a merging procedure, we can always assume that such balls are disjoint. In view of \eqref{ste4}, for $\ep$ small enough, we have that  $K_{\ep,\gamma^\ep_1,\gamma^\ep_2}$  is contained in the union of such balls.
Therefore thanks to the previous argument, by \eqref{ste5},  we have proved that there exists a family of balls, that are denoted by $\B'_{\ep}$\,, whose union contains $K_{\ep,\gamma^\ep_1,\gamma^\ep_2}$\, and such that
\begin{equation}\label{rad}
\rad(\B'_{\ep})\le C_{\bar\tau}\mathcal{H}^1(\partial K_{\ep,\gamma^\ep_1,\gamma^\ep_2})\le
C_{\bar\tau}\ep|\log\ep|.
\end{equation}
For every $\ep>0$, let $\B'_\ep(t)$ be a time parametrized family of balls constructed as in Proposition \ref{ballconstr} starting from $\B'_\ep(0):=\B'_\ep$.
Set $\B_\ep:=\B'_\ep(1)$, $\Ccal_\ep:= \{ B\in\B_\ep \,:\, B\subset \Om\}$ and
$\mu_\ep:=\sum_{{B\in\Ccal_\ep}}\deg(v_\ep,\partial B)\delta_{x_B}\,,$ where $x_B$ denotes the center of the ball $B$.
Note that
\begin{equation}\label{alsolater}
1-\gamma^{\bar\tau}<|v_\ep|<1+\gamma^{\bar\tau}\quad\textrm{in }\Omega(\B'_\ep)\supset \Omega(\B_\ep).
\end{equation}
Now we consider the function $w_\ep:\Omega(\B'_\ep)\to \R^2$ defined as $w_\ep(x):=\frac{v_\ep(x)}{|v_\ep(x)|}$\,, and we note that $w_\ep\in H^1(\Omega(\B'_\ep);\mathcal S^1)$. Moreover, considering \eqref{simplf}, \eqref{alsolater}, using the relation $|\nabla v_{\ep}|^{2}=|v_\ep|^2\Big|\nabla\frac{v_\ep}{|v_\ep|}\Big|^2+|\nabla|v_{\e}||^{2}$ and  applying Proposition \ref{ballconstr}(4) with $t_1=0$, $t_2=1$, and $U=\Omega$, we obtain that
\begin{equation*}
\begin{aligned}
C|\log\ep|& \ge \int_{\Omega(\B'_\ep)}\alpha |\nabla v_\ep|^2\ud x\ge \int_{\Omega(\B'_\ep)}\alpha |v_\ep|^2\Big|\nabla\frac{v_\ep}{|v_\ep|}\Big|^2\ud x\ge \alpha (1-\gamma^{\bar\tau})^2 \int_{\Omega(\B'_\ep)} |\nabla w_\ep|^2\ud x \\
&\ge2\pi\alpha (1-\gamma^{\bar\tau})^2|\mu_\ep|(\Omega)\log 2\,,
\end{aligned}
\end{equation*}
from which we deduce that
\begin{equation}\label{ste7.0}
|\mu_\ep|(\Omega)\le C_{\bar\tau}|\log\ep|.
\end{equation}
Furthermore, by Proposition \ref{ballconstr}(5) and by \eqref{rad} it follows that
\begin{equation}\label{ste7.1}
\rad(\B_\ep)\le 2\rad(\B'_\ep)\le C_{\bar\tau}\ep|\log\ep|.
\end{equation}
Now we show that
\begin{equation}\label{ste7}
\mu_\ep \fla\mu.
\end{equation}
By \eqref{defcur} $\deg(v_\ep,\partial B)=\deg(w_\ep,\partial B)$ for every $B\in\Ccal_\ep$. Hence, recalling the notion of modified Jacobian introduced in \eqref{moja}, we have that
\begin{equation}\label{ste5.5}
(J_{1-\gamma_1^\ep}v_\ep-\pi\mu_\ep)(B)=0\qquad\textrm{ for every }B\in\Ccal_\ep\,.
\end{equation}
Using the triangle inequality, Proposition \ref{sempre}, \eqref{ste7.0}, \eqref{ste7.1} and \eqref{ste5.5} we also have that
\begin{equation}\label{ste8}
\begin{aligned}
\|Jv_\ep-\pi\mu'_\ep\|_{\flt}&\le \|Jv_\ep-J_{1-\gamma_1^\ep}v_\ep\|_{\flt}+\|J_{1-\gamma_1^\ep}v_\ep-\pi\mu'_\ep\|_{\flt}\\
&\le  C_{\bar\tau}\ep|\log\ep|+2\sup_{\|\ffi\|_{C_c^{0,1}(\Omega)}\le 1}\sum_{B\in\B_\ep}|\mu_{\ep}|(B)\mathrm{osc}_B(\ffi)\\
&\le  C_{\bar\tau}\ep|\log\ep|+2\rad(\B_\ep)|\mu_{\ep}|(\Omega)\\
&\le  C_{\bar\tau}\ep|\log\ep|+2C_{\bar\tau}\ep|\log\ep|^2.
\end{aligned}
\end{equation}
Eventually, \eqref{ste7} follows from \eqref{ste8} and from the assumption $J v_\ep\fla\pi\mu$ applying the triangle inequality.

Thanks to \eqref{simplf} and \eqref{alsolater}, we get that
\begin{equation}\label{ste6}
\begin{aligned}
\pregl^W(v_\ep)&\ge \int_{\Omega(\B_\ep)}a\Big(\frac{x}{\delta_\ep}\Big)|\nabla v_\ep|^2\ud x\ge\int_{\Omega(\B_\ep)}a\Big(\frac{x}{\delta_\ep}\Big)|v_\ep|^2\Big|\nabla\frac{v_\ep}{|v_\ep|}\Big|^2\ud x\\
&\ge (1-\gamma^{\bar\tau})^2 \int_{\Omega(\B_\ep)}a\Big(\frac{x}{\delta_\ep}\Big)|\nabla w_\ep|^2\ud x\ge (1-\gamma^{\bar\tau})^2 \precr(\mu_\ep,\B_\ep),
\end{aligned}
\end{equation}
where $\precr$ is defined in \eqref{cra_new} with $f(\frac{x}{\delta_\ep},\nabla w(x))=a(\frac{x}{\delta_\ep})|\nabla w(x)|^2$.
Thanks to \eqref{ste7.0}, \eqref{ste7.1}, \eqref{ste7} and \eqref{ste6} we are in a position to apply Proposition \ref{mainthmcranew}, obtaining
$$
\liminf_{\ep\to 0}\frac{\pregl^W(v_\ep)}{|\log\ep|}\ge(1-\gamma^{\bar\tau})^2\Fzero(\mu)=\sqrt{\det\ho}|\mu|(\Omega)\,,
$$
where the last equality follows by Proposition \ref{prop:psisimplf}.
The claim follows letting $\bar\tau\to 0$.
%%%%%%%%%%%%%%%%%%%%%%%%%%%%
%%%%%%%%%%%%%%%%%%%%%%%%%%%%
%%%%%%%%%%%%%%%%%%%%%%%%%%%%

{\it Proof of {\rm(iii)}.}
%\subsection{Proof of the $\Gamma$-limsup inequality}\label{lsglwithh}
%To this end, we argue as in Subsection \ref{ss:lscra}.
We prove the claim under more general assumptions on the functional $\pregl^W$. Specifically, let $f$ be a function satisfying \eqref{Pprop}, \eqref{Gprop}, \eqref{Hprop} and define the energy functional $\pregl^{W,f}:H^{1}(\Omega;\R^2)\to [0,+\infty)$ as
$$
\pregl^{W,f}(v):=\int_\Om f\Big(\frac{x} {\delta_\ep},\nabla v(x)\Big) \ud x+\frac{1}{\ep^2}\int_\Om W(|v(x)|)  \ud x\,.
$$
 We prove that for every $\mu=\sum_{i=1}^nz_{i}\delta_{x_{i}}\in X(\Om)$ there exists a sequence $\{v_\ep\}_\ep\subset H^{1}(\Om;\R^2)$ such that $Jv_\ep\fla\pi\mu$ and
 \begin{equation}\label{lsgldopo2}
 \limsup_{\ep\to 0}\frac{\pregl^{W,f}(v_\ep)}{|\log\ep|}\le \Fzero(\mu)\,,
 \end{equation}
 where $\Fzero$ is defined in \eqref{gali}. In view of Proposition \ref{prop:psisimplf}  we have that \eqref{lsgldopo} is a consequence of \eqref{lsgldopo2}.

By arguing as in  the proof of Theorem \ref{mainthmcradopo} (iii), we may reduce to the case $\Psi(z_{i};T\!f_{\hom})=\psi(z_{i};T\!f_{\hom})$ for every $i=1,\ldots,n$\,.
Let $\rho>0$ be such that $\overline B_{2\rho}(x_{i})\subset\Omega$ for every $i=1,\ldots,n$ and
$\overline B_{2\rho}(x_{i})\cap \overline B_{2\rho}(x_{j})=\emptyset$ for every $i,j=1,\ldots,n$ with $i\neq j$.
For every $0<s<1$, since by assumption $\limsup_{\ep\to 0}\frac{\delta_\ep}{\ep}$ is finite, we have that
$\lim_{\ep\to 0}\frac{\delta_\ep}{\ep^s}=0$.
Finally, we set $\Omega_\ep(\mu):=\Omega\setminus\bigcup_{i=1}^n B_{\ep}(x_{i})$ and we let $u_\ep^s$ be the function defined in \eqref{reccra}.
For every $\ep>0$ we set
 \begin{equation*}%\label{recgl}
v_{\ep,s}(x):=\begin{cases}
e^{\iota u^s_\ep(x)}&\textrm{ if }x\in\Omega_{\ep}(\mu)\,,\\
\frac{|x-x_{i}|}{\ep}\Big(\frac{x-x_{i}}{|x-x_{i}|}\Big)^{z_{i}}&\textrm{ if }x\in B_{\ep}(x_{i})\textrm{ for some }i\,,
\end{cases}
\end{equation*}
and we note that $v_{\ep,s}\in H^1(\Omega;\R^2)$ and that $Jv_{\ep,s}=\pi\mu$ for every $\ep>0$. In addition, for almost every $x\in\Om_\ep(\mu)$ we have that $|v_{\ep,s}(x)|=1$, hence $W(|v_{\ep,s}(x)|)=0$. The latter yields
\begin{equation}\label{estcoresvep}
\int_{\Omega}W(|v_{\ep,s}|)\ud x=\sum_{i=1}^{n}\int_{B_\ep(x_{i})}W(|v_{\ep,s}|)\ud x\le nC\pi\ep^2\,
\end{equation}
by the continuity of $W$. Furthermore, by the very definition of $v_{\ep,s}$, we have that
\begin{equation}\label{estcoresgrad}
\sum_{i=1}^{n}\int_{B_\ep(x_{i})}f\Big(\frac{x}{\delta_\ep},v_{\ep,s}\Big)\ud x\le\sum_{i=1}^{n}\int_{B_\ep(x_{i})}|\nabla v_{\ep,s}|^2\ud x\,
\le 2\sum_{i=1}^{n}\pi(1+|z_{i}|^2)\,.
\end{equation}
Gathering together \eqref{estcoresvep}, \eqref{estcoresgrad} and \eqref{conclub}, we eventually obtain that
\begin{equation*}%\label{almconcl}
\begin{aligned}
\frac{1}{|\log\ep|}\pregl^{W,f}(v_\ep)&\le \frac{1}{|\log\ep|}\int_{\Omega_\ep(\mu)}f\Big(\frac{x}{\delta_\ep},v_{\ep,s}\Big)\ud x+\mathrm{o}(1)\\
&\le s\sum_{i=1}^n \psi(z_{i};T\!f_{\hom})+2\pi \beta(1-s)\sum_{i=1}^{n}|z_{i}|^2+\mathrm{o}(1)\,,
\end{aligned}
\end{equation*}
which, suitably choosing $s_\ep\to 1$ as $\ep\to 0$ and setting $v_{\ep}:=v_{\ep,s_\ep}$\,, gives \eqref{lsgldopo2}. % by taking first the limit as $\ep\to 0$ and then as $s\to 1$\,.
\end{proof}
%%%%%%%%%%%%%%%%%%%%%%%%%%%%
%%%%%%%%%%%%%%%%%%%%%%%%%%%%
%%%%%%%%%%%%%%%%%%%%%%%%%%%%
\section{The case $\delta_\ep\gg\ep$}\label{sec:delta>ep}
This section is devoted to the proofs of Theorems \ref{cratbw} and of \ref{gltbw}.
We will prove the above $\Gamma$-convergence results under the assumption that
%$\delta_\ep\to 0$ as $\ep\to 0$ and that
\begin{equation}\label{limlog0}
\lim_{\ep\to 0}\delta_\ep= 0\,,\qquad\lambda:=\lim_{\ep\to 0}\frac{|\log\delta_\ep|}{|\log\ep|}\in[0,1)\,.
\end{equation}
%Notice that the assumption above implies that
%$$
%\lim_{\ep\to 0}\frac{\delta_\ep}{\ep}=+\infty\,.
%$$
%%Notice that, if $\limsup_{\ep\to 0}\frac{\delta_\ep}{\ep}=+\infty$, then
%%$$
%%\lambda:=\limsup_{\ep\to 0}\frac{|\log\delta_\ep|}{|\log\ep|}\in[0,1)\,.
%%$$
\subsection{The core-radius approach}
%%%%%%%%%%%%%%%%%%%%%%%%%%%%
%%%%%%%%%%%%%%%%%%%%%%%%%%%%
%%%%%%%%%%%%%%%%%%%%%%%%%%%%
For the reader's convenience, we re-state Theorem \ref{cratbw} and we recall that $X_\ep(\Omega)$ is defined in \eqref{measep}.
%%%%%%%%%%%%%%%%%%%%%%%%%%%%
%%%%%%%%%%%%%%%%%%%%%%%%%%%%
%%%%%%%%%%%%%%%%%%%%%%%%%%%%
\begin{theorem}\label{cradelta>ep}
Let $\precrw,\,\precr$ be defined in \eqref{aform}, \eqref{intro:cra_en}, respectively, with $f$ of the form \eqref{simplf}, where $a$ is a measurable $(0,1)^2$-periodic function satisfying $a(x)\in [\alpha,\beta]\subset (0,+\infty)$ for a.e. $x\in\R^2$\,. Let moreover $\ho$ be the matrix defined in \eqref{defahom}. If
\eqref{limlog0} is satisfied,
%\begin{equation}\label{limlog}
%\lambda:=\limsup_{\ep\to 0}\frac{|\log\delta_\ep|}{|\log\ep|}\in[0,1)\,,
%\end{equation}
then the following statements hold true.
\begin{itemize}
\item[(i)] (\,$\Gamma$-liminf inequality)
For any family $\left\{\mu_\ep\right\}_\ep\subset X(\Omega)$ such that $\mu_\ep\in X_\ep(\Omega)$ for every $\ep>0$ and
$\mu_\ep\fla\mu$ with $\mu\in X(\Omega)$  we have
\begin{equation*}%\label{liminfcra}
\liminf_{\ep\to 0}\frac{\precr(\mu_\ep)}{|\log\ep|}\ge2\pi\bigg(\Big(1-\lambda)\mathrm{ess}\inf a+ \lambda\sqrt{\det\ho}\bigg)|\mu|(\Omega)\,.
\end{equation*}
\item[(ii)](\,$\Gamma$-limsup inequality) For every $\mu\in X(\Omega)$, there exists a sequence
$\left\{\mu_\ep\right\}_\ep\subset X(\Omega)$ with $\mu_\ep\in X_\ep(\Omega)$ for every $\ep>0$
such that  $\mu_\ep\fla\mu$ and
\begin{equation*}%\label{lisufor}
\limsup_{\ep\to 0}\frac{\precr(\mu_\ep)}{|\log\ep|}\le2\pi\bigg(\Big(1-\lambda)\mathrm{ess}\inf a+ \lambda\sqrt{\det\ho}\bigg)|\mu|(\Omega)\,.
\end{equation*}
\end{itemize}
\end{theorem}
%%%%%%%%%%%%%%%%%%%%%%%%%%%%
%%%%%%%%%%%%%%%%%%%%%%%%%%%%
%%%%%%%%%%%%%%%%%%%%%%%%%%%%
%\begin{proof}
%Since $a\le\beta$ a.e., the compactness property (i) is a corollary of classical results in the variational analysis of the classical GL functional (see for instance \cite[Theorem 4.1]{AP}).
%\vskip8pt
%%%%%%%%%%%%%%%%%%%%%%%%%%%%
%%%%%%%%%%%%%%%%%%%%%%%%%%%%
%%%%%%%%%%%%%%%%%%%%%%%%%%%%

\noindent{\it Proof of {\rm(i)}.}
For every $\ep>0$ we set $\B_\ep:=\{B_\ep(x)\,:\,x\in \supp(\mu_\ep)\}$ and choose $w_\ep\in \AF_\ep(\mu_\ep)$ in such a way that
\begin{equation}\label{infmin000}
\precrw(w_\ep; \Omega_\ep(\mu_\ep))\le \precr(\mu_\ep)+C
\end{equation}
for some constant $C$ independent of $\ep$.
We can assume without loss of generality that
\begin{equation}\label{energybound4}
\alpha\int_{\Omega_\ep(\mu_\ep)}|\nabla w_\ep|^2\ud x\le\precrw(w_\ep; \Omega_\ep(\mu_\ep))\le \precr(\mu_\ep)+C \le C|\log\ep|\,.
\end{equation}
By arguing as in the first part of the proof of Theorem \ref{mainthmcradopo}(ii) we get that $|\mu_\ep|(\Omega)\le C|\log\ep|$ and hence $\rad(\B_\ep)\le C\ep|\log\ep|$\,. Therefore, Proposition \ref{mainthmcranewdelta>ep} yields the claim.
%Furthermore, let $c>1$ be such that $\log c<\frac p 2\frac{|\log\rad(\B_\ep)|}{|\mu_\ep|(\Omega)+1}$; by reasoning as in the proof of Proposition \ref{mainthmcranew} we are allowed to take the constant $c$ independent of $\ep$\,. Moreover,
\vskip8pt
%%%%%%%%%%%%%%%%%%%%%%%%%%%%
%%%%%%%%%%%%%%%%%%%%%%%%%%%%
%%%%%%%%%%%%%%%%%%%%%%%%%%%%

{\it Proof of {\rm(ii)}.}
By standard density arguments in the Ginzburg-Landau theory, we can reduce to the case that $\mu=\sum_{i=1}^nz_i\delta_{x_i}$ with $|z_i|=1$\,.
We set $m:=\mathrm{ess}\inf a$ and for every $\eta\in (0,1)$ let
$$
E_\eta:=\{y\in [0,1)^2\,:\,a(y)\le m+\eta\}\,.
$$
By the very definition of $\mathrm{ess}\inf$ we have that $|E_\eta|>0$ for every $\eta$ and there exists $y_\eta\in E_\eta$ having density $1$ in $E_\eta$\,, i.e.,
\begin{equation}\label{den1}
\lim_{r\to 0}\frac{|E_\eta\cap B_{r}(y_\eta)|}{ r^2 }=1\,.
\end{equation}
For every $i=1,\ldots,n$ we set
\begin{equation}\label{newsing}
x_i^{\delta_\ep,\eta}:=\delta_\ep\Big\lfloor {\frac{x_i}{\delta_\ep}}\Big\rfloor+\delta_\ep y_\eta\,.
\end{equation}
Since $\delta_\ep\to 0$ as $\ep\to 0$ we have that $x^{\delta_\ep,\eta}_i\to x_i$ as $\ep\to 0$ for every $i=1,\ldots,n$.  Therefore, setting
\begin{equation}\label{recose}
\mu_{\ep,\eta}:=\sum_{i=1}^{n}z_i\delta_{x^{\delta_\ep,\eta}_i}\,,
\end{equation}
 we have that
 \begin{equation}\label{flareco}
 \mu_{\ep,\eta}\fla\mu\qquad\textrm{ as }\ep\to 0\,.
 \end{equation}
Now we prove that for every $0<s<1$ there exists a function $w_{\ep,\eta,s}\in \AS_{\ep}(\mu_{\ep,\eta})$ such that
\begin{equation}\label{lisu}
\begin{aligned}
\limsup_{\ep\to 0}\frac{1}{|\log\ep|}\precrw(w_{\ep,\eta,s},\Omega_\ep(\mu_{\ep,\eta}))\le &2\pi\Big((1-s\lambda)(m+\eta)+\lambda s\sqrt{\det\ho}\Big)|\mu|(\Omega)\,\\
&+\Big(2\pi\beta(1-s)\lambda+ 2\pi(1-s)\lambda\Big)|\mu|(\Omega).
\end{aligned}
\end{equation}
%\begin{equation}\label{lisu}
%\begin{aligned}
%\limsup_{\ep\to 0}\frac{1}{|\log\ep|}\int_{\Omega_\ep(\mu_{\ep,\eta})}a\Big(\frac{x}{\delta_\ep}\Big)|\nabla w_{\ep,\eta}|^2\ud x\le &2\pi\Big((1-s\lambda)(m+\eta)+\lambda \sqrt{\det\ho}\Big)|\mu|(\Omega)\,\\
%&+\Big(2\pi\beta(1-s)\lambda+ 2\pi(1-s)\lambda\Big)|\mu|(\Omega).
%\end{aligned}
%\end{equation}
We fix $\rho>0$ such that $\overline B_{2\rho}(x^{\delta_\ep,\eta}_{i})\subset\Omega$ for every $i=1,\ldots,n$ and $\bar B_{2\rho}(x^{\delta_\ep,\eta}_i)\cap\bar B_{2\rho}(x^{\delta_\ep,\eta}_j)\neq\emptyset$ for every $i,j=1,\ldots,n$ with $i\neq j$\,. Let $\bar\rho:=\min\{\rho,\frac 1 2\}$\,.

Furthermore, for every $0<s<1$, we let $w_{\delta_\ep,s}^i\in \widetilde\Anew_{\delta_\ep^s,2\bar\rho}(z_i)$ be such that
\begin{equation}\label{zeroi}
F_{\delta_\ep}(w_{\delta_\ep,s}^i;A_{\delta_\ep^s,2\bar\rho})\le\inf_{w\in\widetilde\Anew_{\delta_\ep^s,2\bar\rho}}F_{\delta_\ep}(w;A_{\delta_\ep^s,2\bar\rho})+C\,.
\end{equation}

 By arguing as in the proof of Lemma \ref{lm:glim0}, we can write $w_{\delta_\ep,s}^i=e^{\iota u_{\delta_\ep,s}^i}$ for some function  $ u_{\delta_\ep,s}^i\in SBV^2(A_{\delta_\ep^s,2\bar\rho}(x_{i}))$ with  $u_{\delta_\ep,s}^i(\cdot)=z_{i}\theta(\cdot)$ on $\partial B_{\delta_\ep^s}\cup\partial B_{2\bar\rho}$ (where $\theta$ is defined in \eqref{deftheta}).
% Let $\gamma\in (0,1)$ and
%let $\sigma_{\delta_\ep}^\gamma:[\gamma\delta_\ep,\delta_\ep]\to[0,1]$ be the function defined by $\sigma_{\delta_\ep}^\gamma(r)=\frac{r-\gamma\delta_\ep}{\delta_\ep(1-\gamma)}$\,.

Let $\theta$ be the function defined in \eqref{deftheta}
and let $\Theta(\cdot):=\sum_{k=1}^n z_{k}\theta(\cdot-x_{k})$\,.
Let furthermore $\sigma:[\bar\rho,2\bar\rho]\to[0,1]$ be the function defined by $\sigma(r):=\frac{1}{\bar\rho}(r-\bar\rho)$\,. For every $i=1,\ldots,n$ we set
 \begin{equation*}
 u_{\ep,\eta,s}^i(x):=\left\{\begin{array}{ll}
 z_{i}\theta(x-x^{\delta_\ep,\eta}_{i})&\textrm{ if }x\in A_{\ep,\delta_\ep^s}(x^{\delta_\ep,\eta}_{i})\,,
\\
%(1-\sigma_{\delta_\ep}^\gamma(|x-x^{\delta_\ep,\eta}_{i}|))z_i\theta(x-x^{\delta_\ep,\eta}_i)+\sigma_{\delta_\ep}^\gamma(|x-x^{\delta_\ep,\eta}_{i}|) u_\ep^i(x)&\textrm{if }x\in A_{\gamma\delta_\ep,\delta_\ep}(x^{\delta_\ep,\eta}_i)\textrm{ for some }i\,
%\\
u_{\delta_\ep,s}^i(x)&\textrm{ if }x\in A_{\delta_\ep^s,\bar\rho}(x^{\delta_\ep,\eta}_{i})\,,
\\
(1-\sigma(|x-x^{\delta_\ep,\eta}_{i}|)z_{i}\theta(x-x^{\delta_\ep,\eta}_{i})+\sigma(|x-x^{\delta_\ep,\eta}_{i}|)\Theta(x)&\textrm { if }x\in A_{\bar\rho,2\bar\rho}(x^{\delta_\ep,\eta}_{i})\,,
 \end{array}\right.
 \end{equation*}
  and we define the function $w_{\ep,\eta,s}:\Omega_\ep(\mu_{\ep,\eta})\to\mathcal S^1$ as $w_{\ep,\eta,s}:=e^{\iota u_{\ep,\eta,s}}$ where
\begin{equation}\label{reccra2}
u_{\ep,\eta,s}(x)=\begin{cases}
u_{\ep,\eta,s}^i(x)&\textrm{ if }x\in A_{\ep,2\bar\rho}(x^{\delta_\ep,\eta}_{i})\textrm{ for some }i\\
\Theta(x)&\textrm{ elsewhere\,.}
\end{cases}
\end{equation}
Note that $w_{\ep,\eta,s}\in\AS_\ep(\mu_{\ep,\eta})$\,.

Let $i=1,\ldots,n$\,. By \eqref{simplf}, using the change of variable $x=\delta_\ep y+\delta_\ep\lfloor\frac{x_i}{\delta_\ep}\rfloor$ and the $1$-homogeneity of the function $a$, we have
\begin{eqnarray}\label{primoi}
\nonumber
\precrw(w_{\ep,\eta,s},A_{\ep,\delta_\ep^s}(x_i^{\delta_\ep,\eta}))&=&\int_{A_{\frac{\ep}{\delta_\ep},\delta_\ep^{s-1}}(y_\eta)}\frac{a(y)}{|y-y_\eta|^2}\ud y\\
&=&\int_{A_{\frac{\ep}{\delta_\ep},\delta_\ep^{s-1}}(y_\eta)\cap E_\eta}\frac{a(y)}{|y-y_\eta|^2}\ud y+\int_{A_{\frac{\ep}{\delta_\ep},\delta_\ep^{s-1}}(y_\eta)\setminus E_\eta}\frac{a(y)}{|y-y_\eta|^2}\ud y\\
\nonumber
&\le & 2\pi(m+\eta)\big(|\log\ep|-s|\log\delta_\ep|\big)+\beta \int_{A_{\frac{\ep}{\delta_\ep},\delta_\ep^{s-1}}(y_\eta)\setminus E_\eta}\frac{1}{|y-y_\eta|^2}\ud y\,.
\end{eqnarray}
%\begin{equation}\label{primoi}
%\begin{aligned}
%&\int_{A_{\ep,\delta_\ep^s}(x_i^{\delta_\ep,\eta})}a\Big(\frac{x}{\delta_\ep}\Big)|\nabla w_{\ep,\eta,s}|^2\ud x=\int_{A_{\frac{\ep}{\delta_\ep},\delta_\ep^{s-1}}(y_\eta)}a(y)\frac{1}{|y-y_\eta|^2}\ud y\\
%=&\int_{A_{\frac{\ep}{\delta_\ep},\delta_\ep^{s-1}}(y_\eta)\cap E_\eta}a(y)\frac{1}{|y-y_\eta|^2}\ud y+\int_{A_{\frac{\ep}{\delta_\ep},\delta_\ep^{s-1}}(y_\eta)\setminus E_\eta}a(y)\frac{1}{|y-y_\eta|^2}\ud y\\
%\le & 2\pi(m+\eta)\big(|\log\ep|-s|\log\delta_\ep|\big)+\beta \int_{A_{\frac{\ep}{\delta_\ep},\delta_\ep^{s-1}}(y_\eta)\setminus E_\eta}\frac{1}{|y-y_\eta|^2}\ud y\,.
%\end{aligned}
%\end{equation}

We now estimate the last integral in \eqref{primoi}. To this end, let $\gamma\in(0,1)$\,. We note that
\begin{equation}\label{primoi0}
%\begin{aligned}
\int_{A_{\gamma,\delta_\ep^{s-1}}(y_\eta)\setminus E_\eta}\frac{1}{|y-y_\eta|^2}\ud y\le 2\pi\log\frac{\delta_\ep^{s-1}}{\gamma}=2\pi(1-s)|\log\delta_\ep|+2\pi|\log\gamma|\,.
%\end{aligned}
\end{equation}
Let moreover $I:=\lceil \frac{|\log\ep|-|\log\delta_\ep|-|\log\gamma|}{\log 2}\rceil$\, and for every $i=0,1,\ldots,I$ we set $r_i:=2^i \frac{\ep}{\delta_\ep}$; then, using \eqref{den1}, we get
\begin{equation}\label{primoi1}
\begin{aligned}
 \int_{A_{\frac{\ep}{\delta_\ep},\gamma}(y_\eta)\setminus E_\eta}\frac{1}{|y-y_\eta|^2}\ud y
 \le& \sum_{i=1}^I \int_{A_{r_{i-1},r_i}(y_\eta)\setminus E_\eta}\frac{1}{|y-y_\eta|^2}\ud y\le \sum_{i=1}^I 4\frac{|B_{r_i}(y_\eta)\setminus E_\eta|}{r_i^2}\\
 \le&  \Big(\frac{|\log\ep|-|\log\delta_\ep|-|\log\gamma|}{\log 2}+1\Big)C_\eta(\gamma)\,,
 \end{aligned}
\end{equation}
where $\lim_{\gamma\to 0}C_\eta(\gamma)=0$ for every $\eta$\,.

By \eqref{primoi}, \eqref{primoi0} and \eqref{primoi1}, using \eqref{limlog0}, we deduce that
\begin{equation}\label{primoii}
\begin{aligned}
\limsup_{\ep\to 0}\frac{1}{|\log\ep|}\precrw(w_{\ep,\delta_\ep,s},A_{\ep,\delta_\ep^s}(x_i^{\delta_\ep,\eta})
\le 2\pi(1-s\lambda)(m+\eta)+2\pi\beta(1-s)\lambda+\beta\frac{1-\lambda}{\log 2} C_{\eta}(\gamma)\,.
\end{aligned}
\end{equation}
In addition, since $2\bar\rho \le 1$, by  \eqref{zeroi}, Lemma \ref{lm:glim0}, Proposition \ref{homdeg} (applied with $\ep=\delta_\ep$ and $s_2=s$) and \eqref{limlog0}, there exist moduli of continuity $\omega_1,\omega_2$ such that for every $i=1,\ldots,n$ we have
\begin{equation}\label{secondoi}
\begin{aligned}
&\frac{1}{|\log\ep|}\precrw(w_{\ep,\eta,s},A_{{\delta_\ep}^s,\bar\rho}(x^{\delta_\ep,\eta}_i))=\frac{|\log\delta_\ep|}{|\log\ep|}\frac{1}{|\log\delta_\ep|}\precrw(w_{\delta_\ep,s}^i,A_{{\delta_\ep}^s,\bar\rho}(x^{\delta_\ep,\eta}_i))\\
%&\le
%\frac{1}{|\log\ep|}\precrw(w_{\ep,s}; A_{\ep^{s},2\bar\rho}(x_{i}))
%\\
\le&(\lambda+\omega_1(\ep))\frac{1}{|\log\delta_\ep|}\inf_{w\in \Anew_{\delta_\ep^s,2\bar\rho}(x_{i})}\precrw(w; A_{\delta_\ep^{s},2\bar\rho}(x_{i}))+\omega_2(\ep)
\\
=&(\lambda+\omega_1(\ep))\Big(s-\frac{|\log(2\bar\rho)|}{|\log\delta_\ep|}\Big)2\pi\sqrt{\det\ho}+\omega_2({\ep})\,,
\end{aligned}
\end{equation}
where the last equality follows by \eqref{simplf} and \eqref{psisimplf}.
%\begin{equation}\label{secondoi}
%\begin{aligned}
%&\frac{1}{|\log\ep|}\int_{A_{{\delta_\ep}^s,\bar\rho}(x^{\delta_\ep,\eta}_i)}a\Big(\frac{x}{\delta_\ep}\Big)|\nabla w_{\ep,\eta,s}|^2\ud x=\frac{|\log\delta_\ep|}{|\log\ep|}\frac{1}{|\log\delta_\ep|}\int_{A_{{\delta_\ep}^s,\bar\rho}(x^{\delta_\ep,\eta}_i)}a\Big(\frac{x}{\delta_\ep}\Big)|\nabla w_{\ep,s}^i|^2\ud x\\
%%&\le
%%\frac{1}{|\log\ep|}\precrw(w_{\ep,s}; A_{\ep^{s},2\bar\rho}(x_{i}))
%%\\
%\le&(\lambda+\omega_1(\ep))\frac{1}{|\log\delta_\ep|}\inf_{w\in \Anew_{\delta_\ep^s,2\bar\rho}(x_{i})}\precrw(w; A_{\delta_\ep^{s},2\bar\rho}(x_{i}))+\omega_2(\ep)
%\\
%=&(\lambda+\omega_1(\ep))\Big(s-\frac{|\log(2\bar\rho)|}{|\log\delta_\ep|}\Big)2\pi\sqrt{\det\ho}+\omega_2({\ep})\,,
%\end{aligned}
%\end{equation}
%where $\omega_1(\ep),\omega_2(\ep)\to 0$ as $\ep\to 0$\,.
By \eqref{primoii}, \eqref{secondoi}, recalling \eqref{annrho} and \eqref{outann} we have that
\begin{equation*}
\begin{aligned}
\limsup_{\ep\to 0}\frac{1}{|\log\ep|}\precrw(w_{\ep,\eta,s},\Omega_\ep(\mu_{\ep,\eta}))\le &2\pi\Big((1-s\lambda)(m+\eta)+\lambda s  \sqrt{\det\ho}\Big)|\mu|(\Omega)\\
&+\Big(2\pi\beta(1-s)\lambda+ 2\pi(1-s)\lambda+\beta\frac{1-\lambda}{\log 2} C_{\eta}(\gamma)\Big)|\mu|(\Omega)\,,
\end{aligned}
\end{equation*}
%\begin{equation*}
%\begin{aligned}
%\limsup_{\ep\to 0}\frac{1}{|\log\ep|}\int_{\Omega_\ep(\mu)}a\Big(\frac{x}{\delta_\ep}\Big)|\nabla w_{\ep,\eta,s}|^2\ud x\le &2\pi\Big((1-s\lambda)(m+\eta)+\lambda  \sqrt{\det\ho}\Big)|\mu|(\Omega)\\
%&+\Big(2\pi\beta(1-s)\lambda+ 2\pi(1-s)\lambda+\beta\frac{1-\lambda}{\log 2} C_{\eta}(\gamma)\Big)|\mu|(\Omega)\,,
%\end{aligned}
%\end{equation*}
whence, suitably choosing $\gamma=\gamma_\ep\to 0$ as $\ep\to 0$\,, we get \eqref{lisu}.
Therefore, suitably choosing $s_\ep\to 1$ and $\eta_\ep\to 0$ as $\ep\to 0$, by \eqref{lisu} we get that
$\mu_\ep=\mu_{\ep,\eta_\ep}\fla\mu$ and $w_\ep=w_{\ep,\eta_\ep,s_\ep}$ satisfies
\begin{equation*}
\limsup_{\ep\to 0}\frac{1}{|\log\ep|}F_{\delta_\ep}(w_\ep;\Omega_\ep(\mu_\ep))\le\limsup_{\ep\to 0}\frac{\precr(\mu_\ep)}{|\log\ep|}\le 2\pi\bigg(\Big(1-\lambda)\mathrm{ess}\inf a+ \lambda\sqrt{\det\ho}\bigg)|\mu|(\Omega)\,.
\end{equation*}
\qed
%\end{proof}
%%%%%%%%%%%%%%%%%%%%%%%%%%%%
%%%%%%%%%%%%%%%%%%%%%%%%%%%%
%%%%%%%%%%%%%%%%%%%%%%%%%%%%
\subsection{The Ginzburg-Landau model}
%%%%%%%%%%%%%%%%%%%%%%%%%%%%
%%%%%%%%%%%%%%%%%%%%%%%%%%%%
%%%%%%%%%%%%%%%%%%%%%%%%%%%%
Finally, we prove Theorem \ref{gltbw} in the more general setting introduced in Subsection \ref{sec:mainthmgl}.
%%%%%%%%%%%%%%%%%%%%%%%%%%%%
%%%%%%%%%%%%%%%%%%%%%%%%%%%%
%%%%%%%%%%%%%%%%%%%%%%%%%%%%
\begin{theorem}\label{gldelta>ep}
Let $\pregl^W$ be defined in \eqref{GL0} where $a$ is a measurable $(0,1)^2$-periodic function satisfying $a(x)\in [\alpha,\beta]\subset(0,+\infty)$ for a.e. $x\in\R^2$.
Let moreover $\ho$ be the symmetric matrix defined in \eqref{defahom}.
If \eqref{limlog0} is satisfied,
then the following $\Gamma$-convergence result holds true.
\begin{itemize}
\item[(i)] (\,$\Gamma$-liminf inequality)
Let $\{v_\ep\}_\ep\subset H^1(\Omega;\R^2)$ be such that $J v_\ep\fla\pi\mu$ for some $\mu\in X(\Omega)$\,. Then
\begin{equation*}%\label{liminfgl}
\liminf_{\ep\to 0}\frac{\pregl^W(v_\ep)}{|\log\ep|}\ge 2\pi\bigg(\Big(1-\lambda)\mathrm{ess}\inf a+ \lambda\sqrt{\det\ho}\bigg)|\mu|(\Omega).
\end{equation*}
\item[(ii)](\,$\Gamma$-limsup inequality) For every $\mu\in X(\Omega)$, there exists a sequence $\{v_\ep\}_\ep\subset H^1(\Omega;\R^2)$ such that  $Jv_\ep\fla\pi\mu$ and
\begin{equation}\label{lsgl}
\limsup_{\ep\to 0}\frac{\pregl^W(v_\ep)}{|\log\ep|}\le2\pi\bigg(\Big(1-\lambda)\mathrm{ess}\inf a+ \lambda\sqrt{\det\ho}\bigg)|\mu|(\Omega)\,.
\end{equation}
\end{itemize}
\end{theorem}
%%%%%%%%%%%%%%%%%%%%%%%%%%%%%
%%%%%%%%%%%%%%%%%%%%%%%%%%%%%
%%%%%%%%%%%%%%%%%%%%%%%%%%%%%
%\begin{proof}
%Since $a\le\beta$ a.e., the compactness property (i) is a corollary of classical results in the variational analysis of the classical GL functional (see for instance \cite[Theorem 4.1]{AP}).
%\vskip8pt
%%%%%%%%%%%%%%%%%%%%%%%%%%%%%
%%%%%%%%%%%%%%%%%%%%%%%%%%%%%
%%%%%%%%%%%%%%%%%%%%%%%%%%%%%

\noindent{\it Proof of {\rm(i)}.} Without loss of generality we can assume that $\pregl^W(w_\ep)\le C|\log\ep|$ and by standard density arguments we can also assume that $\{v_\ep\}_\ep\subset H^1(\Omega;\R^2)\cap C^1(\Omega;\R^2)$\,.

Let $\bar\tau\in (0,1)$\,. By arguing verbatim as in the proof of Theorem \ref{mainthmgldopo}(ii), one can prove that for every $\ep>0$\,, there exist $\gamma^{\bar\tau}>0$ with $\gamma^{\bar\tau}\to 0$ as $\bar\tau\to 0$\,, a family $\B_\ep$ of balls such that $\rad(\B_\ep)\le C_{\bar\tau}\ep|\log\ep|$ and
 \begin{equation}\label{alsolater2}
 1-\gamma^{\bar\tau}<|v_\ep|<1+\gamma^{\bar\tau}\qquad\textrm{in }\Omega(\B_\ep)\,,
 \end{equation}
and a measure $\mu_\ep$ with $\supp\mu_\ep\subset\bigcup_{B\in\B_\ep}B$ such that
$\mu_\ep\fla\mu$ as $\ep\to 0$\,.
For every $\ep>0$ we defined the function $w_\ep\in H^1(\Omega(\B_\ep); \mathcal S^1)$ as $w_\ep(x):=\frac{v_\ep(x)}{|v_\ep(x)|}$\,.
By \eqref{alsolater2} we get
\begin{equation}\label{ste62}
\begin{aligned}
\pregl^W(v_\ep)\ge&\int_{\Omega(\B_\ep)}a\Big(\frac{x}{\delta_\ep}\Big)|\nabla v_\ep|^2\ud x\ge(1-\gamma^{\bar\tau})\int_{\Omega(\B_\ep)}a\Big(\frac{x}{\delta_\ep}\Big)|\nabla w_\ep|^2\\
\ge&(1-\gamma^{\bar\tau})\precr(\mu_\ep,\Omega(\B_\ep))\,,
\end{aligned}
\end{equation}
where $\precr$ is defined in \eqref{cra_new} with $F_\delta$ defined in \eqref{aform}.
Since the assumptions of Proposition \ref{mainthmcranewdelta>ep} are satisfied, by \eqref{ste62} we have that
$$
\liminf_{\ep\to 0}\frac{\pregl^W(v_\ep)}{|\log\ep|}\ge (1-\gamma^{\bar\tau}) 2\pi\bigg(\Big(1-\lambda)\mathrm{ess}\inf a+ \lambda\sqrt{\det\ho}\bigg)|\mu|(\Omega)\,,
$$
whence the claim follows letting $\bar\tau\to 0$\,.
\vskip8pt
%%%%%%%%%%%%%%%%%%%%%%%%%%%%%
%%%%%%%%%%%%%%%%%%%%%%%%%%%%%
%%%%%%%%%%%%%%%%%%%%%%%%%%%%%

{\it Proof of {\rm(ii)}.}
By arguing as in  the proof of Theorem \ref{cradelta>ep} (iii), we may reduce to the case $|z_i|=1$ for every $i=1,\ldots,n$\,.
For every $0<\eta,s<1$\,, let $\mu_{\ep,\eta}$  be defined as in \eqref{recose} and let
% $w_{\ep,\eta,s}:\Omega_{\ep}(\mu_{\ep,\eta})\to\mathcal S^1$ be defined by $w_{\ep,\eta,s}:=e^{\iota u_{\ep,\eta,s}}$ with
$u_{\ep,\eta,s}$ be the function provided by \eqref{reccra2}\,.
Setting $\Omega_\ep(\mu_{\ep,\eta}):=\Omega\setminus\bigcup_{i=1}^n B_{\ep}(x_{i}^{\delta_\ep,\eta})$ with $x_{i}^{\delta_\ep,\eta}$ defined in \eqref{newsing}\,,
for every $\ep>0$ we define
 \begin{equation*}%\label{recgl}
v_{\ep,\eta,s}(x):=\begin{cases}
e^{\iota u_{\ep,\eta,s}(x)}&\textrm{ if }x\in\Omega_{\ep}(\mu_{\ep,\eta})\,,\\
\frac{|x-x_{i}|}{\ep}\Big(\frac{x-x_{i}}{|x-x_{i}|}\Big)^{z_{i}}&\textrm{ if }x\in B_{\ep}(x^{\delta_\ep,\eta}_{i})\textrm{ for some }i\,.
\end{cases}
\end{equation*}
We note that $v_{\ep,\eta,s}\in H^1(\Omega;\R^2)$ and that $Jv_{\ep,\eta,s}=\pi\mu_{\ep,\eta}$ for every $\ep>0$. In addition, for almost every $x\in\Om_\ep(\mu_{\ep,\eta})$ we have that $|v_{\ep,\eta,s}(x)|=1$, hence $W(|v_{\ep,\eta,s}(x)|)=0$.
By arguing as in \eqref{estcoresvep} and \eqref{estcoresgrad}, and using \eqref{lisu}, we thus obtain that
\begin{equation*}
\begin{aligned}
\limsup_{\ep\to 0}\frac{1}{|\log\ep|}\pregl^W(v_{\ep,\eta,s})\le &2\pi\Big((1-s\lambda)(m+\eta)+\lambda s \sqrt{\det\ho}\Big)|\mu|(\Omega)\,\\
&+\Big(2\pi\beta(1-s)\lambda+ 2\pi(1-s)\lambda\Big)|\mu|(\Omega)\,.
\end{aligned}
\end{equation*}
Therefore, suitably choosing $s_\ep\to 1$ and $\eta_\ep\to 0$ as $\ep\to 0$ and setting $v_\ep=v_{\ep,\eta_\ep,s_\ep}$\,, by \eqref{flareco} we have that $Jv_\ep\fla\mu$ as $\ep\to 0$ and that $\{v_\ep\}_\ep$ satisfies \eqref{lsgl}.
\qed
%\end{proof}
%%%%%%%%%%%%%%%%%%%%%%%%%%%%%
%%%%%%%%%%%%%%%%%%%%%%%%%%%%%
%%%%%%%%%%%%%%%%%%%%%%%%%%%%%
\begin{remark}\label{example}
\rm{
Note that if $\delta_\ep$ tends to zero much slower than $\ep$ in such a way that $\lambda=0$ in \eqref{limlog0}\,, then, within the $|\log\ep|$, scaling the homogenization process is not detected by
the $\Gamma$-limit in Theorems \ref{cradelta>ep} and \ref{gldelta>ep}, which in turns reduces to $2\pi\, \mathrm{ess}\inf a|\mu|(\Omega)$\,.
This is the case if $\lim_{\ep\to 0}\frac{\ep^p}{\delta_\ep}=0$ for all $p\in (0,1]$\,; for example, if $\delta_\e={1\over|\log\e|}$.
}
\end{remark}

\begin{example}\label{r-example}
\rm{
\red{We can give an explicit example, choosing $a$ piecewise constant on a checkerboard taking alternatively the values $\alpha$ and $\beta$. We have that $\ho=\sqrt{\alpha\beta}\,{\rm I}$ (see e.g.~\cite{JKO} Section 1.5), so that the corresponding $\Gamma$-limit is
$$2\pi\Big((1-\lambda)\alpha+\lambda\sqrt{\alpha\beta}\Big)|\mu|(\Omega),
$$
with $\lambda$ given by \eqref{limlog0}.
%$$\lambda:=\lim_{\ep\to 0}\frac{|\log\delta_\ep|}{|\log\ep|}\wedge 1.
%$$
The limit has the same form if we choose $a$ as a laminate taking only the values $\alpha$ and $\beta$ with volume fraction $1/2$, whose homogenized matrix  $\ho$ has the eigenvalues ${\alpha+\beta\over 2}$ and $2\alpha\beta\over \alpha+\beta$ (see e.g.~\cite[ Section 12.2.2]{B02}).
}}
\end{example}
%%%%%%%%%%%%%%%%%%%%%%%%%%%%%
%%%%%%%%%%%%%%%%%%%%%%%%%%%%%
%%%%%%%%%%%%%%%%%%%%%%%%%%%%%

\end{document}